\documentclass[a4paper,12pt]{article}
\usepackage{amsmath}
\usepackage{amssymb}
\usepackage{amsthm}
\usepackage[english]{babel}
\usepackage{hyperref}
\usepackage[all]{xy}
\usepackage{tikz}
\usepackage{mathdots}
\usetikzlibrary{decorations.pathreplacing}

\usepackage{paralist}
\usepackage{enumerate}
\setdefaultleftmargin{15pt}{}{}{}{}{}

\newtheorem{theor}{Theorem}[section]
\newtheorem{lemma}[theor]{Lemma}
\newtheorem{cor}[theor]{Corollary}
\newtheorem{rem}[theor]{Remark}

\newcommand{\hd}{\mathrm{hd}}
\newcommand{\soc}{\mathrm{soc}}
\newcommand{\Hom}{\mathrm{Hom}}
\newcommand{\End}{\mathrm{End}}

\newcommand{\GL}{\mathrm{GL}}
\newcommand{\s}{\Sigma}

\newcommand{\md}{\mbox{-}\mathrm{mod}}
\newcommand{\Md}{\!\mod}

\renewcommand{\Im}{\mathrm{Im}}
\renewcommand{\epsilon}{\varepsilon}
\renewcommand{\phi}{\varphi}
\newcommand{\xymat}{\xymatrix@R=6pt@C=10pt}

\begin{document}

\begin{center}
{\Large 
Irreducible tensor products for symmetric groups in characteristic 2}

\vspace{12pt}

Lucia Morotti

{\small
Institut f\"{u}r Algebra, Zahlentheorie und Diskrete Mathematik

Leibniz Universit\"{a}t Hannover

Welfengarten 1

30167 Hannover

Germany

\tt morotti@math.uni-hannover.de
}
\end{center}

\begin{abstract}
We consider non-trivial irreducible tensor products of modular representations of a symmetric group $\s_n$ in characteristic 2 for even $n$ completing the proof of a classification conjecture of Gow and Kleshchev about such products.
\end{abstract}


\section{Introduction}

Let $D_1$ and $D_2$ be irreducible representations of $\s_n$ of dimension greater than 1. We would like to know when the tensor product $D_1\!\otimes\! D_2$ is irreducible. We say that $D_1\!\otimes\! D_2$ is a non-trivial irreducible tensor product if $D_1\!\otimes\! D_2$ is irreducible and neither $D_1$ nor $D_2$ has dimension 1. In \cite{z1} Zisser proved that there are no non-trivial irreducible tensor products of ordinary representations of symmetric group. In \cite{gk} Gow and Kleshchev conjectured that the same holds also for modular representations, unless $p=2$ and $n=2m$ with $m$ odd. In this case they also conjectured which tensor products are irreducible. In \cite{gj} Graham and James proved that the tensor products appearing in the conjecture of Gow and Kleshchev are irreducible. Further in \cite{ck} Bessenrodt and Kleshchev proved that non-trivial irreducible tensor products are only possible when $p=2$, $n$ is even and one of the modules is indexed by a JS-partition.

In this paper we will prove the following two theorems which will prove the conjecture from \cite{gk} for $p=2$ and $n=2m$ even in the cases $m$ odd and $m$ even respectively.

\begin{theor}\label{t3}
If $p=2$ and $n=2m$ with $m$ odd then the only irreducible tensor products of 2 representations of $\s_n$ of dimension greater than 1 are those of the form
\[D^{(m+1,m-1)}\otimes D^{(2m-2j-1,2j+1)}\cong D^{(m-j,m-j-1,j+1,j)}\]
with $0\leq j<(m-1)/2$.
\end{theor}

\begin{theor}\label{t4}
Let $p=2$ and $n=2m$ with $m$ even. If $D_1$ and $D_2$ are irreducible representations of $\s_n$ of dimension greater than 1 then $D_1\otimes D_2$ is not irreducible.
\end{theor}

Together with the results from \cite{ck} and \cite{gj}, Theorems \ref{t3} and \ref{t4} prove the conjecture which Gow and Kleshchev made in \cite{gk}.

The classification of non-trivial irreducible tensor products is relevant to the description of maximal subgroups in finite groups of Lie type, see \cite{a} and \cite{as}. 
For alternating groups, non-trivial irreducible tensor products have been classified in most characteristics in \cite{bk2} and \cite{bk3}. Differently than for symmetric groups, there exist non-trivial irreducible tensor products in arbitrary characteristic. For covering groups of symmetric and alternating groups a partial classification of non-trivial irreducible tensor products can be found in \cite{b2}, \cite{bk4} and \cite{kt}. When considering groups of Lie type in defining characteristic, non-trivial irreducible tensor products are not unusual, due to Steinberg tensor product theorem. In not defining characteristic however it has been proved that in almost all cases non non-trivial irreducible tensor products exist, see \cite{kt2} and \cite{mt}.

In Sections \ref{s6} to \ref{s4} we will prove preliminary lemmas on the structure of the endomorphism rings of restrictions of the modules $D^\lambda$ and on the structure of certain permutation modules. Some of these results are of independent interest. Using these lemmas we will then prove Theorems \ref{t3} and \ref{t4} in Section \ref{s1}.

\section{Notations and basic results}

Let $F$ be an algebraically closed field  of characteristic $p$. In most of the paper we will assume that $p=2$. Some of the results in Sections \ref{s6} and \ref{s7} however hold for arbitrary primes.

For a partition $\lambda\vdash n$ let $S^\lambda$ be the corresponding Specht module. If $\lambda$ is a $p$-regular partition (that is a partition where no part is repeated $p$ or more times) we define $D^\lambda$ to be the irreducible $F\s_n$-module indexed by $\lambda$. Further for a composition $\alpha=(\alpha_1,\alpha_2,\ldots)\vdash n$ let $\s_\alpha\cong \s_{\alpha_1}\times \s_{\alpha_2}\times \ldots$ be the Young subgroup corresponding to $\alpha$ and define $M^\alpha:=1\uparrow_{\s_\alpha}^{\s_n}$ to be the permutation module induced from $\s_\alpha$. The modules $D^\lambda$ and $M^\alpha$ are known to be self-dual. From their definition we have that $D^{(n)}\cong S^{(n)}\cong M^{(n)}\cong 1_{\s_n}$. For more informations on such modules see \cite{j1} and \cite{jk}. For any partition $\lambda$ let $h(\lambda)$ be the number of parts of $\lambda$.

Let $M$ be a $F\s_n$-module corresponding to a unique block $B$ with content $(b_0,\ldots,b_{p-1})$ (see \cite{k1}). For $0\leq i\leq p-1$, we can define $e_iM$ as the restriction of $M\downarrow_{\s_{n-1}}$ to the block with content $(b_0,\ldots,b_{i-1},b_i-1,b_{i+1},\ldots,b_{p-1})$. Similarly, for $0\leq i\leq p-1$, we can define $f_iM$ as the restriction of $M\uparrow^{\s_{n+1}}$ to the block with content $(b_0,\ldots,b_{i-1},b_i+1,b_{i+1},\ldots,b_{p-1})$. We can then extend the definition of $e_iM$ and $f_iM$ to arbitrary $F\s_n$-modules additively. The following result holds by Theorems 11.2.7 and 11.2.8 of \cite{k1}.

\begin{lemma}\label{l45}
For $M$ a $F\s_n$-module we have that
\[M\downarrow_{\s_{n-1}}\cong e_0M\oplus\ldots\oplus e_{p-1}M\hspace{24pt}\mbox{and}\hspace{24pt}M\uparrow^{\s_{n+1}}\cong f_0M\oplus\ldots\oplus f_{p-1}M.\]
\end{lemma}

For $r\geq 1$ let $e_i^{(r)}:F\s_n\md\rightarrow F\s_{n-r}\md$ and $f_i^{(r)}:F\s_n\md\rightarrow F\s_{n+r}\md$ denote the divided power functors (see Section 11.2 of \cite{k1} for the definitions). For $r=0$ define $e_i^{(0)}D^\lambda$ and $f_i^{(0)}D^\lambda$ to be equal to $D^\lambda$. The modules $e_i^rD^\lambda$ and $e_i^{(r)}D^\lambda$ (and similarly $f_i^rD^\lambda$ and $f_i^{(r)}D^\lambda$) are quite closely connected as we will see in the next two lemmas. The following notation will be used in the lemmas. For a partition $\lambda$ and $0\leq i\leq 1$ let $\epsilon_i(\lambda)$ be the number of normal nodes of $\lambda$ of residue $i$ and $\phi_i(\lambda)$ be the number of conormal nodes of $\lambda$ of residue $i$ (see Section 11.1 of \cite{k1} or Section 2 of \cite{bk2} for two different but equivalent definitions of normal and conormal nodes). Normal and conormal nodes of partitions will play a crucial role throughout the paper.

\begin{lemma}\label{l39}
Let $\lambda\vdash n$ be a $p$-regular partition. Also let $0\leq i\leq p-1$ and $r\geq 0$. Then $e_i^rD^\lambda\cong(e_i^{(r)}D^\lambda)^{\oplus r!}$. Further $e_i^{(r)}D^\lambda\not=0$ if and only if $\epsilon_i(\lambda)\geq r$. In this case, if $\nu$ is obtained by $\lambda$ by removing the $r$ bottom $i$-normal nodes, then
\begin{enumerate}
\item\label{l39a}
$e_i^{(r)}D^\lambda$ is a self-dual indecomposable module with head and socle isomorphic to $D^\nu$,

\item\label{l39b}
$[e_i^{(r)}D^\lambda:D^\nu]=\binom{\epsilon_i(\lambda)}{r}=\dim\End_{\s_{n-1}}(e_i^{(r)}D^\lambda)$,

\item\label{l39c}
if $D^\psi$ is a composition factor of $e_i^{(r)}D^\lambda$ then $\epsilon_i(\psi)\leq \epsilon_i(\lambda)-r$, with equality holding if and only if $\psi=\nu$.
\end{enumerate}
\end{lemma}

\begin{lemma}\label{l40}
Let $\lambda\vdash n$ be a $p$-regular partition. Also let $0\leq i\leq p-1$ and $r\geq 0$. Then $f_i^rD^\lambda\cong(f_i^{(r)}D^\lambda)^{\oplus r!}$. Further $f_i^{(r)}D^\lambda\not=0$ if and only if $\phi_i(\lambda)\geq r$. In this case, if $\pi$ is obtained by $\lambda$ by adding the $r$ top $i$-conormal nodes, then
\begin{enumerate}
\item\label{l40a}
$f_i^{(r)}D^\lambda$ is a self-dual indecomposable module with head and socle isomorphic to $D^\pi$,

\item\label{l40b}
$[f_i^{(r)}D^\lambda:D^\pi]=\binom{\phi_i(\lambda)}{r}=\dim\End_{\s_{n+1}}(f_i^{(r)}D^\lambda)$,

\item\label{l40c}
if $D^\psi$ is a composition factor of $f_i^{(r)}D^\lambda$ then $\phi_i(\psi)\leq \phi_i(\lambda)-r$, with equality holding if and only if $\psi=\pi$.
\end{enumerate}
\end{lemma}

For proofs see Theorems 11.2.10 and 11.2.11 of \cite{k1} (the case $r=0$ holds trivially).
In particular, for $r=1$, we have that $e_i=e_i^{(1)}$ and $f_i=f_i^{(1)}$. In this case there are other compositions factors of $e_iD^\lambda$ and $f_iD^\lambda$ which are known (see Remark 11.2.9 of \cite{k1}).

\begin{lemma}\label{l56}
Let $\lambda$ be a $p$-regular partition. If $\alpha$ is $p$-regular and is obtained from $\lambda$ by removing a normal node of residue $i$ then $D^\alpha$ is a composition factor of $e_iD^\lambda$.

Similarly if $\beta$ is $p$-regular and is obtained from $\lambda$ by adding a conormal node of residue $i$ then $D^\beta$ is a composition factor of $f_iD^\lambda$.
\end{lemma}

The following properties of $e_i$ and $f_i$ are just a special cases of Lemma 8.2.2(ii) and Theorem 8.3.2(i) of \cite{k1}.

\begin{lemma}\label{l57}
If $M$ is self dual then so are $e_iM$ and $f_i(M)$.
\end{lemma}

\begin{lemma}\label{l48}
The functors $e_i$ and $f_i$ are left and right adjoint of each others.
\end{lemma}

Using notations from Lemma \ref{l39} define $\tilde{e}^r_iD^\lambda:=D^\nu$ if $\epsilon_i(\lambda)\geq r$, otherwise define $\tilde{e}^r_iD^\lambda:=0$. Similarly, with notations from Lemma \ref{l40}, define $\tilde{f}^r_iD^\lambda:=D^\pi$ if $\phi_i(\lambda)\geq r$, otherwise define $\tilde{f}_i^rD^\lambda:=0$. The first part of the next lemma follows from Lemma 5.2.3 of \cite{k1}. The second part follows by the definition of $\tilde{e}^r_i$ and $\tilde{f}^r_i$ and from Lemmas \hyperref[l39c]{\ref*{l39}\ref*{l39c}} and \hyperref[l40c]{\ref*{l40}\ref*{l40c}}.

\begin{lemma}\label{l47}
For $r\geq 0$ and $p$-regular partitions $\lambda,\nu$ we have that $\tilde{e}^r_i(D^\lambda)=D^\nu$ if and only if $D^\lambda=\tilde{f}^r_i(D^\nu)$. Further in this case $\epsilon_i(\nu)=\epsilon_i(\lambda)-r$ and $\phi_i(\nu)=\phi_i(\lambda)+r$.
\end{lemma}

In particular, if $\epsilon_i(\lambda)\geq r$ and notation is as above, then $\phi_i(\nu)\geq r$ and the $r$ bottom normal $i$-nodes of $\lambda$ are the $r$ top conormal $i$-nodes of $\nu$. Similarly, if $\phi_i(\lambda)\geq r$, then $\epsilon_i(\pi)\geq r$ and the $r$ top conormal $i$-nodes of $\lambda$ are the $r$ bottom normal $i$-nodes of $\pi$.

When considering the number of normal and conormal nodes of a partition we have the following result (which for $p$-regular partitions follows from Lemmas \ref{l45}, \ref{l39}, \ref{l40} and Corollary 4.2 of \cite{k3}):

\begin{lemma}\label{l52}
Any partition has 1 more conormal node than it has normal nodes.
\end{lemma}

\begin{proof}
For each residue $i$ the reduced $i$-signature is obtained from the $i$-signature by recursively removing pairs corresponding to an addable and a removable node. The lemma then follows from the definition of normal and conormal nodes, as any partition has 1 more addable node than it has removable nodes.
\end{proof}

From Lemmas \ref{l45}, \ref{l39} and \ref{l40} and as the modules $e_iD^\lambda$ (or the modules $f_iD^\lambda$) correspond to pairwise distinct blocks, we have the following result.

\begin{lemma}\label{l53}
For a $p$-regular partition $\lambda\vdash n$ we have that
\[\dim\End_{\s_{n-1}}(D^\lambda\downarrow_{\s_{n-1}})=\epsilon_0(\lambda)+\ldots+\epsilon_{p-1}(\lambda)\]
and
\[\dim\End_{\s_{n+1}}(D^\lambda\uparrow^{\s_{n+1}})=\phi_0(\lambda)+\ldots+\phi_{p-1}(\lambda).\]
\end{lemma}

A $p$-regular partition $\lambda\vdash n$ for which $D^\lambda\downarrow_{\s_{n-1}}$ is irreducible is called a JS-partition. When $p=2$, such partitions are easily classified, as can be seen in the next lemma, which follows from Theorem D of \cite{k2}, Lemma \ref{l53} and from $D^\lambda\downarrow_{\s_{n-1}}$ being self-dual.

\begin{lemma}\label{l55}
For $p=2$ and $\lambda$ a 2-regular partition the following are equivalent:
\begin{enumerate}
\item
$\lambda$ is a JS-partition,

\item
the parts of $\lambda$ are all congruent modulo 2,

\item
$\lambda$ has only one normal node.
\end{enumerate}
\end{lemma}

As we will also consider restrictions of $\s_n$-modules to $\s_{(n-2,2)}$ we define some notation for the restriction of modules to certain blocks of $\s_{(n-2,2)}$. For $M$ an $\s_n\md$ corresponding to the block with content $(b_0,\ldots,b_{p-1})$ we define $\overline{e}_i^2M$ as the restriction of $M\downarrow_{\s_{n-2,2}}$ to the block with content $(b_0,\ldots,b_{i-1},b_i-2,b_{i+1},\ldots,b_{p-1})$ for the $\s_{n-2}$ factor of $\s_{n-2,2}$. Notice that from the definition $(\overline{e}_i^2M)\downarrow_{\s_{n-2}}=e_i^2M$.

For arbitrary modules $M_1,\ldots,M_h$ we will write $M\sim M_1|\ldots|M_h$ if $M$ has a filtration with factors $M_1,\ldots,M_h$ counted from the bottom. For irreducible modules $D_1,\ldots,D_h$ we will write $M=D_1|\dots|D_h$ if $M$ is a uniserial module with composition factors $D_1,\ldots,D_h$ counted from the bottom. For irreducible modules $D_1,D_2,D_3$ we will also write $M=(D_1\oplus D_2)|D_3$ for a module $M$ with socle $D_1\oplus D_2$ and head $D_3$ and no other composition factor. Similarly we will write $M=D_1|(D_2\oplus D_3)$ for a module with socle $D_1$ and head $D_2\oplus D_3$ and no other composition factor.

Let $M\sim D_1|\ldots|D_h$ with $D_i$ irreducible. When writing
\[\xymatrix@R=1pt@C=10pt{
&&&\\
&&&\\
&&D_k\ar@{-}[ddd]\ar@{.}[dl]\ar@{.}[uul]\ar@{.}[uu]\ar@{.}[uur]\ar@{.}[dr]\\
&&&\\
M=&&&&,\\
&&&\\
&&D_j\ar@{.}[ddl]\ar@{.}[ul]\ar@{.}[dd]\ar@{.}[ur]\ar@{.}[ddr]\\
&&&\\
&&&
}\]
edges will correspond to uniserial subquotients. For example uniserial modules can be written as
\[\xymatrix@R=3pt@C=10pt{
&D_h\ar@{-}[dd]\\
\\
M=D_1|D_2|\ldots|D_h=&\vdots\ar@{-}[d]\\
&D_2\ar@{-}[dd]\\
\\
&D_1
}\]
and modules of the form $(D_1\oplus D_2)|D_3$ or $D_1|(D_2\oplus D_3)$ can be written as
\[\xymatrix@R=1pt@C=0pt{
&&D_3\ar@{-}[ddr]\ar@{-}[ddl]&&&&D_2\ar@{-}[ddr]&&D_3\ar@{-}[ddl]\\
M=(D_1\oplus D_2)|D_3=&&&&\mbox{ or }&M=D_1|(D_2\oplus D_3)=&&&&\!.\\
&D_1&&D_2&&&&D_1
}\]
Information on socle, head and direct summands can be obtained from the diagrams, for example if
\[\xymat{
&D_4\ar@{-}[dd]&D_5\ar@{-}[d]&D_6\ar@{-}[dl]&&D_4\ar@{-}[dd]&&D_5\ar@{-}[d]&D_6\ar@{-}[dl]\\
M=&&D_3\ar@{-}[d]&&=&&\oplus&D_3\ar@{-}[d]\\
&D_1&D_2&&&D_1&&D_2
}\]
then $\soc(M)=D_1\oplus D_2$ and $\hd(M)=D_4\oplus D_5\oplus D_6$.

\section{Modules structure}\label{s6}

Before being able to prove Theorems \ref{t3} and \ref{t4} we need some lemmas. We start by showing that $f_i^{(a)}D^\pi\cong e_i^{(\phi_i(\pi)-a)}\tilde{f}_i^{\phi_i(\pi)}D^\pi$ if $\epsilon_i(\pi)=0$.

\begin{lemma}\label{l34}
Let $M$, $A$ and $B$ be $G$-modules with $M\sim A|B$. Then for any $G$-module $N$ we have that
\[\dim\Hom_G(M,N)\leq\dim\Hom_G(A,N)+\dim\Hom_G(B,N).\]
\end{lemma}

\begin{proof}
This follows from $\Hom_G(\,\cdot\,,N)$ being left exact.
\end{proof}

\begin{lemma}\label{ll1}\label{t1}
Let $A$ and $B$ be finite dimensional modules with $\soc(B)\cong C$ simple. Then $\dim\Hom_G(A,B)\leq[A:C]$.

If $\soc(A)\cong C$ and $\dim\Hom_G(A,B)=[A:C]$ then $A$ is isomorphic to a submodule of $B$.
\end{lemma}

\begin{proof}
Let $A\sim D_1|\ldots|D_h$ with $D_i$ simple. Then, as $C\cong\soc(B)$ is simple, we have from Lemma \ref{l34}
\begin{align*}
\dim\Hom_G(A,B)&\leq\dim\Hom_G(D_1,B)+\ldots+\dim\Hom_G(D_h,B)\\
&=\dim\Hom_G(D_1,C)+\ldots+\dim\Hom_G(D_h,C)\\
&=[A:C].
\end{align*}

Assume now that $\soc(A)\cong C$ and $\dim\Hom_G(A,B)=[A:C]$. Then $A\sim C|\overline{A}$ for a certain module $\overline{A}$. From the previous part and by assumption
\[\dim\Hom_G(A,B)=[A:C]>[\overline{A}:C]\geq \dim\Hom_G(\overline{A},B).\]
In particular there exists $f\in\Hom_G(A,B)$ with $C\cong\soc(A)\not\subseteq\ker(f)$. As the socle of $A$ is simple we have that $\soc(A)\cap\ker(f)=0$ and so $\ker(f)=0$, from which the lemma follows.
\end{proof}

\begin{lemma}\label{l9'}
If $\pi$ is a $p$-regular partition with $\epsilon_i(\pi)=0$ and we let $\nu$ with $D^\nu=\tilde{f}_i^aD^\pi$ for some $0\leq a\leq\phi_i(\pi)$, then for $0\leq b\leq a$ we have that $e_i^{(b)}D^\nu\subseteq f_i^{(a-b)}D^\pi$.

Similarly if $\psi$ is a $p$-regular partition with $\phi_i(\psi)=0$ and we let $\rho$ with $D^\rho=\tilde{e}_i^cD^\rho$ for some $0\leq c\leq\epsilon_i(\psi)$, then for $0\leq d\leq c$ we have that $f_i^{(d)}D^\rho\subseteq e_i^{(c-d)}D^\psi$.
\end{lemma}

\begin{proof}
We will prove only the first part of the lemma, as the second part can be proved similarly.

First notice that $\epsilon_i(\nu)=a$ from Lemma \ref{l47}. Also, if $\gamma$ is such that $D^\gamma=\tilde{f}_i^{a-b}D^\pi$, then, again by Lemma \ref{l47},
\[\tilde{e}_i^{a-b}\tilde{e}_i^bD^\nu=\tilde{e}_i^aD^\nu=D^\pi=\tilde{e}_i^{a-b}D^\gamma\]
and then $\tilde{e}_i^bD^\nu=D^\gamma$. So $\soc(f_i^{(a-b)}D^\pi)\cong D^\gamma\cong\soc(e_i^{(b)}D^\nu)$ from Lemmas \ref{l39} and \ref{l40}. We also have that $[e_i^{(b)}D^\nu:D^\gamma]=\binom{a}{b}$.

Further, if $\pi\vdash n$, from Lemmas \ref{l39}, \ref{l40} and \ref{l48}
\begin{align*}
&\dim\Hom_{\s_{n+a-b}}(e_i^{(b)}D^\nu,f_i^{(a-b)}D^\pi)\\
&\hspace{12pt}=\frac{1}{b!(a-b)!}\dim\Hom_{\s_{n+a-b}}(e_i^bD^\nu,f_i^{a-b}D^\pi)\\
&\hspace{12pt}=\frac{1}{b!(a-b)!}\dim\Hom_{\s_n}(e_i^aD^\nu,D^\pi)\\
&\hspace{12pt}=\binom{a}{b}.
\end{align*}
The result now follows from Lemma \ref{t1}.
\end{proof}

\begin{lemma}\label{l9}
Let $\pi$ be a $p$-regular partition with $\epsilon_i(\pi)=0$ and let $\psi$ with $D^\psi=\tilde{f}_i^{\phi_i(\pi)}D^\pi$. For $0\leq a\leq \phi_i(\pi)$ we have $f_i^{(a)}D^\pi\cong e_i^{(\phi_i(\pi)-a)}D^\psi$.
\end{lemma}

\begin{proof}
As $\epsilon_i(\psi)=\phi_i(\pi)$ and $\phi_i(\psi)=0$ from Lemma \ref{l47}, we have from Lemma \ref{l9'} that, up to isomorphism,
\[e_i^{(\phi_i(\pi)-a)}D^\psi\subseteq f_i^{(a)}D^\pi\subseteq e_i^{(\phi_i(\pi)-a)}D^\psi\]
and so the lemma holds.
\end{proof}

We will now prove some lemmas about the structure of certain modules of $\s_n$ and $\s_{n-1}$.

\begin{lemma}\label{l3}
If $n\geq 4$ and $p\!\not|\,n-1$ then
\[D^{(n-2,1)}\uparrow^{\s_n}\sim S^{(n-2,1,1)}|S^{(n-2,2)}|S^{(n-1,1)}.\]
\end{lemma}

\begin{proof}
As $p\!\not\hspace{-1pt}|\hspace{3pt} n-1$ we have that $D^{(n-1)}$ and $D^{(n-2,1)}$ are in distinct blocks. So $D^{(n-2,1)}\cong S^{(n-2,1)}$ from Corollary 12.2 of \cite{j1}. The lemma then follows from Corollary 17.14 of \cite{j1}.
\end{proof}

\begin{lemma}\label{l1}
If $p=2$ and $n\geq 4$ is even then $S^{(n-1,1)}=D^{(n)}|D^{(n-1,1)}$ and $M^{(n-1,1)}=D^{(n)}|D^{(n-1,1)}|D^{(n)}\sim S^{(n-1,1)}|S^{(n)}$.
\end{lemma}

\begin{proof}
The structure of $S^{(n-1,1)}$ follows from Corollary 12.2 and Theorem 24.15 of \cite{j1}. Since $M^{(n-1,1)}$ is self-dual, the lemma then follows from Example 17.17 of \cite{j1}.
\end{proof}

\begin{lemma}\label{l2}
Let $p=2$ and $n\geq 6$ be even.

\begin{itemize}
\item
If $n\equiv 0\Md 4$ then $S^{(n-2,2)}=D^{(n-1,1)}|D^{(n-2,2)}$.

\item
If $n\equiv 2\Md 4$ then $S^{(n-2,2)}=D^{(n-1,1)}|D^{(n)}|D^{(n-2,2)}$.
\end{itemize}
\end{lemma}

\begin{proof}
If $n\equiv 0\Md 4$ then the lemma holds from  Corollary 12.2 and Theorem 24.15 of \cite{j1}.

If $n\equiv 2\Md 4$ then $S^{(n-2,2)}$ has compositions factors $D^{(n)}$, $D^{(n-1,1)}$ and $D^{(n-2,2)}$ from Theorem 24.15 of \cite{j1}. Also $D^{(n-2,2)}\cong\hd(S^{(n-2,2)})$ from Corollary 12.2 of \cite{j1} and $D^{(n)}\not\subseteq S^{(n-2,2)}$ from Theorem 24.4 of \cite{j1}. So the lemma holds also in this case.
\end{proof}

\begin{lemma}\label{l27}
If $p=2$ and $n\equiv 0\Md 2$ with $n\geq 4$ then $M^{(n-2,1)}=D^{(n-1)}\oplus D^{(n-2,1)}$. If further $n\equiv 0\Md 4$ with $n\geq 8$ then $M^{(n-3,2)}=D^{(n-1)}\oplus D^{(n-2,1)}\oplus D^{(n-3,2)}$.
\end{lemma}

\begin{proof}
From Theorem 24.15 of \cite{j1} we have $D^{(n-1)}\cong S^{(n)}$, $D^{(n-2,1)}\cong S^{(n-1,1)}$ as $n-1$ is odd. For $n\equiv 0\Md 4$ Theorem 24.15 of \cite{j1} also gives $D^{(n-3,2)}\cong S^{(n-3,2)}$. The lemma now follows as $M^{(n-2,1)}\sim S^{(n-2,1)}|S^{(n-1)}$ and $M^{(n-3,2)}\sim S^{(n-3,2)}|S^{(n-2,1)}|S^{(n-1)}$ by Example 17.17 of \cite{j1} and as $M^{(n-2,1)}$ and $M^{(n-3,2)}$ are self-dual.
\end{proof}

\begin{lemma}\label{l49}
If $p=2$ and $n\geq 4$ is even then $D^{(n-2,1)}\uparrow^{\s_n}\cong f_0D^{(n-2,1)}\cong Y^{(n-2,1,1)}$ is indecomposable with head and socle isomorphic to $D^{(n-1,1)}$. Also $[D^{(n-2,1)}\uparrow^{\s_n}:D^{(n-1,1)}]=3$.
\end{lemma}

\begin{proof}
For $n$ even we have that the residues of $(n-2,1)$ are given by
\[\begin{tikzpicture}
\draw (0,0) node {0};
\draw (0.3,0) node {1};
\draw (0.7,0) node {$\cdots$};
\draw (1.1,0) node {0};
\draw (1.4,0) node {1};
\draw (0,-0.4) node {1};
\draw (0.3,-0.4) node {0};
\draw (1.7,0) node {0};
\draw (0,-0.80) node {0};

\draw (-0.15,0.2)--(1.55,0.2)--(1.55,-0.2)--(0.15,-0.2)--(0.15,-0.6)--(-0.15,-0.6)--(-0.15,0.2);
\end{tikzpicture}\]
In particular $(1,n-1),(2,2),(3,1)$ are the conormal nodes of $(n-2,1)$ and they all have residue 0. So $D^{(n-2,1)}\uparrow^{\s_n}\cong f_0D^{(n-2,1)}$ is indecomposable with head and socle isomorphic to $D^{(n-1,1)}$ and $[D^{(n-2,1)}\uparrow^{\s_n}:D^{(n-1,1)}]=3$.

From Lemma \ref{l27} we have that
\[M^{(n-2,1,1)}\cong M^{(n-2,1)}\uparrow^{\s_n}\cong M^{(n-1,1)}\oplus D^{(n-2,1)}\uparrow^{\s_n}.\]
Since $D^{(n-2,1)}\uparrow^{\s_n}$ is indecomposable, it follows that $D^{(n-2,1)}\uparrow^{\s_n}\cong Y^{(n-2,1,1)}$.
\end{proof}

\begin{lemma}\label{l33}
Let $p=2$ and $n\geq 6$ be even.

\begin{itemize}
\item
If $n\equiv 0 \Md 4$ then $[S^{(n-2,1,1)}:D^{(n)}],[S^{(n-2,1,1)}:D^{(n-1,1)}],[S^{(n-2,1,1)}:D^{(n-2,2)}]=1$.

\item
If $n\equiv 2 \Md 4$ then $[S^{(n-2,1,1)}:D^{(n-1,1)}],[S^{(n-2,1,1)}:D^{(n-2,2)}]=1$ and $[S^{(n-2,1,1)}:D^{(n)}]=2$.
\end{itemize}

Also $\soc(S^{(n-2,1,1)})\cong D^{(n-1,1)}$.
\end{lemma}

\begin{proof}
From page 93 of \cite{j1} we have that the character of $S^{(n-2,1,1)}$ is equal to the sum of the characters of $S^{(n-2,2)}$ and of $S^{(n)}$. In particular, from Lemma \ref{l2}, the composition factors of $S^{(n-2,1,1)}$ are as given in the lemma.

From Lemmas \ref{l3} and \ref{l49} we have that
\[\soc(S^{(n-2,1,1)})\subseteq\soc(D^{(n-2,1)}\uparrow^{\s_n})\cong D^{(n-1,1)},\]
and so the second part of the lemma also holds.
\end{proof}

\begin{lemma}\label{l35}
If $p=2$ and $n\geq 4$ is even then $(S^{(n-1,1)})^*\subseteq S^{(n-2,1,1)}$.
\end{lemma}

\begin{proof}
For $2\leq i<j\leq n$ let $e_{i,j}$ represent the polytabloid corresponding to the standard tableau of shape $(n-2,1,1)$ with second and third row entries $i$ and $j$ respectively. We have
\[S^{(n-2,1,1)}=\langle e_{i,j}:2\leq i<j\leq n\rangle.\]
Also
\[(n-1,n)e_{i,j}=\left\{\begin{array}{ll}
e_{i,j},&j\leq n-2,\\
e_{i,n},&j=n-1,\\
e_{i,n-1},&i\leq n-2,\,\,j=n,\\
e_{n-1,n},&i=n-1,\,\,j=n
\end{array}\right.\]
and
\[(2,\ldots,n)e_{i,j}=\left\{\begin{array}{ll}
e_{i+1,j+1},&j\leq n-1,\\
e_{2,i+1},&j=n.
\end{array}\right.\]
So the restriction of $S^{(n-2,1,1)}$ to $\s_{1,n-1}=\langle (n-1,n),(2,\ldots,n)\rangle$ is a permutation representation. In particular
\begin{align*}
\dim\Hom_{\s_n}(M^{(n-1,1)},S^{(n-2,1,1)})&=\dim\Hom_{\s_n}(M^{(1,n-1)},S^{(n-2,1,1)})\\
&=\dim\Hom_{\s_{1,n-1}}(1_{\s_{1,n-1}},S^{(n-2,1,1)}\downarrow_{\s_{1,n-1}})\\
&\geq 1.
\end{align*}
From Lemma \ref{l1} and from the self-duality of $M^{(n-1,1)}$ we have that 
\[M^{(n-1,1)}=D^{(n)}|D^{(n-1,1)}|D^{(n)}\sim D^{(n)}|(S^{(n-1,1)})^*.\]
As $\soc(S^{(n-2,1,1)})\cong D^{(n-1,1)}$ from Lemma \ref{l33}, so that $D^{(n)}$ and $M^{(n-1,1)}$ are not contained in $S^{(n-2,1,1)}$, the lemma follows.
\end{proof}

\begin{lemma}\label{l50}
Let $p=2$ and $n\geq 6$ be even. Then $e_0D^{(n-1,2)}$ is both a submodule and a quotient of $D^{(n-2,1)}\uparrow^{\s_n}$ and $S^{(n-2,2)}\subseteq e_0D^{(n-1,2)}$.

Further $[e_0D^{(n-1,2)}:D^{(n-1,1)}]=2$, socle and head of $e_0D^{(n-1,2)}$ are isomorphic to $D^{(n-1,1)}$ and 
\begin{itemize}
\item
$e_0D^{(n-1,2)}\subseteq M\subseteq D^{(n-2,1)}\uparrow^{\s_n}$ with $M\sim S^{(n-2,1,1)}|D^{(n-1,1)}$,

\item
$e_0D^{(n-1,2)}\cong N$ or $e_0D^{(n-1,2)}\cong N/D^{(n)}$ with
\[N=D^{(n-2,1)}\uparrow^{\s_n}/S^{(n-2,1,1)}\sim S^{(n-2,2)}|S^{(n-1,1)}.\]
\end{itemize}
\end{lemma}

\begin{proof}
For $n$ even we have that
\[\begin{tikzpicture}
\draw (0,0) node {0};
\draw (0.3,0) node {1};
\draw (0.7,0) node {$\cdots$};
\draw (1.1,0) node {0};
\draw (1.4,0) node {1};
\draw (0,-0.4) node {1};
\draw (0.3,-0.4) node {0};
\draw (1.7,0) node {0};
\draw (0,-0.80) node {0};

\draw (-1.5,-0.2) node {$(n-2,1)=$};

\draw (-0.15,0.2)--(1.55,0.2)--(1.55,-0.2)--(0.15,-0.2)--(0.15,-0.6)--(-0.15,-0.6)--(-0.15,0.2);
\end{tikzpicture}\]
In particular $(1,n-2)$ and $(2,1)$ are the normal nodes (all of residue 1) and $(1,n-1),(2,2),(3,1)$ are the conormal nodes (all of residue 0) of $(n-2,1)$. It follows that $\epsilon_0(n-2,1)=0$ and $\phi_0(n-2,1)=3$. Also $\tilde{f}_0^2D^{(n-2,1)}=D^{(n-1,2)}$. So $e_0D^{(n-1,2)}\subseteq f_0D^{(n-2,1)}=D^{(n-2,1)}\uparrow^{\s_n}$ from Lemmas \ref{l9'} and \ref{l49}. Since $e_0D^{(n-1,2)}$ and $f_0D^{(n-2,1)}$ are self-dual by Lemma \ref{l57}, we also have that $e_0D^{(n-1,2)}$ is a quotient of $D^{(n-2,1)}\uparrow^{\s_n}$. Also head and socle of $e_0D^{(n-1,2)}$ are isomorphic to $D^{(n-1,1)}$, as so are those of $D^{(n-2,1)}\uparrow^{\s_n}$ (Lemma \ref{l49}). Further $\epsilon_0(n-1,2)=\epsilon_0(n-2,1)+2=2$ from Lemma \ref{l47}. So from Lemma \ref{l39} it follows that $[e_0D^{(n-1,2)}:D^{(n-1,1)}]=2$.

From Lemmas \ref{l3}, \ref{l1}, \ref{l2}, \ref{l33} and \ref{l49} we have that
\[D^{(n-2,1)}\hspace{-1pt}\uparrow^{\s_n}\sim \!\overbrace{\underbrace{D^{(n-1,1)}}_{\soc(D^{(n-1,1)}\uparrow^{\s_n})}|\underbrace{\ldots}_{\mbox{no }D^{(n-1,1)}}}^{S^{(n-2,1,1)}}\!|\!\overbrace{D^{(n-1,1)}|\underbrace{\ldots}_{\mbox{no }D^{(n-1,1)}}}^{S^{(n-2,2)}}\!|\!\overbrace{D^{(n)}|\underbrace{D^{(n-1,1)}}_{\hd(D^{(n-1,1)}\uparrow^{\s_n})}}^{S^{(n-1,1)}}\!\!.\]
From the previous part it follows that $e_0D^{(n-1,2)}\subseteq M\subseteq D^{(n-2,1)}\uparrow^{\s_n}$ with $M\sim S^{(n-2,1,1)}|D^{(n-1,1)}$ and that $e_0D^{(n-1,2)}$ is a quotient of $N=D^{(n-2,1)}\uparrow^{\s_n}/S^{(n-2,1,1)}\sim S^{(n-2,2)}|S^{(n-1,1)}$.

Let $A$ be a submodule with $N/A\cong e_0D^{(n-1,2)}$. Further let $B\subseteq N$ with $B\cong S^{(n-2,2)}$ and $N/B\cong S^{(n-1,1)}$. As $e_0D^{(n-1,2)}$ and $N$ both have exactly 2 composition factor isomorphic to $D^{(n-1,1)}$, no composition factor of $A$ is isomorphic to $D^{(n-1,1)}$. In particular $D^{(n-1,1)}\not\subseteq A$ and so, as $\soc(S^{(n-2,2)})\cong D^{(n-1,1)}$ (Lemma \ref{l2}), it follows that $A\cap B=0$. In particular $S^{(n-2,2)}\cong B\subseteq N/A\cong e_0D^{(n-1,1)}$. Also $A$ is isomorphic to a submodule of $N/B\cong S^{(n-1,1)}$ with no composition factor isomorphic to $D^{(n-1,1)}$. So $A=0$ or $A\cong D^{(n)}$ as $S^{(n-1,1)}\cong D^{(n)}|D^{(n-1,1)}$ (Lemma \ref{l1}).
\end{proof}

\begin{lemma}\label{l5}
If $p=2$ and $n\equiv 2\Md 4$ with $n\geq 6$ then
\[D^{(n-2,1)}\hspace{-1pt}\uparrow^{\s_n}=\hspace{-1pt}D^{(n-1,1)}|D^{(n)}|D^{(n-2,2)}|D^{(n)}|D^{(n-1,1)}|D^{(n)}|D^{(n-2,2)}|D^{(n)}|D^{(n-1,1)}.\]
\end{lemma}

\begin{proof}
This follows from Theorem 1.1 of \cite{mo} and from Lemma \ref{l49}.
\end{proof}

\begin{cor}
If $p=2$ and $n\equiv 2\Md 4$ with $n\geq 6$ then
\[S^{(n-2,1,1)}=D^{(n-1,1)}|D^{(n)}|D^{(n-2,2)}|D^{(n)}\]
and
\[e_0D^{(n-1,2)}=D^{(n-1,1)}|D^{(n)}|D^{(n-2,2)}|D^{(n)}|D^{(n-1,1)}.\]
\end{cor}

\begin{proof}
The structure of $S^{(n-2,1,1)}$ follows from Lemmas \ref{l3}, \ref{l33} and \ref{l5} (or from (2.4)c of \cite{mo}).

The structure of $e_0D^{(n-1,2)}$ follows from Lemmas \ref{l50} and \ref{l5}.
\end{proof}

\begin{lemma}\label{l25}
If $p=2$ and $n\equiv 0\Md 4$ with $n\geq 8$ then
\[\xymat{
\hspace{-24pt}&D^{(n-1)}\ar@{-}[dd]\ar@{-}[ddr]&D^{(n-3,2)}\ar@{-}[dd]\ar@{-}[ddl]\\
M^{(n-3,1,1)}=D^{(n-2,1)}\oplus D^{(n-2,1)}\oplus\hspace{-24pt}&&&.\\
\hspace{-24pt}&D^{(n-1)}&D^{(n-3,2)}
}\]
\end{lemma}

\begin{proof}
As $n$ is even we have that the content of $(n-3)$ is
\[\begin{tikzpicture}
\draw (0,0) node {0};
\draw (0.3,0) node {1};
\draw (0.7,0) node {$\cdots$};
\draw (1.1,0) node {0};
\draw (1.4,0) node {1};
\draw (0,-0.4) node {1};

\draw (-0.15,0.2)--(1.25,0.2)--(1.25,-0.2)--(-0.15,-0.2)--(-0.15,0.2);
\end{tikzpicture}\]
and so $\phi_0((n-3))=0$ and $\phi_1((n-3))=2$. So from Lemmas \ref{l45} and \ref{l40}, $f_1f_1D^{(n-3)}\cong D^{(n-2,1)}\oplus D^{(n-2,1)}$ is a direct summand of $M^{(n-3,1,1)}$. From Theorem 1.1 of \cite{mo} we have that
\[\xymat{
\hspace{-24pt}&D^{(n-1)}\ar@{-}[dd]\ar@{-}[ddr]&D^{(n-3,2)}\ar@{-}[dd]\ar@{-}[ddl]\\
Y^{(n-3,1,1)}=\hspace{-6pt}&&&.\\
\hspace{-24pt}&D^{(n-1)}&D^{(n-3,2)}
}\]
The lemma then follows by comparing degrees (through Theorem 24.15 of \cite{j1}).
\end{proof}

\begin{lemma}\label{l6'}
If $p=2$ and $n\equiv 0\Md 4$ with $n\geq 8$ then
\[e_0D^{(n-1,2)}=D^{(n-1,1)}|D^{(n-2,2)}|D^{(n-1,1)}.\]
\end{lemma}

\begin{proof}
 Notice that, as $n$ is even,
\[\begin{tikzpicture}
\draw (0,0) node {0};
\draw (0.3,0) node {1};
\draw (1,0) node {$\cdots$};
\draw (1.4,0) node {1};
\draw (1.7,0) node {0};
\draw (0,-0.4) node {1};
\draw (0.3,-0.4) node {0};
\draw (2,0) node {1};
\draw (0,-0.80) node {0};
\draw (0.6,0) node {0};
\draw (0.6,-0.4) node {1};

\draw (-1.5,-0.2) node {$(n-1,2)=$};

\draw (-0.15,0.2)--(1.85,0.2)--(1.85,-0.2)--(0.45,-0.2)--(0.45,-0.6)--(-0.15,-0.6)--(-0.15,0.2);
\end{tikzpicture}\]
and so the normal nodes of $(n-1,2)$ are $(n-1,1)$ and $(2,2)$ and they both have residue 0. In particular, from Lemmas \ref{l45}, \ref{l39} and \ref{l56},
\[D^{(n-1,2)}\downarrow_{\s_n}\cong e_0D^{(n-1,2)}\sim \overbrace{D^{(n-1,1)}}^{\soc(e_0D^{(n-1,2)})}|\ldots|D^{(n-2,2)}|\ldots|\overbrace{D^{(n-1,1)}}^{\hd(e_0D^{(n-1,2)})}.\]
The lemma then follows by comparing dimensions.
\end{proof}

\begin{lemma}\label{l6}
If $p=2$ and $n\equiv 0\Md 4$ with $n\geq 8$ then
\[\xymat{
&&D^{(n-1,1)}\ar@{-}[dr]\ar@{-}[dl]\\
&D^{(n)}\ar@{-}[ddrr]\ar@{-}[dd]&&D^{(n-2,2)}\ar@{-}[d]\ar@{-}[ddll]\\
D^{(n-2,1)}\uparrow^{\s_n}=&&&D^{(n-1,1)}\ar@{-}[d]&.\\
&D^{(n)}\ar@{-}[dr]&&D^{(n-2,2)}\ar@{-}[dl]\\
&&D^{(n-1,1)}
}\]
\end{lemma}

\begin{proof}
The lemma follows from Theorem 1.1 of \cite{mo} and from Lemma \ref{l49}.
\end{proof}

\begin{cor}
If $p=2$ and $n\equiv 0\Md 4$ with $n\geq 8$ then $S^{(n-2,1,1)}=D^{(n-1,1)}|(D^{(n)}\oplus D^{(n-2,2)})$.
\end{cor}

\begin{proof}
It follows from Lemmas \ref{l3}, \ref{l33} and \ref{l6} (or from (2.4)c of \cite{mo}).
\end{proof}

\begin{lemma}\label{l12}
If $p=2$ and $n\equiv 0\Md 4$ with $n\geq 8$ then
\[\xymat{
&D^{(n)}\ar@{-}[ddr]&D^{(n-1,1)}\ar@{-}[ddl]\ar@{-}[d]\\
M^{(n-2,2)}=&&D^{(n-2,2)}\ar@{-}[d]&.\\
&D^{(n)}&D^{(n-1,1)}
}\]
\end{lemma}

\begin{proof}
The lemma follows from Theorem 1.1 of \cite{mo} (and comparing dimensions) or from Figure 1 of \cite{kst}. 
\end{proof}

The next lemmas study the structure of certain submodules and quotients of $M^{(n-2,2)}$ and $D^{(n-2,1)}\uparrow^{\s_n}$.

\begin{lemma}\label{l59}
Let $p=2$ and $n\equiv 0\Md 4$ with $n\geq 8$. Then $M^{(n-2,2)}$ has unique submodules isomorphic to $D^{(n-1,1)}$ and $S^{(n-2,2)}$.

Further $M^{(n-2,2)}/D^{(n-1,1)}\cong D^{(n)}\oplus N$ with
\[N=(D^{(n)}\oplus D^{(n-2,2)})|D^{(n-1,1)}\sim D^{(n-2,2)}|S^{(n-1,1)}\]
and $M^{(n-2,2)}/S^{(n-2,2)}\cong D^{(n)}\oplus S^{(n-1,1)}$.
\end{lemma}

\begin{proof}
From Lemma \ref{l2} we have that $S^{(n-2,2)}=D^{(n-1,1)}|D^{(n-2,2)}$. Also by definition $S^{(n-2,2)}\subseteq M^{(n-2,2)}$. From Lemma \ref{l12}
\[\xymat{
&D^{(n)}\ar@{-}[ddr]&D^{(n-1,1)}\ar@{-}[ddl]\ar@{-}[d]\\
M^{(n-2,2)}=&&D^{(n-2,2)}\ar@{-}[d]\\
&D^{(n)}&D^{(n-1,1)}
}\]
and so $D^{(n-1,1)}$ is contained only once in $M^{(n-2,2)}$ and
\[\xymat{
&D^{(n)}&D^{(n-1,1)}\ar@{-}[ddl]\ar@{-}[d]&\\
M^{(n-2,2)}/D^{(n-1,1)}=&&D^{(n-2,2)}&=D^{(n)}\oplus N\\
&D^{(n)}&&
}\]
with $N=(D^{(n)}\oplus D^{(n-2,2)})|D^{(n-1,1)}$. Since $D^{(n-2,2)}$ is contained only once in $M^{(n-2,2)}/D^{(n-1,1)}$, we have that $M^{(n-2,2)}$ has a unique submodule of the form $D^{(n-1,1)}|D^{(n-2,2)}$ which is then isomorphic to $S^{(n-2,2)}$. Also
\[M^{(n-2,2)}/S^{(n-2,2)}=(M^{(n-2,2)}/D^{(n-1,1)})/D^{(n-2,2)}=D^{(n)}\oplus(N/D^{(n-2,2)}).\]
So to prove the lemma it is enough to prove that $N/D^{(n-2,2)}\cong S^{(n-1,1)}$. From Example 17.17 of \cite{j1} and Lemma \ref{l1} we have that
\[M^{(n-2,2)}/S^{(n-2,2)}\sim \overbrace{D^{(n)}|D^{(n-1,1)}}^{S^{(n-1,1)}}|\overbrace{D^{(n)}}^{S^{(n)}}.\]
Further, from the previous part of the proof,
\[M^{(n-2,2)}/S^{(n-2,2)}\cong (N/D^{(n-2,2)})\oplus D^{(n)}=(D^{(n)}|D^{(n-1,1)})\oplus D^{(n)}.\]
Since the head of $(N/D^{(n-2,2)})\oplus D^{(n)}$ contains a unique copy of $D^{(n)}$ it follows that $N/D^{(n-2,2)}\cong S^{(n-1,1)}$.
\end{proof}

\begin{lemma}\label{l61}
Let $p=2$ and $n\equiv 0\Md 4$ with $n\geq 8$. If $M\subseteq M^{(n-2,2)}$ and $N\subseteq D^{(n-2,1)}\uparrow^{\s_n}$ with $M,N=D^{(n-1,1)}|(D^{(n)}\oplus D^{(n-2,2)})$ then $M\cong N$.
\end{lemma}

\begin{proof}
From Lemma \ref{l6} we have that
\[\xymat{
&&D^{(n-1,1)}\ar@{-}[dr]\ar@{-}[dl]\\
&D^{(n)}\ar@{-}[ddrr]\ar@{-}[dd]&&D^{(n-2,2)}\ar@{-}[d]\ar@{-}[ddll]\\
D^{(n-2,1)}\uparrow^{\s_n}=&&&D^{(n-1,1)}\ar@{-}[d]&.\\
&D^{(n)}\ar@{-}[dr]&&D^{(n-2,2)}\ar@{-}[dl]\\
&&D^{(n-1,1)}
}\]
Also from Lemma \ref{l12}
\[\xymat{
&D^{(n)}\ar@{-}[ddr]&D^{(n-1,1)}\ar@{-}[ddl]\ar@{-}[d]\\
M^{(n-2,2)}=&&D^{(n-2,2)}\ar@{-}[d]&.\\
&D^{(n)}&D^{(n-1,1)}
}\]
As $\soc(D^{(n-2,1)}\uparrow^{\s_n})=D^{(n-1,1)}$, the only quotients of $M^{(n-2,2)}$ which can be contained in $D^{(n-2,1)}\uparrow^{\s_n}$ are $D^{(n-1,1)}$ and $M^{(n-2,2)}/D^{(n)}$. Let $M=D^{(n-1,1)}|(D^{(n)}\oplus D^{(n-2,2)})\subseteq M^{(n-2,2)}$. Then
\[M^{(n-2,2)}\sim (D^{(n)}\oplus M)|D^{(n-1,1)}\sim D^{(n)}|\overbrace{M|D^{(n-1,1)}}^{M^{(n-2,2)}/D^{(n)}}.\]
In particular, up to isomorphism, $M\subseteq M^{(n-2,2)}/D^{(n)}$. Using  Lemma \ref{l27} and (1.3.5) and Corollary 1.3.11 of \cite{jk} we have that
\begin{align*}
&\dim\Hom_{\s_n}(M^{(n-2,2)},D^{(n-2,1)}\uparrow^{\s_n})\\
&\hspace{12pt}=\dim\Hom_{\s_n}(M^{(n-2,2)},M^{(n-2,1,1)})-\dim\Hom_{\s_n}(M^{(n-2,2)},M^{(n-1,1)})\\
&\hspace{12pt}=4-2\\
&\hspace{12pt}=2.
\end{align*}
As $D^{(n-1,1)}$ is contained only once in $D^{(n-2,1)}\uparrow^{\s_n}$ it follows that (up to isomorphism) $M\subseteq M^{(n-2,2)}/D^{(n)}\subseteq D^{(n-2,1)}\uparrow^{\s_n}$. As there exists a unique submodule of $D^{(n-2,1)}\uparrow^{\s_n}$ of the form $D^{(n-1,1)}|(D^{(n)}\oplus D^{(n-2,2)})$, the lemma follows.
\end{proof}

\begin{lemma}\label{l62}
Let $p=2$ and $n\equiv 0\Md 4$ with $n\geq 8$. Let $M$ be a quotient of $D^{(n-2,1)}\uparrow^{\s_n}$ with $D^{(n-2,2)}\subseteq M$. Then there exists a quotient $N$ of $M^{(n-2,2)}$ with $D^{(n-2,2)}\subseteq N\subseteq M$.
\end{lemma}

\begin{proof}
From Lemma \ref{l6} if $M$ is a quotient of $D^{(n-2,1)}\uparrow^{\s_n}$ with $D^{(n-2,2)}\subseteq M$ then $M$ has one of the following forms:
\[\xymat{
&&D^{(n-1,1)}\ar@{-}[dr]\ar@{-}[dl]\\
&D^{(n)}\ar@{-}[ddrr]\ar@{-}[dd]&&D^{(n-2,2)}\ar@{-}[d]\ar@{-}[ddll]\\
M_1=&&&D^{(n-1,1)}\ar@{-}[d]&,\\
&D^{(n)}&&D^{(n-2,2)}
}\]
\vspace{6pt}
\[\xymat{
&&D^{(n-1,1)}\ar@{-}[dr]\ar@{-}[dl]\\
&D^{(n)}\ar@{-}[ddrr]&&D^{(n-2,2)}\ar@{-}[d]\\
M_2=&&&D^{(n-1,1)}\ar@{-}[d]&,\\
&&&D^{(n-2,2)}
}\]
\vspace{6pt}
\[\xymat{
&&D^{(n-1,1)}\ar@{-}[dr]\ar@{-}[dl]\\
M_3=&D^{(n)}&&D^{(n-2,2)}&,
}\]
\vspace{6pt}
\[\xymat{
&&D^{(n-1,1)}\ar@{-}[dr]\\
M_4=&&&D^{(n-2,2)}&.
}\]
From Lemmas \ref{l50}, \ref{l6'} and \ref{l6} we have that
\[M_4\cong (e_0D^{(n-1,2)})/D^{(n-1,1)}\subseteq M_1,M_2.\]
In order to prove the lemma it is then enough to prove that $M_3$ and $M_4$ are isomorphic to quotients of $M^{(n-2,2)}$. For $M_3$ this holds from Lemma \ref{l61} and by self-duality of $D^{(n-2,1)}\uparrow^{\s_n}$ and $M^{(n-2,2)}$ (since $M_3=(D^{(n)}\oplus D^{(n-2,2)})|D^{(n-1,1)}$). Since $M_4\cong M_3/D^{(n)}$ it follows that $M_4$ is also isomorphic to a quotient of $M^{(n-2,2)}$ and so the lemma holds.
\end{proof}

\begin{lemma}\label{l63}
Let $p=2$ and $n\equiv 0\Md 4$ with $n\geq 8$. Let $M$ be a quotient of $D^{(n-2,1)}\uparrow^{\s_n}$ with $M\not\cong D^{(n-1,1)}$ and such that no submodule of $M$ is of the form $D^{(n)}|D^{(n)}$. Also let $N\cong (S^{(n-1,1)})^*$ with $D^{(n-1,1)}\subseteq M\cap N$. Then $K\subseteq L\subseteq M+N$ with $L$ a quotient of $M^{(n-2,2)}$ and $K=D^{(n-1,1)}|(D^{(n)}\oplus D^{(n-2,2)})\subseteq M^{(n-2,2)}$.
\end{lemma}

\begin{proof}
From Lemma \ref{l6} $M$ is of one of the following forms:
\[M_1=D^{(n-2,1)}\uparrow^{\s_n},\]
\vspace{6pt}
\[\xymat{
&&D^{(n-1,1)}\ar@{-}[dr]\ar@{-}[dl]\\
M_2=&D^{(n)}&&D^{(n-2,2)}\ar@{-}[d]&,\\
&&&D^{(n-1,1)}\\
}\]
\vspace{6pt}
\[\xymat{
&&D^{(n-1,1)}\ar@{-}[dr]\\
M_3=&&&D^{(n-2,2)}\ar@{-}[d]&.\\
&&&D^{(n-1,1)}\\
}\]
From Lemmas \ref{l50}, \ref{l6'} and \ref{l6} we have that
\[M_3\cong e_0D^{(n-1,2)}\subseteq M_1.\]
In particular we can assume that $M\cong M_2$ or $M\cong M_3$. Let $N_2,N_3\cong (S^{(n-1,1)})^*$ with $D^{(n-1,1)}\subseteq M_2\cap N_2,M_3\cap N_3$. Since $(S^{(n-1,1)})^*\not\subseteq M_2,M_3$ (by the structure of $M_2$ and $M_3$ and Lemma \ref{l1}), we have that $M_2\cap N_2,M_3\cap N_3\cong D^{(n-1,1)}$. As $M_3\cong M_2/D^{(n)}$ it then follows that
\[M_3+N_3\cong (M_2+N_2)/D^{(n)}.\]
Assume that $M_2+N_2\cong M^{(n-2,2)}$. Then $M_3+N_3\cong M^{(n-2,2)}/D^{(n)}$ and so the lemma holds (the submodules of $M^{(n-2,2)}$ and of $M^{(n-2,2)}$ of the form $D^{(n-1,1)}|(D^{(n)}\oplus D^{(n-2,2)})$ are isomorphic, as
\[M^{(n-2,2)}\sim (D^{(n)}\oplus (D^{(n-1,1)}|(D^{(n)}\oplus D^{(n-2,2)})))|D^{(n-1,1)},\]
from Lemma \ref{l12}).

So it is enough to prove that $M_2+N_2\cong M^{(n-2,2)}$. Since $M_2\cap N_2\cong D^{(n-1,1)}$ it is then enough to prove that $M^{(n-2,2)}=\overline{M_2}+\overline{N_2}$ with $\overline{M_2}\cong M_2$, $\overline{N_2}\cong N_2\cong (S^{(n-1,1)})^*$ and $\overline{M_2}\cap\overline{N_2}\cong D^{(n-1,1)}$. From Lemmas \ref{l3} and \ref{l35} there exists $\overline{N_2}\subseteq M^{(n-2,2)}$ with $\overline{N_2}\cong (S^{(n-1,1)})^*$. Let $\overline{M_2}\subseteq M^{(n-2,2)}$ with
\[\overline{M_2}\sim(D^{(n)}\oplus (D^{(n-1,1)}|D^{(n-2,2)}))|D^{(n-1,1)}\]
(such a submodule of $M^{(n-2,2)}$ exists and is isomorphic to $(M^{(n-2,2)}/D^{(n)})^*$ from Lemma \ref{l12} and self-duality of $M^{(n-2,2)}$). From the proof of Lemma \ref{l61} and by self-duality of $M^{(n-2,2)}$ and $D^{(n-2,1)}\uparrow^{\s_n}$ we have that $\overline{M_2}$ is isomorphic to a quotient of $D^{(n-2,1)}\uparrow^{\s_n}$. From the structure of $\overline{M_2}$ and from Lemma \ref{l6} it follows that $\overline{M_2}\cong M_2$ (since $M_2$ is the only quotient of $D^{(n-2,1)}\uparrow^{\s_n}$ of the form $(D^{(n)}\oplus (D^{(n-1,1)}|D^{(n-2,2)}))|D^{(n-1,1)}$).

We will next prove that $\overline{M_2}\cap\overline{N_2}\cong D^{(n-1,1)}$. From Lemma \ref{l1} and definition of $\overline{M_2}$ and of $\overline{N_2}$ we have that both $\overline{M_2}$ and $\overline{N_2}$ contain a submodule isomorphic to $D^{(n-1,1)}$. Since both $\overline{M_2}$ and $\overline{N_2}$ are submodules of $M^{(n-2,2)}$ and since $M^{(n-2,2)}$ contains only one submodule isomorphic to $D^{(n-1,1)}$ (from Lemma \ref{l12}), it follows that $D^{(n-1,1)}\subseteq \overline{M_2}\cap \overline{N_2}$. As $\overline{N_2}\cong (S^{(n-1,1)})^*=D^{(n-1,1)}|D^{(n)}$ and as $\overline{M_2}$ does not contain any uniserial module of the form $D^{(n-1,1)}|D^{(n)}$ (the only composition factor of $\overline{M_2}$ isomorphic to $D^{(n)}$ is in its socle), we then have that $\overline{M_2}\cap\overline{N_2}\cong D^{(n-1,1)}$.

Since $\overline{M_2}\cap\overline{N_2}\cong D^{(n-1,1)}$ and $\overline{M_2}+\overline{N_2}\subseteq M^{(n-2,2)}$, comparing compositions factors we obtain that $\overline{M_2}+\overline{N_2}=M^{(n-2,2)}$. The lemma then follows.
\end{proof}

\begin{lemma}\label{l26}
If $p=2$ and $n\geq 2$ and $M$ is an $\s_n$ module with $M=D^{(n)}|D^{(n)}$, then $M\cong 1\uparrow_{A_n}^{\s_n}$.
\end{lemma}

\begin{proof}
We can choose a basis of $M$ so that the matrix representation has the form $\pi\mapsto\left(\begin{array}{cc}1&x_\pi\\0&1\end{array}\right)$. Further we may assume that $x_{(1,2)}\in\{0,1\}$. As matrices of the form $\left(\begin{array}{cc}1&x\\0&1\end{array}\right)$ commute, we have that $x_\rho=x_{(1,2)}$ for any transposition $\rho$ and then that $x_\pi=x_{(1,2)}\delta_{\pi\not\in A_n}$ for $\pi\in \s_n$. As $M=D^{(n)}|D^{(n)}$ it follows $x_{(1,2)}=1$. As $\left(\begin{array}{cc}1&1\\0&1\end{array}\right)$ and $\left(\begin{array}{cc}0&1\\1&0\end{array}\right)$ are conjugated in $\GL_2(2)$, the lemma follows.
\end{proof}

\section{Dimensions of homomorphism rings}

The next lemmas show that under, some assumptions on the dimensions of some homomorphism spaces, certain specific modules are contained in $\End_F(D^\lambda)$. These assumptions will be verified for most partitions in Sections \ref{s2} and \ref{s3}.

\begin{lemma}\label{l17}
Let $p=2$ and $n\equiv 2\Md 4$ with $n\geq 6$. If
\begin{align*}
&\dim\End_{\s_{n-1}}(D^\lambda\downarrow_{\s_{n-1}})+\dim\Hom_{\s_n}(S^{(n-1,1)},\End_F(D^\lambda))+1\\
&\hspace{12pt}<\dim\End_{\s_{n-2}}(D^\lambda\downarrow_{\s_{n-2}}),
\end{align*}
then $M\subseteq \End_F(D^\lambda)$, where $M$ is a quotient of $D^{(n-2,1)}\uparrow^{\s_n}$ of one of the following forms:
\begin{itemize}
\item
$D^{(n-2,2)}|D^{(n)}|D^{(n-1,1)}$,

\item
$D^{(n)}|D^{(n-2,2)}|D^{(n)}|D^{(n-1,1)}$,

\item
$D^{(n-1,1)}|D^{(n)}|D^{(n-2,2)}|D^{(n)}|D^{(n-1,1)}$.
\end{itemize}
\end{lemma}

\begin{proof}
As $n$ is even, so that $M^{(n-2,1)}=D^{(n-1)}\oplus D^{(n-2,1)}$,
\begin{align*}
&\dim\End_{\s_{n-2}}(D^\lambda\downarrow_{\s_{n-2}})\\
&\hspace{12pt}=\dim\Hom_{\s_{n-2}}(D^{(n-2)},\End_F(D^\lambda)\downarrow_{\s_{n-2}})\\
&\hspace{12pt}=\dim\Hom_{\s_{n-1}}(M^{(n-2,1)},\End_F(D^\lambda)\downarrow_{\s_{n-1}})\\
&\hspace{12pt}=\dim\Hom_{\s_{n-1}}(D^{(n-1)},\End_F(D^\lambda\downarrow_{\s_{n-1}}))\\
&\hspace{36pt}+\dim\Hom_{\s_{n-1}}(D^{(n-2,1)},\End_F(D^\lambda)\downarrow_{\s_{n-1}})\\
&\hspace{12pt}=\dim\End_{\s_{n-1}}(D^\lambda\downarrow_{\s_{n-1}})+\dim\Hom_{\s_n}(D^{(n-2,1)}\uparrow^{\s_n},\End_F(D^\lambda)).
\end{align*}
So the assumption of the lemma is equivalent to
\[\dim\Hom_{\s_n}(D^{(n-2,1)}\uparrow^{\s_n},\End_F(D^\lambda))>\dim\Hom_{\s_n}(S^{(n-1,1)},\End_F(D^\lambda))+1.\]
From Lemmas \ref{l3}, \ref{l1}, \ref{l2}, \ref{l33} and \ref{l5} we have that $D^{(n-2,1)}\uparrow^{\s_n}$ is uniserial of the form
\[\overbrace{D^{(n-1,1)}|D^{(n)}|D^{(n-2,2)}|D^{(n)}}^{S^{(n-2,1,1)}}|\overbrace{D^{(n-1,1)}|D^{(n)}|D^{(n-2,2)}}^{S^{(n-2,2)}}|\overbrace{D^{(n)}|D^{(n-1,1)}}^{S^{(n-1,1)}}.\]
From Lemma \ref{l50} it follows that
\[\rlap{$\overbrace{\phantom{D^{(n-1,1)}|D^{(n)}|D^{(n-2,2)}|D^{(n)}|D^{(n-1,1)}}}^{e_0D^{(n-1,2)}}$}D^{(n-1,1)}|D^{(n)}|D^{(n-2,2)}|D^{(n)}|\underbrace{D^{(n-1,1)}|D^{(n)}|D^{(n-2,2)}|\overbrace{D^{(n)}|D^{(n-1,1)}}^{S^{n-1,1}}}_{e_0D^{(n-1,2)}}.\]

If
\begin{align*}
&\dim\Hom_{\s_n}(D^{(n-2,1)}\uparrow^{\s_n},\End_F(D^\lambda))\\
&\hspace{12pt}>\dim\Hom_{\s_n}(D^{(n-2,1)}\uparrow^{\s_n}/(D^{(n-1,1)}|D^{(n)}|D^{(n-2,2)}),\End_F(D^\lambda))
\end{align*}
then $\End_F(D^\lambda)$ contains a quotient of $D^{(n-2,1)}\uparrow^{\s_n}$ which is not a quotient of $D^{(n-2,1)}\uparrow^{\s_n}/(D^{(n-1,1)}|D^{(n)}|D^{(n-2,2)})$. As $D^{(n-2,1)}\uparrow^{\s_n}$ is uniserial such quotients are given by the modules $D^{(n-2,1)}\uparrow^{\s_n}/M$ with $M$ strictly contained in $D^{(n-1,1)}|D^{(n)}|D^{(n-2,2)}$. So $\End_F(D^\lambda)$ contains a module of one of the forms
\begin{align*}
\overbrace{D^{(n-1,1)}|D^{(n)}|D^{(n-2,2)}|D^{(n)}|D^{(n-1,1)}}^{e_0D^{(n-1,2)}}|D^{(n)}|D^{(n-2,2)}|D^{(n)}|D^{(n-1,1)},\\
\overbrace{D^{(n)}|D^{(n-2,2)}|D^{(n)}|D^{(n-1,1)}}^{(e_0D^{(n-1,2)})/D^{(n-1,1)}}|D^{(n)}|D^{(n-2,2)}|D^{(n)}|D^{(n-1,1)},\\
\overbrace{D^{(n-2,2)}|D^{(n)}|D^{(n-1,1)}}^{(e_0D^{(n-1,2)})/(D^{(n-1,1)}|D^{(n)})}|D^{(n)}|D^{(n-2,2)}|D^{(n)}|D^{(n-1,1)}.
\end{align*}
As $e_0D^{(n-1,2)}$ is a quotient of $D^{(n-2,1)}\uparrow^{\s_n}$ the lemma holds.

So assume now that 
\begin{align*}
&\dim\Hom_{\s_n}(D^{(n-2,1)}\uparrow^{\s_n},\End_F(D^\lambda))\\
&\hspace{12pt}=\dim\Hom_{\s_n}(D^{(n-2,1)}\uparrow^{\s_n}/(D^{(n-1,1)}|D^{(n)}|D^{(n-2,2)}),\End_F(D^\lambda)).
\end{align*}
As $D^{(n-2,1)}\uparrow^{\s_n}/(D^{(n-1,1)}|D^{(n)}|D^{(n-2,2)})\sim D^{(n)}|e_0D^{(n-1,2)}$ we have from Lemma \ref{l34} that
\begin{align*}
&\dim\Hom_{\s_n}(e_0D^{(n-1,2)},\End_F(D^\lambda))\\
&\hspace{12pt}\geq \dim\Hom_{\s_n}(D^{(n-2,1)}\uparrow^{\s_n}/(D^{(n-1,1)}|D^{(n)}|D^{(n-2,2)}),\End_F(D^\lambda))\\
&\hspace{36pt}-\dim\Hom_{\s_n}(D^{(n)},\End_F(D^\lambda))\\
&\hspace{12pt}=\dim\Hom_{\s_n}(D^{(n-2,1)}\uparrow^{\s_n}/(D^{(n-1,1)}|D^{(n)}|D^{(n-2,2)}),\End_F(D^\lambda))-1\\
&\hspace{12pt}>\dim\Hom_{\s_n}(S^{(n-1,1)},\End_F(D^\lambda)).
\end{align*}
So, from
\[e_0D^{(n-1,2)}=D^{(n-1,1)}|D^{(n)}|D^{(n-2,2)}|\overbrace{D^{(n)}|D^{(n-2,2)}}^{S^{(n-1,1)}},\]
it follows that one of $(e_0D^{(n-1,2)})/(D^{(n-1,1)}|D^{(n)})$, $(e_0D^{(n-1,2)})/D^{(n-1,1)}$ or $e_0D^{(n-1,2)}$ is contained in $\End_F(D^\lambda)$ and so also in this case the lemma holds.
\end{proof}

\begin{lemma}\label{l36}
Let $p=2$ and $n\equiv 0\Md 4$ with $n\geq 8$. If
\[\dim\End_{\s_{n-2,2}}(D^\lambda\downarrow_{\s_{n-2,2}})>\dim\Hom_{\s_n}(S^{(n-1,1)},\End_F(D^\lambda))+1,\]
then $M\subseteq L\subseteq \End_F(D^\lambda)$, where $L$ is a quotient of $M^{(n-2,2)}$ and $M$ is of one of the following forms:
\begin{itemize}
\item
$M\cong D^{(n-2,2)}$,

\item
$M=D^{(n-1,1)}|(D^{(n)}\oplus D^{(n-2,2)})\subseteq M^{(n-2,2)}$.
\end{itemize}
\end{lemma}

\begin{proof}
From Lemma \ref{l12}
\[\xymat{
&D^{(n)}\ar@{-}[ddr]&D^{(n-1,1)}\ar@{-}[ddl]\ar@{-}[d]\\
M^{(n-2,2)}=&&D^{(n-2,2)}\ar@{-}[d]&.\\
&D^{(n)}&D^{(n-1,1)}
}\]
Notice that the only quotients of $M^{(n-2,2)}$ but not of $M^{(n-2,2)}/D^{(n-1,1)}$ are of the form $M^{(n-2,2)}$ or $M^{(n-2,2)}/D^{(n)}$, where
\[\xymat{
&D^{(n)}\ar@{-}[ddr]&D^{(n-1,1)}\ar@{-}[d]\\
M^{(n-2,2)}/D^{(n)}=&&D^{(n-2,2)}\ar@{-}[d]\\
&&D^{(n-1,1)}
}\]
(no other quotients are possible since the socle of $M^{(n-2,2)}/D^{(n)}$ is ismorphic to $D^{(n-1,1)}$). Further the submodules of the form $D^{(n-1,1)}|(D^{(n)}\oplus D^{(n-2,2)})$ contained in $M^{(n-2,2)}$ and $M^{(n-2,2)}/D^{(n)}$ are isomorphic (by the structure of $M^{(n-2,2)}$ or from the proof of Lemma \ref{l63}).

Assume now that $\End_F(D^\lambda)$ does not contain any module
\[M=D^{(n-1,1)}|(D^{(n)}\oplus D^{(n-2,2)})\subseteq M^{(n-2,2)}.\]
Then neither $M^{(n-2,2)}$ nor $M^{(n-2,2)}/D^{(n)}$ are contained in $\End_F(D^\lambda)$ and so
\[\dim\Hom_{\s_n}\hspace{-1pt}(M^{(n-2,2)},\End_F(D^\lambda)\hspace{-1pt})\!=\!\dim\Hom_{\s_n}\hspace{-1pt}(M^{(n-2,2)}\hspace{-1pt}/\hspace{-1pt}D^{(n-1,1)},\End_F(D^\lambda)\hspace{-1pt}).\]
So, from Lemma \ref{l34} and by assumption,
\begin{align*}
&\dim\Hom_{\s_n}\hspace{-1pt}(M^{(n-2,2)}/D^{(n-1,1)},\End_F(D^\lambda)\hspace{-1pt})\\
&\hspace{12pt}=\dim\Hom_{\s_n}\hspace{-1pt}(M^{(n-2,2)},\End_F(D^\lambda)\hspace{-1pt})\\
&\hspace{12pt}=\dim\End_{\s_{n-2,2}}\hspace{-1pt}(D^\lambda\downarrow_{\s_{n-2,2}})\\
&\hspace{12pt}>\dim\Hom_{\s_n}\hspace{-1pt}(S^{(n-1,1)},\End_F(D^\lambda)\hspace{-1pt})\!+\!1\\
&\hspace{12pt}=\dim\Hom_{\s_n}\hspace{-1pt}(D^{(n)}\oplus S^{(n-1,1)},\End_F(D^\lambda)\hspace{-1pt}).
\end{align*}
Since from Lemma \ref{l59}
\[M^{(n-2,2)}/D^{(n-1,1)}\sim D^{(n)}\oplus (D^{(n-2,2)}|S^{(n-1,1)})\sim D^{(n-2,2)}|(D^{(n)}\oplus S^{(n-1,1)}),\]
it follows that there exists a quotient $L$ of $M^{(n-2,2)}$ (and so also of $M^{(n-2,2)}$) with $D^{(n-2,2)}\subseteq L\subseteq\End_F(D^\lambda)$.
\end{proof}

\label{s7}

We will now find lower/upper bounds depending on $\epsilon_0(\lambda)$ and $\epsilon_1(\lambda)$ for $\dim\End_{\s_{n-2}}(D^\lambda\downarrow_{\s_{n-2}})$ and $\dim\Hom_{\s_n}(S^{(n-1,1)},\End_F(D^\lambda))$. These bounds will be later used to prove that the assumptions in Lemmas \ref{l17} and \ref{l36} hold, at least when $\epsilon_0(\lambda)+\epsilon_1(\lambda)$ is large enough.

\begin{lemma}\label{l23}
Let $\lambda\vdash n$ be $p$-regular. Assume that $\epsilon_i(\lambda)\geq 1$ and let $\nu$ with $D^\nu=\tilde{e}_iD^\lambda$. Assume further that $\epsilon_j(\nu)\geq 1$. Then
\[\dim\End_{\s_{n-2}}(e_je_iD^\lambda)\geq \epsilon_i(\lambda)+\epsilon_j(\nu)-1.\]
\end{lemma}

\begin{proof}
For $f\in\End_{\s_{n-1}}(e_iD^\lambda)$ let $\overline{f}$ be the restriction of $f$ to $e_je_iD^\lambda$. By definition of $e_i$ and $e_j$ we have that $e_iD^\lambda$ is contained in a unique block of $F\s_{n-1}$ and so $e_je_iD^\lambda$ is the restriction of $e_iD^\lambda\downarrow_{\s_{n-2}}$ to a certain block of $\s_{n-2}$. As $f\in\End_{\s_{n-1}}(e_iD^\lambda)\subseteq\End_{\s_{n-2}}(e_iD^\lambda\downarrow_{\s_{n-2}})$, so that $f$ acts blockwise on $e_iD^\lambda\downarrow_{\s_{n-2}}$, we have that $\overline{f}\in\End_{\s_{n-2}}(e_je_iD^\lambda)$.

From Lemma \hyperref[l39a]{\ref*{l39}\ref*{l39a}} we have $\hd(e_iD^\lambda)\cong D^\nu$ and by assumption $e_jD^\nu\not=0$, so that $0\not=e_j\hd(e_iD^\lambda)\subseteq \hd(e_iD^\lambda)\downarrow_{\s_{n-2}}$. If $f\in\End_{\s_{n-1}}(e_iD^\lambda)$ is non-zero, then $\hd(e_iD^\lambda)\not\subseteq \ker(f)$. As $\hd(e_iD^\lambda)$ is simple it then follows that $\hd(e_iD^\lambda)\cap\ker(f)=0$ and so $e_j\hd(e_iD^\lambda)\cap\ker(f)=0$. So $e_j\hd(e_iD^\lambda)\not\subseteq\ker(f)$ and then $\overline{f}\not=0$.

Let now $g\in\End_{\s_{n-2}}(e_jD^\nu)$. As $\soc(e_iD^\lambda)$ and $\hd(e_iD^\lambda)$ are isomorphic to $D^\nu$ from Lemma \hyperref[l39a]{\ref*{l39}\ref*{l39a}}, we can consider $g$ as
\[g\in\Hom_{\s_{n-2}}(e_j\hd(e_iD^\lambda),e_j\soc(e_iD^\lambda)).\]
In particular $g$ defines an endomorphism $\overline{g}\in\End_{\s_{n-2}}(e_je_iD^\lambda)$ with $\overline{g}\not=0$ if $g\not=0$.

Assume now that $\overline{f}=\overline{g}\not=0$ for some $f\in\End_{\s_{n-1}}(e_iD^\lambda)$ and $g\in\End_{\s_{n-2}}(e_jD^\nu)$. Then $f\not=0$, in particular $\soc(e_iD^\lambda)\subseteq\Im(f)$ (as the socle of $e_iD^\lambda$ is simple). From
\[e_j\soc(e_iD^\lambda)\subseteq \Im(\overline{f})=\Im(\overline{g})\subseteq e_j\soc(e_iD^\lambda)\]
it follows that
\[\Im(\overline{f})=e_j\soc(e_iD^\lambda)\subseteq \soc(e_iD^\lambda)\downarrow_{\s_{n-2}}.\]
As $\overline{f}\not=0$ and so $f\not=0$ and as $\hd(e_iD^\lambda)\cong D^\nu$ from Lemma \hyperref[l39a]{\ref*{l39}\ref*{l39a}} it follows that
\[0\not=\hd(\Im(f))\subseteq \hd(e_iD^\lambda)\cong D^\nu\]
and then $\hd(\Im(f))\cong D^\nu$. Further $e_j(\Im(f))=\Im(\overline{f})$, as $f$ acts blockwise on $e_i(D^\lambda)\downarrow_{\s_{n-2}}$ and by definition of $\overline{f}$. In particular $e_j\hd(\Im(f))$ is a quotient of $\Im(\overline{f})$. So
\[e_jD^\nu\cong e_j\hd(\Im(f))=\Im(\overline{f})/A=(e_j\soc(e_iD^\lambda))/A\cong (e_jD^\nu)/B,\]
for certain submodules $A\subseteq e_j\soc(e_iD^\lambda)$ and $B\subseteq e_jD^\nu$ with $A\cong B$. It follows that $A,B=0$ and then
\[0\not= e_j\soc(e_iD^\lambda)=e_j\hd(\Im(f))\subseteq \hd(\Im(f))\downarrow_{\s_{n-2}}.\]
As $\soc(e_iD^\lambda)\cong D^\nu$ is simple we then have (by the previous part) that
\[D^\nu\cong\soc(e_iD^\lambda)\subseteq \hd(\Im(f))\cong D^\nu\]
and then that $\hd(\Im(f))=\soc(e_iD^\lambda)$. In particular $f$ can be seen as an element of $\Hom_{\s_{n-1}}(\hd(e_iD^\lambda),\soc(e_iD^\lambda))\cong\End_{\s_{n-1}}(D^\nu)$.

By the first part of the proof we also have that $f\mapsto \overline{f}$ and $g\mapsto \overline{g}$ are injective. So
\begin{align*}
&\dim\End_{\s_{n-2}}(e_je_iD^\lambda)\\
&\hspace{12pt}\geq\dim\langle\{ \overline{f}:f\in\End_{\s_{n-1}}(e_iD^\lambda)\}\cup\{\overline{g}:g\in\End_{\s_{n-2}}(e_jD^\nu)\}\rangle\\
&\hspace{12pt}=\dim\End_{\s_{n-1}}(e_iD^\lambda)+\dim\End_{\s_{n-2}}(e_jD^\nu)-\dim\End_{\s_{n-1}}(D^\nu)\\
&\hspace{12pt}=\epsilon_i(\lambda)+\epsilon_j(\nu)-1.
\end{align*}
\end{proof}

\begin{lemma}\label{l4'}
Let $p=2$ and $\lambda$ be 2-regular. Assume that $\epsilon_0(\lambda),\epsilon_1(\lambda)\geq 1$. Let $x_0$ and $x_1$ be the lowest normal nodes of $\lambda$ of residue 0 and 1 respectively. Also let $\nu_0=\lambda\setminus\{x_0\}$ and $\nu_1=\lambda\setminus\{x_1\}$. Then $\tilde{e}_0D^\lambda=D^{\nu_0}$ and $\tilde{e}_1D^\lambda=D^{\nu_1}$.

Further if $0\leq i\leq 1$ and $j=1-i$ are such that $x_i$ is above $x_j$ then $\epsilon_i(\nu_j)\geq\epsilon_i(\lambda)$ and $\epsilon_j(\nu_i)=\epsilon_j(\lambda)+2$.
\end{lemma}

\begin{proof}
The first part of the lemma follows from the definition of $\tilde{e}_0D^\lambda$ and $\tilde{e}_1D^\lambda$.

By definition of $i$, all normal nodes of $\lambda$ of residue $i$ are above the node $x_j$. So all normal nodes of $\lambda$ of residue $i$ are also normal in $\nu_j$ and then $\epsilon_i(\nu_j)\geq\epsilon_i(\lambda)$.

We will now show that $\epsilon_j(\nu_i)=\epsilon_j(\lambda)+2$. To see this let $y$ and $z$ be the nodes to the right and to the left of $x_i$ (notice that $z$ is a node of $\nu_i$, since the row of $\lambda$ containing $x_i$ contains at least 2 nodes, as $x_i$ is above $x_j$ and so not on the last row of $\lambda$ and as $\lambda$ is 2-regular). Then $y$ and $z$ both have residue $j$. As $x_i$ is the bottom normal $i$-node of $\lambda$ we have that $\nu_i$ is also 2-regular (since $D^{\nu_i}$ is defined). As all parts of $\lambda$ and $\nu_i$ are distinct
\[\begin{tikzpicture}
\draw (-2,-1.5) node {$\lambda=$};

\draw (2.3,-0.6) node {$\iddots$};
\draw (1,-1.2) node {$z$};
\draw (1.37,-1.22) node {$x_i$};
\draw (1.8,-1.2) node {$y$};
\draw (0.5,-1.6) node {$\iddots$};
\draw (-0.0,-2.22) node {$x_j$};
\draw (-0.5,-2.6) node {$\iddots$};

\draw (2,-1.0) -- (1.6,-1.0) -- (1.6,-1.4) -- (0.8,-1.4);
\draw (0.2,-2.0) -- (0.2,-2.4) -- (-0.2,-2.4);
\draw (-0.8,-3) -- (-1.4,-3) -- (-1.4,0) -- (3.2,0) -- (3.2,-0.4) -- (2.6,-0.4);

\draw (3.4,-1.7) node {,};
\end{tikzpicture}
\hspace{12pt}
\begin{tikzpicture}
\draw (-2,-1.5) node {$\nu_i=$};

\draw (2.3,-0.6) node {$\iddots$};
\draw (1,-1.2) node {$z$};
\draw (1.37,-1.22) node {$x_i$};
\draw (1.8,-1.2) node {$y$};
\draw (0.5,-1.6) node {$\iddots$};
\draw (-0.0,-2.22) node {$x_j$};
\draw (-0.5,-2.6) node {$\iddots$};

\draw (2,-1.0) -- (1.15,-1.0) -- (1.15,-1.4) -- (0.8,-1.4);
\draw (0.2,-2.0) -- (0.2,-2.4) -- (-0.2,-2.4);
\draw (-0.8,-3) -- (-1.4,-3) -- (-1.4,0) -- (3.2,0) -- (3.2,-0.4) -- (2.6,-0.4);
\end{tikzpicture}\]
and so $y$ is an addable node of $\lambda$ but not of $\nu_i$ and $z$ is a removable node of $\nu_i$ but not of $\lambda$. All other addable and removable nodes of residue $j$ of $\lambda$ and $\nu_i$ are equal.
So (as $y$ and $z$ both have residue $j$) the $j$-signatures of $\lambda$ and $\nu_i$ are equal, apart for the position corresponding to $y$ and $z$ respectively. In this position the $j$-signature of $\lambda$ is $-$ and that of $\nu_i$ is $+$. From the definition of $x_j$ (it is the bottom normal $j$-node of $\lambda$) we can then (partly) reduce the $j$-signature of $\lambda$ and $\nu_i$ to
 \[\begin{array}{lllllll}
&&x_j&&&y\\
\lambda:&-\ldots-&+&+\ldots+&+&-&+\ldots+\\
\\
&&x_j&&&z\\
\nu_i:&-\ldots-&+&+\ldots+&+&+&+\ldots+.
\end{array}\]
So the reduced $j$-signature are given by
 \[\begin{array}{lllllll}
\lambda:&-\ldots-&+&+\ldots+&&&+\ldots+\\
\nu_i:&-\ldots-&+&+\ldots+&+&+&+\ldots+.
\end{array}\]
In particular $\epsilon_j(\nu_i)=\epsilon_j(\lambda)+2$.
\end{proof}

\begin{lemma}\label{l4}
Let $p=2$ and $\lambda\vdash n$ be 2-regular. Assume that $\epsilon_0(\lambda),\epsilon_1(\lambda)\geq 1$. Then
\[\dim\End_{\s_{n-2}}(e_0e_1D^\lambda)+\dim\End_{\s_{n-2}}(e_1e_0D^\lambda)\geq 2\epsilon_0(\lambda)+2\epsilon_1(\lambda).\]
\end{lemma}

\begin{proof}
It follows from Lemmas \ref{l23} and \ref{l4'} as if $\nu_0$ and $\nu_1$ are such that $\tilde{e}_0D^\lambda=D^{\nu_0}$ and $\tilde{e}_1D^\lambda=D^{\nu_1}$ then, for a certain $0\leq i\leq 1$ and for $j=1-i$,
\begin{align*}
&\dim\End_{\s_{n-2}}(e_0e_1D^\lambda)+\dim\End_{\s_{n-2}}(e_1e_0D^\lambda)\\
&\hspace{12pt}=\dim\End_{\s_{n-2}}(e_je_iD^\lambda)+\dim\End_{\s_{n-2}}(e_ie_jD^\lambda)\\
&\hspace{12pt}\geq\epsilon_i(\lambda)+\epsilon_j(\nu_i)-1+\epsilon_j(\lambda)+\epsilon_i(\nu_j)-1\\
&\hspace{12pt}\geq\epsilon_i(\lambda)+\epsilon_j(\lambda)+2-1+\epsilon_j(\lambda)+\epsilon_i(\lambda)-1\\
&\hspace{12pt}=2\epsilon_0(\lambda)+2\epsilon_1(\lambda).
\end{align*}
\end{proof}

\begin{lemma}\label{l60}
Let $M$ be an $\s_n$ module. Then
\[M\uparrow^{\s_{n+1}}\downarrow_{\s_n}\cong M\oplus M\downarrow_{\s_{n-1}}\uparrow^{\s_n}.\]
\end{lemma}

\begin{proof}
This follows from Lemma 1.3.8 and Theorem 1.3.10 of \cite{jk} and Mackey's theorem.
\end{proof}

\begin{lemma}\label{l21}
If $i\not=j$ then $e_if_j$ and $f_je_i$ are isomorphic functors.
\end{lemma}

\begin{proof}
Let $M$ be an $\s_n$-module corresponding to a single block $B$ with content $(b_0,\ldots,b_{p-1})$. Let $C$ be the block with content $(c_0,\ldots,c_{p-1})$, where $c_i=b_i-1$, $c_j=b_j+1$ and $c_k=b_k$ if $k\not=i,j$. As $i\not=j$ we have from Lemma \ref{l45} that the components of
\[M\uparrow^{\s_{n+1}}\downarrow_{\s_n}\cong \sum_{k,l}e_kf_lD^\lambda\]
and
\[M\downarrow_{\s_{n-1}}\uparrow^{\s_n}\cong \sum_{k,l}f_le_kD^\lambda\]
corresponding to block $C$ are $e_if_jM$ and $f_je_iM$ respectively. Since from Lemma \ref{l60}
\[M\uparrow^{\s_{n+1}}\downarrow_{\s_n}\cong M\oplus M\downarrow_{\s_{n-1}}\uparrow^{\s_n}\]
and $B\not=C$, so that the components of $M\downarrow_{\s_{n-1}}\uparrow^{\s_n}$ and $M\uparrow^{\s_{n+1}}\downarrow_{\s_n}$ in block $C$ are isomorphic (as $M$ corresponds to $B$), it follows that $e_if_jM\cong f_je_iM$.

The lemma holds as $e_if_j$ and $f_je_i$ respect direct sums.
\end{proof}

\begin{lemma}\label{l22}
Let $\lambda\vdash n$ be $p$-regular. For $i\not=j$ we have that
\[\dim\Hom_{\s_{n-2}}(e_je_iD^\lambda,e_ie_jD^\lambda)\geq\epsilon_i(\lambda)\epsilon_j(\lambda).\]
\end{lemma}

\begin{proof}
From Lemmas \hyperref[l39b]{\ref*{l39}\ref*{l39b}} and \ref{l48}
\begin{align*}
\dim\Hom_{\s_n}(f_ke_kD^\lambda,D^\lambda)&=\dim\Hom_{\s_n}(D^\lambda,f_ke_kD^\lambda)\\
&=\dim\End_{\s_n}(e_kD^\lambda)\\
&=\epsilon_k(\lambda).
\end{align*}
Since $D^\lambda$ is simple, using Lemmas \ref{l48} and \ref{l21} it then follows that
\begin{align*}
&\dim\Hom_{\s_{n-2}}(e_je_iD^\lambda,e_ie_jD^\lambda)\\
&\hspace{12pt}=\dim\Hom_{\s_{n-1}}(f_ie_je_iD^\lambda,e_jD^\lambda)\\
&\hspace{12pt}=\dim\Hom_{\s_{n-1}}(e_jf_ie_iD^\lambda,e_jD^\lambda)\\
&\hspace{12pt}=\dim\Hom_{\s_n}(f_ie_iD^\lambda,f_je_jD^\lambda)\\
&\hspace{12pt}\geq\dim\Hom_{\s_n}(f_ie_iD^\lambda,D^\lambda)\cdot\dim\Hom_{\s_n}(D^\lambda,f_je_jD^\lambda)\\
&\hspace{12pt}=\epsilon_i(\lambda)\epsilon_j(\lambda).
\end{align*}
\end{proof}

\begin{lemma}\label{l7}
Let $p=2$ and $\lambda\vdash n$ be 2-regular. If $\epsilon_0(\lambda),\epsilon_1(\lambda)\geq 1$ let $a=2\epsilon_0(\lambda)+2\epsilon_1(\lambda)+2\epsilon_0(\lambda)\epsilon_1(\lambda)$, else let $a=0$. Then
\[\dim\End_{\s_{n-2}}(D^\lambda\downarrow_{\s_{n-2}})\geq2\epsilon_0(\lambda)(\epsilon_0(\lambda)-1)+2\epsilon_1(\lambda)(\epsilon_1(\lambda)-1)+a.\]
\end{lemma}

\begin{proof}
From Lemma \ref{l45} we have that
\[D^\lambda\downarrow_{\s_{n-2}}=e_0e_0D^\lambda\oplus e_1e_1D^\lambda\oplus e_0e_1D^\lambda\oplus e_1e_0D^\lambda.\]
Further $e_0e_0D^\lambda$, $e_1e_1D^\lambda$ and $(e_0e_1D^\lambda\oplus e_1e_0D^\lambda)$ correspond to 3 distinct blocks, from the definition of $e_0$ and $e_1$.

So, from Lemmas \ref{l39} and \ref{l22},
\begin{align*}
&\dim\End_{\s_{n-2}}(D^\lambda\downarrow_{\s_{n-2}})\\
&\hspace{6pt}=\dim\End_{\s_{n-2}}(e_0e_0D^\lambda\oplus e_0e_1D^\lambda\oplus e_1e_0D^\lambda\oplus e_1e_1D^\lambda)\\
&\hspace{6pt}=\dim\End_{\s_{n-2}}(e_0e_0D^\lambda)\!+\!\dim\End_{\s_{n-2}}(e_1e_1D^\lambda)\\
&\hspace{24pt}+\!\dim\End_{\s_{n-2}}(e_0e_1D^\lambda\oplus e_1e_0D^\lambda)\\
&\hspace{6pt}=4\binom{\!\epsilon_0(\lambda)\!}{2}\!+\!4\binom{\!\epsilon_1(\lambda)\!}{2}\!+\!\dim\End_{\s_{n-2}}(e_0e_1D^\lambda)\!+\!\dim\End_{\s_{n-2}}(e_1e_0D^\lambda)\\
&\hspace{24pt}+\!\dim\Hom_{\s_{n-2}}(e_0e_1D^\lambda,e_1e_0D^\lambda)\!+\!\dim\Hom_{\s_{n-2}}(e_1e_0D^\lambda,e_0e_1D^\lambda)\\
&\hspace{6pt}\geq4\binom{\!\epsilon_0(\lambda)\!}{2}\!+\!4\binom{\!\epsilon_1(\lambda)\!}{2}\!+\!2\epsilon_0(\lambda)\epsilon_1(\lambda)\!+\!\dim\End_{\s_{n-2}}(e_0e_1D^\lambda)\\
&\hspace{24pt}+\!\dim\End_{\s_{n-2}}(e_1e_0D^\lambda).
\end{align*}
The lemma now follows from Lemma \ref{l4}.
\end{proof}

\begin{lemma}\label{l58}
If $M,N$ are $G$-modules then
\[\Hom_G(M,\End_F(N))\cong\Hom_G(N,N\otimes M^*).\]
\end{lemma}

\begin{proof}
We have that
\[\Hom_G(M,\End_F(N))\cong\Hom_G(M,N\otimes N^*)\cong \Hom_G(N,N\otimes M^*).\]
\end{proof}

\begin{lemma}\label{l45'}
If $n$ is an $\s_n$-module then
\[N\otimes M^{(n-1,1)}\cong \sum_{i,j}f_ie_j N.\]
\end{lemma}

\begin{proof}
From Lemma \ref{l45} it follows that
\[N\otimes M^{(n-1,1)}=N\otimes 1\uparrow^{\s_n}_{\s_{n-1}}\cong N\downarrow_{\s_{n-1}}\uparrow^{\s_n}\cong\sum_{i,j}f_ie_j N.\]
\end{proof}

\begin{lemma}\label{l51}
Let $n\geq 3$ with $p|n$ and $\lambda$ be $p$-regular. Then
\begin{align*}
&\min\{\max_{i:\epsilon_i(\lambda)\geq1}\hspace{-1pt}\{[\soc((f_i\tilde{e}_iD^\lambda)/D^\lambda):\hspace{-1pt}D^\lambda]\},\max_{i:\phi_i(\lambda)\geq1}\hspace{-1pt}\{[\soc((e_i\tilde{f}_iD^\lambda)/D^\lambda):\hspace{-1pt}D^\lambda]\}\hspace{-1pt}\}\\
&\hspace{36pt} +\epsilon_0(\lambda)+\ldots+\epsilon_{p-1}(\lambda)-1\\
&\hspace{12pt}\geq\dim\Hom_{\s_n}(S^{(n-1,1)},\End_F(D^\lambda))\\
\end{align*}
\end{lemma}

\begin{proof}
Since $p|n$ we have that
\[M^{(n-1,1)}=D^{(n)}|D^{(n-1,1)}|D^{(n)}\sim D^{(n)}|(S^{(n-1,1)})^*\]
from Lemma \ref{l1} and as $M^{(n-1,1)}$ is self-dual. Also, from Lemma \ref{l58},
\[\Hom_{\s_n}(S^{(n-1,1)},\End_F(D^\lambda))
\cong\Hom_{\s_n}(D^\lambda,D^\lambda\otimes (S^{(n-1,1)})^*).\]
Since $(S^{(n-1,1)})^*\cong M^{(n-1,1)}/D^{(n)}$, there exists $D\subseteq D^\lambda\otimes M^{(n-1,1)}$ with $D\cong D^\lambda$ for which $D^\lambda\otimes (S^{(n-1,1)})^*\cong(D^\lambda\otimes M^{(n-1,1)})/D$. We will now show that for an arbitrary $D\subseteq D^\lambda\otimes M^{(n-1,1)}$ with $D\cong D^\lambda$ we have
\begin{align*}
&\dim\Hom_{\s_n}(D^\lambda,(D^\lambda\otimes M^{(n-1,1)})/D)\\
&\hspace{12pt}\leq \epsilon_0(\lambda)+\ldots+\epsilon_{p-1}(\lambda)-1+\max_{i:\epsilon_i(\lambda)\geq1}\{[\soc((f_i\tilde{e}_iD^\lambda)/D^\lambda):D^\lambda]\}.
\end{align*}
Notice first that, from Lemma \ref{l45'},
\[D^\lambda\otimes M^{(n-1,1)}\cong\sum_{i,j}f_ie_jD^\lambda.\]
By definition of $e_i$ and $f_j$, the block component of $D^\lambda\otimes M^{(n-1,1)}$ is isomorphic to $\sum_i f_ie_iD^\lambda$. So, up to isomorphism, 
if $D\subseteq D^\lambda\otimes M^{(n-1,1)}$ with $D\cong D^\lambda$, then $D\subseteq \sum_if_ie_iD^\lambda$. Further
\[\dim\Hom_{\s_n}(D^\lambda,(D^\lambda\otimes M^{(n-1,1)})/D)=\dim\Hom_{\s_n}(D^\lambda,(\sum_if_ie_iD^\lambda)/D).\]
So we will consider $\sum_if_ie_iD^\lambda$ instead of $D^\lambda\otimes M^{(n-1,1)}$.

Let
\[e_iD^\lambda\sim D_{i,1}|\ldots|D_{i,h_i},\]
with $D_{i,k}$ simple. Then
\[\sum_if_ie_iD^\lambda\sim f_0D_{0,1}|\ldots|f_0D_{0,h_0}|\ldots|f_{p-1}D_{p-1,1}|\ldots|f_{p-1}D_{p-1,h_{p-1}}.\]
Notice that, from Lemmas \hyperref[l40a]{\ref*{l40}\ref*{l40a}} and \ref{l47}, we have that $D^\lambda\subseteq f_iD_{i,k}$ if and only if $D_{i,k}\cong \tilde{e}_iD^\lambda$. By definition and from the previous part of the proof we have that $D\cong D^\lambda$ and $D\subseteq \sum_if_ie_iD^\lambda$. So there exist $\overline{i}$ and $\overline{k}$ with $1\leq \overline{k}\leq h_{\overline{i}}$ and with $D_{\overline{i},\overline{k}}\cong \tilde{e}_{\overline{i}}D^\lambda$ such that
\[(\sum_if_ie_iD^\lambda)\hspace{-1pt}/\hspace{-1pt}D\hspace{-1pt}\sim\hspace{-1pt} f_0D_{0,1}|\ldots|f_{\overline{i}}D_{\overline{i},\overline{k}-1}|(\hspace{-1pt}(f_{\overline{i}}D_{\overline{i},\overline{k}})\hspace{-1pt}/\hspace{-1pt}D)|f_{\overline{i}}D_{\overline{i},\overline{k}+1}|\ldots|f_{p-1}D_{p-1,h_{p-1}}.\]
In particular, from Lemma \hyperref[l39b]{\ref*{l39}\ref*{l39b}} and by definition of $D_{i,k}$,
\begin{align*}
&\dim\Hom_{\s_n}(D^\lambda,(\sum_if_ie_iD^\lambda)/D)\\
&\hspace{12pt}=[\soc((\sum_if_ie_iD^\lambda)/D):D^\lambda]\\
&\hspace{12pt}\leq[\soc((f_{\overline{i}}D_{\overline{i},\overline{k}})/D):D^\lambda]+\sum_{(i,k)\not=(\overline{i},\overline{k})}[\soc(f_iD_{i,k}):D^\lambda]\\
&\hspace{12pt}\leq[\soc((f_{\overline{i}}D_{\overline{i},\overline{k}})/D^\lambda):D^\lambda]+|\{(i,k)\not=(\overline{i},\overline{k}):\tilde{e}_iD^\lambda\cong D_{i,k}\}|\\
&\hspace{12pt}=[\soc((f_{\overline{i}}D_{\overline{i},\overline{k}})/D^\lambda):D^\lambda]+|\{(i,k):\tilde{e}_iD^\lambda\cong D_{i,k}\}|-1\\
&\hspace{12pt}=[\soc((f_{\overline{i}}D_{\overline{i},\overline{k}})/D^\lambda):D^\lambda]+\epsilon_0(\lambda)+\ldots+\epsilon_{p-1}(\lambda)-1.
\end{align*}
As $e_{\overline{i}}D^\lambda\not=0$ this gives
\begin{align*}
&\dim\Hom_{\s_n}(S^{(n-1,1)},\End_F(D^\lambda))\\
&\hspace{12pt}=\dim\Hom_{\s_n}(D^\lambda,(\sum_if_ie_iD^\lambda)/D)\\
&\hspace{12pt}\leq \epsilon_0(\lambda)+\ldots+\epsilon_{p-1}(\lambda)-1+\max_{i:\epsilon_i(\lambda)\geq1}\{[\soc((f_i\tilde{e}_iD^\lambda)/D^\lambda):D^\lambda]\}.
\end{align*}

From Lemma \ref{l60} we have that
\[D^\lambda\uparrow^{\s_{n+1}}\downarrow_{\s_n}\cong D^\lambda\oplus D^\lambda\downarrow_{\s_{n-1}}\uparrow^{\s_n}.\]
So there exists $D\subseteq D^\lambda\uparrow^{\s_{n+1}}\downarrow_{\s_n}$ with $D\cong D^\lambda$ and
\[(D^\lambda\uparrow^{\s_{n+1}}\downarrow_{\s_n})/D\cong D^\lambda\oplus (D^\lambda\otimes (S^{(n-1,1)})^*).\]
In particular, with a similar proof, we obtain
\begin{align*}
&\dim\Hom_{\s_n}(D^\lambda,D^\lambda\otimes (S^{(n-1,1)})^*)\\
&\hspace{12pt}=\dim\Hom_{\s_n}(D^\lambda,(D^\lambda\uparrow^{\s_{n+1}}\downarrow_{\s_n})/D)-\dim\End_{\s_n}(D^\lambda)\\
&\hspace{12pt}\leq \phi_0(\lambda)+\ldots+\phi_{p-1}(\lambda)-1+\max_{i:\phi_i(\lambda)\geq1}\{[\soc((e_i\tilde{f}_iD^\lambda)/D^\lambda):D^\lambda]\}-1.
\end{align*}

As $\epsilon_0(\lambda)+\ldots+\epsilon_{p-1}(\lambda)+1=\phi_0(\lambda)+\ldots+\phi_{p-1}(\lambda)$ from Lemma \ref{l52}, the lemma follows.
\end{proof}

\begin{lemma}\label{l8}
Let $n\geq 3$ with $p|n$ and $\lambda$ be $p$-regular. Then
\begin{align*}
&\min\{\max_{i:\epsilon_i(\lambda)\geq1}\{\phi_i(\lambda)-\delta_{\phi_i(\lambda)>1}\},\max_{i:\phi_i(\lambda)\geq1}\{\epsilon_i(\lambda)-\delta_{\epsilon_i(\lambda)>1}\}\}\\
&\hspace{36pt} +\epsilon_0(\lambda)+\ldots+\epsilon_{p-1}(\lambda)-1\\
&\hspace{12pt}\geq\dim\Hom_{\s_n}(S^{(n-1,1)},\End_F(D^\lambda))\\
\end{align*}
In particular if $\epsilon_0(\lambda)+\ldots+\epsilon_{p-1}(\lambda)\geq 2$ then
\[\dim\Hom_{\s_n}(S^{(n-1,1)},\End_F(D^\lambda))\leq 2\epsilon_0(\lambda)+\ldots+2\epsilon_{p-1}(\lambda)-2.\]
\end{lemma}

\begin{proof}
If $\epsilon_0(\lambda)+\ldots+\epsilon_{p-1}(\lambda)\geq 2$ then
\[\epsilon_0(\lambda)+\ldots+\epsilon_{p-1}(\lambda)-1\geq \max_i\{\epsilon_i(\lambda)-\delta_{\epsilon_i(\lambda)>1}\},\]
since if $\epsilon_i(\lambda)=\epsilon_0(\lambda)+\ldots+\epsilon_{p-1}(\lambda)$ then $\epsilon_i(\lambda)>1$. So it is enough to prove the first part of the lemma.

We will now prove that if $\epsilon_i(\lambda)\geq 1$ then $\soc((f_i\tilde{e}_iD^\lambda)/D^\lambda)$ contains at most $\phi_i(\lambda)-\delta_{\phi_i(\lambda)>1}$ copies of $D^\lambda$. Since it can be proved similarly that  if $\phi_i(\lambda)\geq 1$ then $\soc((e_i\tilde{f}_iD^\lambda)/D^\lambda)$ contains at most $\epsilon_i(\lambda)-\delta_{\epsilon_i(\lambda)>1}$ copies of $D^\lambda$, this will complete the proof of the lemma.

Assume that $\epsilon_i(\lambda)\geq 1$. Then $\tilde{e}_iD^\lambda\not=0$ from definition of $\tilde{e}_i$. Let $\nu$ with $D^\nu=\tilde{e}_iD^\lambda$. Then $\tilde{f}_iD^\nu=D^\lambda$ and $\phi_i(\nu)=\phi_i(\lambda)+1\geq 1$ from Lemma \ref{l47}. So we have to prove that $\soc((f_iD^\nu)/D^\lambda)$ contains at most $\phi_i(\lambda)-\delta_{\phi_i(\lambda)>1}$ copies of $D^\lambda$ (notice that $D^\lambda\cong \soc(f_iD^\nu)$ from Lemma \hyperref[l40a]{\ref*{l40}\ref*{l40a}}). From Lemma \hyperref[l40b]{\ref*{l40}\ref*{l40b}} we have
\[[\soc((f_iD^\nu)/D^\lambda)\hspace{-1pt}:\hspace{-2pt}D^\lambda]\!\leq\![(f_iD^\nu)/D^\lambda\hspace{-1pt}:\hspace{-2pt}D^\lambda]\!=\![f_iD^\nu\hspace{-2pt}:\hspace{-2pt}D^\lambda]-1\!=\!\phi_i(\nu)-1\!=\!\phi_i(\lambda).\]
So we can assume that $\phi_i(\lambda)>1$. In this case $(f_iD^\nu)/D^\lambda\not=0$ and so
\[0\not=\hd((f_iD^\nu)/D^\lambda)\subseteq \hd(f_iD^\nu)\cong D^\lambda\]
by Lemma \hyperref[l40a]{\ref*{l40}\ref*{l40a}}. In particular $\hd((f_iD^\nu)/D^\lambda)\cong D^\lambda$ and then
\[(f_iD^\nu)/D^\lambda\sim M|\overbrace{D^\lambda}^{\hd((f_iD^\nu)/D^\lambda)}\]
for a certain module $M\subseteq (f_iD^\nu)/D^\lambda$. As
\[[M:D^\lambda]=[(f_iD^\nu)/D^\lambda:D^\lambda]-1=\phi_i(\lambda)-1>0\]
we have that $M\not=0$. So $\soc((f_iD^\nu)/D^\lambda)\subseteq M$, as $\hd((f_iD^\nu)/D^\lambda)$ is simple and $(f_iD^\nu)/D^\lambda\sim M|\hd((f_iD^\nu)/D^\lambda)$. In particular
\[[\soc((f_iD^\nu)/D^\lambda):D^\lambda]\leq[M:D^\lambda]=\phi_i(\lambda)-1.\]
\end{proof}

We will also need the following lemma which compares the dimensions of $\End_{\s_{n-2}}(M\downarrow_{\s_{n-2}})$ and of $\End_{\s_{n-2,2}}(M\downarrow_{\s_{n-2,2}})$ for any $\s_n$-module $M$.

\begin{lemma}\label{l54}
If $p=2$ and $M$ is an $\s_n$-module then
\[\dim\End_{\s_{n-2}}(M\downarrow_{\s_{n-2}})\leq 2\dim\End_{\s_{n-2,2}}(M\downarrow_{\s_{n-2,2}}).\]
\end{lemma}

\begin{proof}
As $p=2$, so that $M^{(2)}\sim D^{(2)}|D^{(2)}$ (as $D^{(2)}$ is the only simple module of $\s_2$ in characteristic 2), we have that
\[1\hspace{-1pt}\uparrow_{\s_{n-2}}^{\s_{n-2,2}}\cong 1_{\s_{n-2}}\otimes M^{(2)}\!\sim\! D^{(n-2)}\otimes (D^{(2)}|D^{(2)})\!\sim\! (D^{(n-2)}\otimes D^{(2)})|(D^{(n-2)}\otimes D^{(2)}).\]
So, from Lemma \ref{l34},
\begin{align*}
&\dim\End_{\s_{n-2}}(M\downarrow_{\s_{n-2}})\\
&\hspace{12pt}=\dim\Hom_{\s_{n-2}}(D^{(n-2)},\End_F(M)\downarrow_{\s_{n-2}})\\
&\hspace{12pt}=\dim\Hom_{\s_{n-2,2}}(D^{(n-2)}\uparrow^{\s_{n-2,2}},\End_F(M)\downarrow_{\s_{n-2,2}})\\
&\hspace{12pt}=\dim\Hom_{\s_{n-2,2}}((D^{(n-2)}\otimes D^{(2)})|(D^{(n-2)}\otimes D^{(2)}),\End_F(M)\downarrow_{\s_{n-2,2}})\\
&\hspace{12pt}\leq2\dim\Hom_{\s_{n-2,2}}(D^{(n-2)}\otimes D^{(2)},\End_F(M)\downarrow_{\s_{n-2,2}})\\
&\hspace{12pt}=2\dim\End_{\s_{n-2,2}}(M\downarrow_{\s_{n-2}}).
\end{align*}
\end{proof}

The next lemma considers the structures of $\overline{e}_i^2D^\lambda$ when $\epsilon_i(\lambda)=2$.

\begin{lemma}\label{l31}
Let $p=2$ and $\lambda$ be a 2-regular partition with $\epsilon_i(\lambda)=2$. If $D^\nu\cong\tilde{e}_i^2D^\lambda$, then $\overline{e}_i^2D^\lambda=(D^\nu\otimes D^{(2)})|(D^\nu\otimes D^{(2)})$.
\end{lemma}

\begin{proof}
As $D^{(2)}$ is the only simple module of $\s_2$ in characteristic 2 and as $(\overline{e}_i^2D^\lambda)\downarrow_{\s_{n-2}}=e_i^2D^\lambda=D^\nu\oplus D^\nu$ from Lemma \ref{l39}, we have that $\overline{e}_i^2D^\lambda\sim(D^\nu\otimes D^{(2)})|(D^\nu\otimes D^{(2)})$.

Considering the block decomposition of $D^\lambda\downarrow_{\s_{n-2,2}}$ and of $(D^\nu\otimes D^{(2)})\uparrow^{\s_n}$ and from the definitions of $\overline{e}_i^2D^\lambda$ and of $f_i^{(2)}D^\nu$ (restrictions of $D^\lambda\downarrow_{\s_{n-2,2}}$ and $(D^\nu\otimes D^{(2)})\uparrow^{\s_n}$ to certain blocks), we have that
\begin{align*}
\dim\Hom_{\s_{n-2,2}}(D^\nu\otimes D^{(2)},\overline{e}_i^2D^\lambda)&=\dim\Hom_{\s_{n-2,2}}(D^\nu\otimes D^{(2)},D^\lambda\downarrow_{\s_{n-2,2}})\\
&=\dim\Hom_{\s_n}((D^\nu\otimes D^{(2)})\uparrow^{\s_n},D^\lambda)\\
&=\dim\Hom_{\s_n}(f_i^{(2)}D^\nu,D^\lambda),\\
&=1
\end{align*}
from which the lemma follows.
\end{proof}

\begin{rem}
Notice that Lemmas \ref{l34}, \ref{ll1}, \ref{l9'}, \ref{l9}, \ref{l3}, \ref{l23}, \ref{l21}, \ref{l22}, \ref{l58}, \ref{l51} and \ref{l8} hold in arbitrary characteristic.
\end{rem}

\section{Partitions with at least 3 normal nodes}\label{s2}

We now consider the structure of $\End_F(D^\lambda)$ for partitions with at least 3 normal nodes. We first show that $D^{(n-1,1)}$ is contained in $\End_F(D^\lambda)$. This will be used when considering tensor products with $D^{(m+1,m-1)}$ for $m$ even.

\begin{lemma}\label{l18}
If $p=2$, $\lambda\vdash n$ with $n\geq 4$ even is 2-regular and $\epsilon_0(\lambda)+\epsilon_1(\lambda)\geq 3$ then $D^{(n-1,1)}\subseteq \End_F(D^\lambda)$.
\end{lemma}

\begin{proof}
As $n$ is even $M^{(n-1,1)}=D^{(n)}|D^{(n-1,1)}|D^{(n)}$. From Lemmas \ref{l53} and \ref{l34} it then follows that
\begin{align*}
&\dim\Hom_{\s_n}(D^{(n-1,1)},\End_F(D^\lambda))\\
&\hspace{12pt}\geq \dim\Hom_{\s_n}(M^{(n-1,1)},\End_F(D^\lambda))-2\dim\Hom_{\s_n}(D^{(n)},\End_F(D^\lambda))\\
&\hspace{12pt}=\dim\End_{\s_{n-1}}(D^\lambda\downarrow_{\s_{n-1}})-2\\
&\hspace{12pt}=\epsilon_0(\lambda)+\epsilon_1(\lambda)-2\\
&\hspace{12pt}\geq 1.
\end{align*}
\end{proof}

We will now show that the assumptions of Lemmas \ref{l17} and \ref{l36} holds when $\lambda$ has at least 3 normal nodes.

\begin{lemma}\label{l15}
Let $p=2$ and $\lambda\vdash n$ with $n\geq 4$ even be 2-regular and assume that $\epsilon_0(\lambda)+\epsilon_1(\lambda)\geq 3$. Then
\begin{align*}
&\dim\End_{\s_{n-1}}(D^\lambda\downarrow_{\s_{n-1}})+\dim\Hom_{\s_n}(S^{(n-1,1)},\End_F(D^\lambda))+1\\
&\hspace{12pt}<\dim\End_{\s_{n-2}}(D^\lambda\downarrow_{\s_{n-2}}).
\end{align*}
\end{lemma}

\begin{proof}
Let $a=2\epsilon_0(\lambda)+2\epsilon_1(\lambda)+2\epsilon_0(\lambda)\epsilon_1(\lambda)$ if $\epsilon_0(\lambda),\epsilon_1(\lambda)\geq 1$ and $a=0$ otherwise. From Lemmas \ref{l53}, \ref{l7} and \ref{l8} it is enough to prove that
\[3\epsilon_0(\lambda)+3\epsilon_1(\lambda)-1<2\epsilon_0(\lambda)(\epsilon_0(\lambda)-1)+2\epsilon_1(\lambda)(\epsilon_1(\lambda)-1)+a\]
and so it is also enough to prove that
\[\epsilon_0(\lambda)(2\epsilon_0(\lambda)-5)+\epsilon_1(\lambda)(2\epsilon_1(\lambda)-5)+a\geq 0.\]

Since $x(2x-5)\geq 0$ for $x=0$ or $x\geq 3$ and since $a\geq 0$, we only need to check the lemma when $1\leq \epsilon_i(\lambda)\leq 2$ for some $0\leq i\leq 1$. Let $j=1-i$. If $\epsilon_i(\lambda)=1$ then $\epsilon_j(\lambda)\geq 2$ and so
\[\epsilon_0(\lambda)(2\epsilon_0(\lambda)-5)+\epsilon_1(\lambda)(2\epsilon_1(\lambda)-5)+a=\epsilon_j(\lambda)(2\epsilon_j(\lambda)-5)-3+4\epsilon_j(\lambda)+2\geq 0.\]
If instead $\epsilon_i(\lambda)=2$ then $\epsilon_j(\lambda)\geq 1$ and so
\[\epsilon_0(\lambda)(2\epsilon_0(\lambda)-5)+\epsilon_1(\lambda)(2\epsilon_1(\lambda)-5)+a=\epsilon_j(\lambda)(2\epsilon_j(\lambda)-5)-2+6\epsilon_j(\lambda)+4\geq 0.\]
\end{proof}

\begin{lemma}\label{l41}
Let $p=2$ and $\lambda\vdash n$ with $n\geq 4$ even be 2-regular and assume that $\epsilon_0(\lambda)+\epsilon_1(\lambda)\geq 3$. Then
\[\dim\End_{\s_{n-2,2}}(D^\lambda\downarrow_{\s_{n-2,2}})>\dim\Hom_{\s_n}(S^{(n-1,1)},\End_F(D^\lambda))+1.\]
\end{lemma}

\begin{proof}
From Lemma \ref{l54} we have that
\[\dim\End_{\s_{n-2}}(D^\lambda\downarrow_{\s_{n-2}})\leq 2\dim\End_{\s_{n-2,2}}(D^\lambda\downarrow_{\s_{n-2,2}}).\]
So it is enough to prove that
\[\dim\End_{\s_{n-2}}(D^\lambda\downarrow_{\s_{n-2}})>2\dim\Hom_{\s_n}(S^{(n-1,1)},\End_F(D^\lambda))+2.\]

Let $a=2\epsilon_0(\lambda)+2\epsilon_1(\lambda)+2\epsilon_0(\lambda)\epsilon_1(\lambda)$ if $\epsilon_0(\lambda),\epsilon_1(\lambda)\geq 1$ or $a=0$ otherwise. From Lemmas \ref{l7} and \ref{l8} it is enough to prove that
\[2\epsilon_0(\lambda)(\epsilon_0(\lambda)-1)+2\epsilon_1(\lambda)(\epsilon_1(\lambda)-1)+a>4\epsilon_0(\lambda)+4\epsilon_1(\lambda)-2.\]
It is then also enough to prove that
\[2\epsilon_0(\lambda)(\epsilon_0(\lambda)-3)+2\epsilon_1(\lambda)(\epsilon_1(\lambda)-3)+a\geq 0.\]
As $2x(x-3)\geq 0$ for $x=0$ or $x\geq 3$ and as $a\geq 0$, we still only need to prove the lemma if $1\leq\epsilon_i(\lambda)\leq 2$ for some $0\leq i\leq 1$. Let $j=1-i$.  If $\epsilon_i(\lambda)=1$ then $\epsilon_j(\lambda)\geq 2$ and so
\[2\epsilon_0(\lambda)(\epsilon_0(\lambda)-3)+2\epsilon_1(\lambda)(\epsilon_1(\lambda)-3)+a=2\epsilon_j(\lambda)(\epsilon_j(\lambda)-3)-4+4\epsilon_j(\lambda)+2\geq 0.\]
If $\epsilon_i(\lambda)=2$ then $\epsilon_j(\lambda)\geq 1$ and so
\[2\epsilon_0(\lambda)(\epsilon_0(\lambda)-3)+2\epsilon_1(\lambda)(\epsilon_1(\lambda)-3)+a=2\epsilon_j(\lambda)(\epsilon_j(\lambda)-3)-4+6\epsilon_j(\lambda)+4\geq 0.\]
\end{proof}

\section{Partitions with 2 normal nodes}\label{s3}

In this section we will consider partitions with 2 normal nodes. Proofs or results will be more complicated than in the previous case as we have to explicitly consider the structure of such partitions. We first start by studying the structure of partitions with 2 normal nodes.

\begin{lemma}\label{l42}
Let $p=2$ and $\lambda\vdash n$ be 2-regular with $\epsilon_0(\lambda)+\epsilon_1(\lambda)=2$. For $1\leq k\leq h(\lambda)$ let $a_k$ be the residue of the removable node of the $k$-th row of $\lambda$. Further let $1<b_1<\ldots<b_t\leq h(\lambda)$ be the set of indexes $k$ for which $a_k=a_{k-1}$. Then the normal nodes of $\lambda$ are on rows 1 and $b_1$, while the conormal nodes of $\lambda$ are on rows $b_t-1$, $h(\lambda)$ and $h(\lambda)+1$. Further $a_{b_k}\not=a_{b_k-1}$ for $1<k\leq t$.
\end{lemma}

\begin{proof}
Notice that $a_k$ is defined for each $1\leq k\leq h(\lambda)$, since $\lambda$ is 2-regular. Further $\lambda$ has an addable node on row $k$ for each $1\leq k\leq h(\lambda)+1$.

Notice first that if the removable node on row $k$ is normal then $k=1$ or $k=b_s$ for some $s$ (as otherwise the addable node on row $k-1$ has residue $1-a_{k-1}=a_k$). From $\lambda$ having 2 normal nodes it then follows that $t\geq 1$. The removable node on the first row is normal (this is always the case). Also, from definition of $a_k$, we have that the residue of the addable node on row $k\leq h(\lambda)$ is given by $1-a_k$. For $1\leq k<b_1-1$ we have $1-a_k=a_{k+1}$ by definition of $b_1$. As $1-a_{b_1-1}\not=a_{b_1-1}=a_{b_1}$, it then follows that the removable node on row $b_1$ is also normal (if the addable node on row $k$, with $1\leq k<b_1$, has residue $a_{b_1}$ then $k<b_1-1$ and the removable node on row $k+1$ also has residue $a_{b_1}$). As $\lambda$ has 2 normal nodes and, by definition, $b_1\geq 2$, there are no other normal nodes of $\lambda$.

The addable nodes on rows $h(\lambda)$ and $h(\lambda)+1$ are conormal (they are always conormal). Further for $b_t+1\leq k\leq h(\lambda)$ we have, by definition of $b_t$, that $1-a_k=a_{k-1}$. Also $a_{b_t}=a_{b_t-1}\not=1-a_{b_t-1}$ and $1-a_{b_t-1}$ is the residue of the addable node on row $b_t-1$. It then follows that the addable node on row $b_t-1$ is also conormal (notice that $b_t-1\geq b_1-1\geq 1$). As $\lambda$ has 2 normal nodes and so 3 conormal (Lemma \ref{l52}), these are all the conormal nodes of $\lambda$.

We will now show that $a_{b_k}\not=a_{b_{k-1}}$ for $1<k\leq t$. To do this, let $k\geq 2$ minimal such that $a_{b_k}=a_{b_{k-1}}$ (if such an $k$ exists). Again notice that the addable node on row $s$ has residue $1-a_s$ for $s\leq h(\lambda)$. From the minimality of $k$ it follows that, for $1\leq s\leq b_k-2$ with $s\not=b_{k-1}-1$,
\[1-a_s=\left\{\begin{array}{ll}
a_{s+1},&s\not\in\{b_1-1,\ldots,b_{k-2}-1\},\\
a_{b_{r+1}},&s=b_r-1.
\end{array}\right.\]
Also $f:\{1,\ldots,b_k-2\}\setminus\{b_{k-1}-1\}\rightarrow\{1,\ldots,b_k-1\}$ given through
\[f(s)=\left\{\begin{array}{ll}
s+1,&s\not\in\{b_1-1,\ldots,b_{k-2}-1\},\\
b_{r+1},&s=b_r-1
\end{array}\right.\]
is injective and satisfies $f(s)>s$ for each $s$.

By assumption on $k$ we have that $a_{b_k}=a_{b_{k-1}}=a_{b_k-1}=a_{b_{k-1}-1}$ (so that the addable nodes on rows $b_k-1$ and $b_{k-1}-1$ do not have residue $a_{b_k}$). It then follows that the removable node on row $b_k$ is normal. As $k\geq 2$ this would then mean that $\lambda$ has at least 3 normal nodes, which contradicts the assumptions. So $a_{b_k}\not=a_{b_{k-1}}$ for $2\leq k\leq t$.
\end{proof}

\begin{lemma}\label{l24}
Let $p=2$ and $\lambda\vdash n$ be 2-regular with $n$ even. Then we cannot have that $\epsilon_i(\lambda)=2$, $\epsilon_j(\lambda)=0$, $\phi_i(\lambda)=0$ and $\phi_j(\lambda)=3$ for some $0\leq i\leq 1$ and $j=1-i$.
\end{lemma}

\begin{proof}
Assume that $\epsilon_i(\lambda)=2$, $\epsilon_j(\lambda)=0$, $\phi_i(\lambda)=0$ and $\phi_j(\lambda)=3$. For $1\leq k\leq h(\lambda)$ let $a_k$ be the residue of the removable node on the $k$-th row of $\lambda$ (on each row of $\lambda$ there are both a removable and an addable node, as $\lambda$ is 2-regular). Also let $1<b_1<\ldots<b_t\leq h(\lambda)$ be the set of indexes $k$ for which $a_k=a_{k-1}$. From Lemma \ref{l42} we have that the removable nodes on rows 1 and $b_1$ are normal. As $\epsilon_j(\lambda)=0$ we have that $a_1,a_{b_1}=i$.

Further, again from Lemma \ref{l42}, the conormal nodes of $\lambda$ are the addable nodes on rows $b_t-1$, $h(\lambda)$ and $h(\lambda)+1$. As $\phi_i(\lambda)=0$ it follows that $a_{b_t-1},a_{h(\lambda)}=i$ and $h(\lambda)\equiv 1-(h(\lambda)+1)\equiv j\Md 2$.

We have that $(a_1,\ldots,a_{b_1-1})=(i,j,\ldots,i)$, $(a_{b_t},\ldots,a_{h(\lambda)})=(i,j,\ldots,i)$ and $(a_{b_{s-1}},\ldots,a_{b_s-1})=(a_{b_{s-1}},1-a_{b_{s-1}},\ldots,a_{b_{s-1}},1-a_{b_{s-1}})$ for $2\leq s\leq t$ from definition of $b_s$ and Lemma \ref{l42}. So
\begin{align*}
&\lambda_1\equiv\lambda_2\equiv\ldots\equiv\lambda_{b_1-1}\Md 2,\\
&\lambda_{b_{s-1}}\equiv\lambda_{b_{s-1}+1}\equiv\ldots\equiv\lambda_{b_s-1}\Md 2,&2\leq s\leq t,\\
&\lambda_{b_t}\equiv\lambda_{b_t+1}\equiv\ldots\equiv\lambda_{h(\lambda)}\Md 2.
\end{align*}
Further $b_1-1$ and $h(\lambda)-b_t+1$ are odd while $b_s-b_{s-1}$ is even for $2\leq s\leq t$. In particular $h(\lambda)$ is even and so $j=0$ and $i=1$. From $a_1=i=1$ it follows that $\lambda_1$ is even, while from $h(\lambda)$ even and $a_{h(\lambda)}=i=1$ it follows that $\lambda_{h(\lambda)}$ is odd. So
\[n\equiv \lambda_1\cdot(b_1-1)+\lambda_{h(\lambda)}\cdot(h(\lambda)-b_t+1)+\sum_{s=2}^t\lambda_{b_{s-1}}\cdot(b_s-b_{s-1})\equiv 1\Md 2,\]
contradicting $n$ being even.
\end{proof}

\begin{lemma}\label{l20}
Let $p=2$ and $\lambda\vdash n$ with $n\geq 4$ even be 2-regular. Assume that $\epsilon_0(\lambda)+\epsilon_1(\lambda)=2$. Then $D^{(n-1,1)}\subseteq\End_F(D^\lambda)$.
\end{lemma}

\begin{proof}
From Lemma \ref{l1} we have that $M^{(n-1,1)}=D^{(n)}|D^{(n-1,1)}|D^{(n)}$, so that
\[D^\lambda\otimes M^{(n-1,1)}\sim D^\lambda|(D^\lambda\otimes D^{(n-1,1)})|D^\lambda.\]

From Lemma \ref{l24} (and Lemma \ref{l52}) there exists $i$ with $\epsilon_i(\lambda)\not=0$ and $\phi_i(\lambda)\not=0$. Let $\nu$ with $\tilde{e}_iD^\lambda=D^\nu$. Then $\phi_i(\nu)=\phi_i(\lambda)+1\geq 2$ from Lemma \ref{l47}. Also, from Lemma \hyperref[l39a]{\ref*{l39}\ref*{l39a}},
\[f_i(D^\nu)=f_i\tilde{e}_iD^\lambda\cong f_i\soc(e_iD^\lambda)\subseteq f_ie_iD^\lambda.\]
From Lemma \ref{l40} it then follows that
\[f_i\tilde{e}_iD^\lambda\sim \underbrace{\overbrace{D^\lambda}^{\soc(f_i\tilde{e}_iD^\lambda)}|\ldots|\overbrace{D^\lambda}^{\hd(f_i\tilde{e}_iD^\lambda)}}_{f_i\tilde{e}_iD^\lambda}|\ldots\]
and so in particular
\[f_ie_iD^\lambda\sim \rlap{$\hspace{3pt}\overbrace{\phantom{D^\lambda}}^{\not\subseteq \hd(f_ie_iD^\lambda)}$}\underbrace{D^\lambda}_{\subseteq\soc(f_ie_iD^\lambda)}|\ldots.\]
Similarly there exists $D\cong D^\lambda$ contained in the head of $f_ie_iD^\lambda$ which is not contained in the socle of $f_ie_iD^\lambda$. So
\[f_ie_iD^\lambda\sim\rlap{$\hspace{3pt}\overbrace{\phantom{D^\lambda}}^{\not\subseteq \hd(f_ie_iD^\lambda)}$}\underbrace{D^\lambda}_{\subseteq\soc(f_ie_iD^\lambda)}|\ldots|\rlap{$\hspace{3pt}\overbrace{\phantom{D^\lambda}}^{\subseteq \hd(f_ie_iD^\lambda)}$}\underbrace{D^\lambda}_{\not\subseteq\soc(f_ie_iD^\lambda)}.\]
From Lemma \ref{l45'}
\[D^\lambda\otimes M^{(n-1,1)}\cong f_0e_0D^\lambda\oplus f_1e_1D^\lambda\oplus f_0e_1D^\lambda\oplus f_1e_0D^\lambda.\]
Also, from Lemma \ref{l53},
\[\dim\Hom_{\s_n}(D^\lambda,D^\lambda\downarrow_{\s_{n-1}}\uparrow^{\s_n})=\dim\End_{\s_{n-1}}(D^\lambda\downarrow_{\s_{n-1}})=\epsilon_0(\lambda)+\epsilon_1(\lambda)=2.\]
As $D^\lambda$ is self-dual we then have that the head and socle of $D^\lambda\otimes M^{(n-1,1)}$ both contain 2 copies of $D^\lambda$. From the previous computations there exists $D,D'\cong D^\lambda$ with $D\subseteq\hd(D^\lambda\otimes M^{(n-1,1)})$ but $D\not\subseteq\soc(D^\lambda\otimes M^{(n-1,1)})$ and similarly $D'\subseteq\soc(D^\lambda\otimes M^{(n-1,1)})$ but $D'\not\subseteq\hd(D^\lambda\otimes M^{(n-1,1)})$. So
\[D^\lambda|(D^\lambda\otimes D^{(n-1,1)})|D^\lambda\sim D^\lambda\otimes M^{(n-1,1)}\sim (D^\lambda\oplus D^\lambda)|\ldots|(D^\lambda\oplus D^\lambda)\]
or
\[D^\lambda|(D^\lambda\otimes D^{(n-1,1)})|D^\lambda\sim D^\lambda\otimes M^{(n-1,1)}\sim D^\lambda\oplus (D^\lambda|\ldots|D^\lambda).\]
In each of these cases we can easily deduce that $D^\lambda\otimes D^{(n-1,1)}$ contains a copy of $D^\lambda$ in its socle or in its head. Since $D^\lambda\otimes D^{(n-1,1)}$ is self-dual, it follows that $D^\lambda\subseteq\soc(D^\lambda\otimes D^{(n-1,1)})$. Then, from Lemma \ref{l58},
\[\dim\Hom_{\s_n}(D^{(n-1,1)},\End_F(D^\lambda))=\dim\Hom_{\s_n}(D^\lambda,D^\lambda\otimes D^{(n-1,1)})\geq 1,\]
and so the lemma holds.
\end{proof}

\begin{lemma}\label{l44}
Let $p=2$ and $\lambda\vdash n$ with $n$ even be 2-regular with $\epsilon_i(\lambda)=2$ and $\epsilon_j(\lambda)=0$ for some $0\leq i\leq 1$ and for $j=1-i$. Then $\dim\End_{\s_{n-2}}(e_je_iD^\lambda)\geq 2$.
\end{lemma}

\begin{proof}
Notice that $h(\lambda)\geq 2$, as $\lambda$ has 2 normal nodes.

First assume that $\lambda_1\equiv\lambda_2\Md 2$ and let $\nu:=(\lambda_1-1,\lambda_2,\lambda_3,\ldots)$. Then $\lambda_1-\lambda_2\geq 2$ and so $\nu=\lambda\setminus (1,\lambda_1)$ is 2-regular. As $(1,\lambda_1)$ is a normal node of $\lambda$, we have that $D^\nu$ is a composition factor of $e_iD^\lambda=D^\lambda\downarrow_{\s_{n-1}}$ (Lemma \ref{l56}). As
\[
\begin{tikzpicture}
\draw (0.6,-0.7) node {$\lambda=$};

\draw (3,-0.2) node {j};
\draw (3.3,-0.2) node {i};
\draw (2.4,-0.6) node {j};

\draw (1.9,-1) node {$\iddots$};

\draw (2.25,-0.8)--(2.55,-0.8)--(2.55,-0.4)--(3.45,-0.4)--(3.45,0)--(1.2,0)--(1.2,-1.4)--(1.5,-1.4);

\draw (4.8,-0.6) node {and\hspace{18pt}$ $};
\end{tikzpicture}
\begin{tikzpicture}
\draw (0.6,-0.7) node {$\nu=$};

\draw (3,-0.2) node {j};
\draw (3.3,-0.2) node {i};
\draw (2.4,-0.6) node {j};

\draw (1.9,-1) node {$\iddots$};

\draw (2.25,-0.8)--(2.55,-0.8)--(2.55,-0.4)--(3.15,-0.4)--(3.15,0)--(1.2,0)--(1.2,-1.4)--(1.5,-1.4);

\draw (3.65,-0.8) node{,};
\end{tikzpicture}
\]
it follows that $\epsilon_j(\nu)\geq 2$ (the removable nodes on the first 2 rows of $\nu$ are normal of residue $j$), from which follows from Lemma \hyperref[l39b]{\ref*{l39}\ref*{l39b}} that $e_jD^\nu$ is non-zero and not simple. In particular
\[e_je_iD^\lambda\sim \ldots|e_jD^\nu|\ldots\]
is also non-zero and not simple and so the lemma holds as $e_je_iD^\lambda$ is self-dual by Lemma \ref{l57}.

Assume now that $\lambda_1\not\equiv\lambda_2\Md 2$. Then $\lambda_1+\lambda_2$ is odd and so $h(\lambda)\geq 3$. The removable nodes on the first 2 rows are the only normal nodes of $\lambda$ (from Lemma \ref{l42} as they have the same residue). In particular if $\rho:=(\lambda_1,\lambda_2-1,\lambda_3,\lambda_4,\ldots)$ then $D^\rho=\tilde{e}_iD^\lambda$. As the removable nodes on the first two rows of $\lambda$ have the same residue the removable node on the third row must have a different residue, from Lemma \ref{l42}. So
\[\begin{tikzpicture}
\draw (-0.6,-0.7) node {$\lambda=$};

\draw (3,-0.2) node {i};
\draw (2.1,-0.6) node {j};
\draw (2.4,-0.6) node {i};
\draw (1.2,-1) node {j};

\draw (0.7,-1.4) node {$\iddots$};

\draw (1.05,-1.2)--(1.35,-1.2)--(1.35,-0.8)--(2.55,-0.8)--(2.55,-0.4)--(3.15,-0.4)--(3.15,0)--(0,0)--(0,-1.8)--(0.3,-1.8);

\draw (4.5,-0.6) node {and\hspace{18pt}$ $};
\end{tikzpicture}
\begin{tikzpicture}
\draw (-0.6,-0.7) node {$\rho=$};

\draw (3,-0.2) node {i};
\draw (2.1,-0.6) node {j};
\draw (2.4,-0.6) node {i};
\draw (1.2,-1) node {j};

\draw (0.7,-1.4) node {$\iddots$};

\draw (1.05,-1.2)--(1.35,-1.2)--(1.35,-0.8)--(2.25,-0.8)--(2.25,-0.4)--(3.15,-0.4)--(3.15,0)--(0,0)--(0,-1.8)--(0.3,-1.8);

\draw (3.45,-0.8) node{.};
\end{tikzpicture}\]
So the removable node at the end of the third row of $\rho$ is normal with residue $j$. In particular $\epsilon_j(\rho)\geq 1$ and then $\dim\End_{\s_{n-2}}(e_je_iD^\lambda)\geq 2$ from Lemma \ref{l23}.
\end{proof}

We will now prove that the assumptions of Lemma \ref{l17} and, in most cases, those of Lemma \ref{l36} hold.

\begin{lemma}\label{l19}
Let $p=2$ and $\lambda\vdash n$ with $n\geq 4$ even be 2-regular and assume that $\epsilon_0(\lambda)+\epsilon_1(\lambda)=2$. Then
\begin{align*}
&\dim\End_{\s_{n-1}}(D^\lambda\downarrow_{\s_{n-1}})+\dim\Hom_{\s_n}(S^{(n-1,1)},\End_F(D^\lambda))+1\\
&\hspace{12pt}<\dim\End_{\s_{n-2}}(D^\lambda\downarrow_{\s_{n-2}}).
\end{align*}
\end{lemma}

\begin{proof}
As $\epsilon_0(\lambda),\epsilon_1(\lambda)\leq 2$, from Lemma \ref{l8} we have that
\[\dim\Hom_{\s_n}(S^{(n-1,1)},\End_F(D^\lambda))\leq 2.\]
So, from Lemma \ref{l53}, it is enough to prove that $\dim\End_{\s_{n-2}}(D^\lambda\downarrow_{\s_{n-2}})>5$.

Let first $\epsilon_0(\lambda),\epsilon_1(\lambda)=1$. Then $e_0e_0D^\lambda,e_1e_1D^\lambda=0$ from Lemma \ref{l39}. So from Lemmas
\ref{l45}, \ref{l4} and \ref{l22} we have that
\begin{align*}
\dim\End_{\s_{n-2}}(D^\lambda\downarrow_{\s_{n-2}})&=\dim\End_{\s_{n-2}}(e_0e_1D^\lambda\oplus e_1e_0D^\lambda)\\
&=\dim\End_{\s_{n-2}}(e_0e_1D^\lambda)+\dim\End_{\s_{n-2}}(e_1e_0D^\lambda)\\
&\hspace{24pt}+2\dim\Hom_{\s_{n-2}}(e_0e_1D^\lambda,e_1e_0D^\lambda)\\
&\geq 2\epsilon_0(\lambda)+2\epsilon_1(\lambda)+2\epsilon_0(\lambda)\epsilon_1(\lambda)\\
&=6.
\end{align*}

Assume now that $\epsilon_i(\lambda)=2$ and $\epsilon_j(\lambda)=0$. Then $e_jD^\lambda=0$ from Lemma \ref{l39} and so, from Lemmas \ref{l45}, \ref{l39} and \ref{l44} and by block decomposition of $D^\lambda\downarrow_{\s_{n-2}}$,
\begin{align*}
\dim\End_{\s_{n-2}}(D^\lambda\downarrow_{\s_{n-2}})&=\dim\End_{\s_{n-2}}(e_ie_iD^\lambda\oplus e_je_iD^\lambda)\\
&=\dim\End_{\s_{n-2}}(e_ie_iD^\lambda)+\dim\End_{\s_{n-2}}(e_je_iD^\lambda)\\
&\geq4+2.
\end{align*}
\end{proof}

\begin{lemma}\label{l43}
Let $p=2$ and $\lambda\vdash n$ with $n\geq 4$ even be 2-regular Assume that $\epsilon_0(\lambda)+\epsilon_1(\lambda)=2$. If $\epsilon_0(\lambda),\epsilon_1(\lambda)=1$ further assume that we do not have
\[\lambda_1\equiv\lambda_2\equiv\ldots\equiv\lambda_{h(\lambda)-1}.\]
Then
\[\dim\End_{\s_{n-2,2}}(D^\lambda\downarrow_{\s_{n-2,2}})>\dim\Hom_{\s_n}(S^{(n-1,1)},\End_F(D^\lambda))+1.\]
\end{lemma}

\begin{proof}
From Lemma \ref{l54} it is enough to prove that
\[\dim\End_{\s_{n-2}}(D^\lambda\downarrow_{\s_{n-2}})>2\dim\Hom_{\s_n}(S^{(n-1,1)},\End_F(D^\lambda))+2.\]
So, from Lemma \ref{l8}, it is enough to prove that $\dim\End_{\s_{n-2,2}}(D^\lambda\downarrow_{\s_{n-2,2}})\geq 4$ or that  $\dim\End_{\s_{n-2}}(D^\lambda\downarrow_{\s_{n-2}})\geq 7$.

Assume first that $\epsilon_i(\lambda)=2$ and $\epsilon_j(\lambda)=0$. In this case $e_jD^\lambda=0$ from Lemma \ref{l39} and then $D^\lambda\downarrow_{\s_{n-2}}=e_ie_iD^\lambda\oplus e_je_iD^\lambda$ from Lemma \ref{l39}. So $D^\lambda\downarrow_{\s_{n-2,2}}=\overline{e}_i^2D^\lambda\oplus M$, where $M\downarrow_{\s_{n-2}}=e_je_iD^\lambda$. It then follows from Lemma \ref{l31} and the block decomposition of $D^\lambda\downarrow_{\s_{n-2,2}}$ that
\[\dim\End_{\s_{n-2,2}}(D^\lambda\downarrow_{\s_{n-2,2}})=2+\dim\End_{\s_{n-2,2}}(M).\]
From Lemma \ref{l44} we have that $\dim\End_{\s_{n-2}}(e_je_iD^\lambda)\geq 2$. In particular $e_je_iD^\lambda$ is non-zero and not simple. It then follows that the same holds for $M$ (as $\s_{n-2,2}\cong \s_{n-2}\times \s_2$ and the only simple module of $\s_2$ in characteristic 2 is $D^{(2)}$). Further $M$ is self-dual, as it is the restriction to a block of $\s_{n-2,2}$ of a self-dual module. So $\dim\End_{\s_{n-2,2}}(M)\geq 2$ and then $\dim\End_{\s_{n-2,2}}(D^\lambda\downarrow_{\s_{n-2,2}})\geq 4$. In particular the lemma holds when $\epsilon_i(\lambda)=2$ and $\epsilon_j(\lambda)=0$.

Assume now that $\epsilon_0(\lambda),\epsilon_1(\lambda)=1$. Then $(\lambda_1,\ldots,\lambda_{h(\lambda)-1})$ is not a JS-partition.  Let $a_k$ be the residue of the removable node on the $k$-th row of $\lambda$. Let $r\geq 2$ minimal such that $a_r\not=a_{r-1}$ (it exists from Lemma \ref{l42}). From Lemma \ref{l42} we have that the removable node on the $r$-th row of $\lambda$ is normal. Since $\epsilon_0(\lambda),\epsilon_1(\lambda)=1$ we have that $a_1\not=a_r$. From Lemma \ref{l39}, $e_0e_0D^\lambda,e_1e_1D^\lambda=0$ and so $D^\lambda\downarrow_{\s_{n-2}}=e_1e_0D^\lambda\oplus e_0e_1D^\lambda$ from Lemma \ref{l45}. If $e_iD^\lambda=D^{\nu_i}$, we then have from Lemmas \ref{l23} and \ref{l22} that
\begin{align*}
\dim\End_{\s_{n-2}}(D^\lambda\downarrow_{\s_{n-2}})&=\dim\End_{\s_{n-2}}(e_0e_1D^\lambda)+\dim\End_{\s_{n-2}}(e_1e_0D^\lambda)\\
&\hspace{24pt}+\dim\Hom_{\s_{n-2}}(e_0e_1D^\lambda,e_1e_0D^\lambda)\\
&\hspace{24pt}+\dim\Hom_{\s_{n-2}}(e_1e_0D^\lambda,e_0e_1D^\lambda)\\
&\geq \epsilon_1(\nu_0)+\epsilon_0(\nu_1)+2.
\end{align*}
So it is enough to prove that $\epsilon_1(\nu_0)+\epsilon_0(\nu_1)\geq 5$. Let $i,j$ with $a_1=i$ and $a_r=j$. Then from Lemma \ref{l4'} 
we already know that $\epsilon_j(\nu_i)=3$. So it is enough to prove that $\epsilon_i(\nu_j)\geq 2$. By definition of $r$ we have that $a_k\not=a_{k-1}$ for $2\leq k<r$. So
\[\lambda_1\equiv\lambda_2\equiv\ldots\equiv\lambda_{r-1}\Md 2\]
and so by assumption $r+1\leq h(\lambda)$. Notice that $a_{r+1}\not=a_r$ from Lemma \ref{l42} as $a_r=a_{r-1}$ by definition of $r$. So 
\[\lambda_1\equiv\lambda_2\equiv\ldots\equiv\lambda_{r-1}\not\equiv \lambda_r\equiv\lambda_{r+1}\Md 2.\]
As $\nu_j=(\lambda_1,\ldots,\lambda_{r-1},\lambda_r-1,\lambda_{r+1},\lambda_{r+2},\ldots)$ and $(\nu_j)_{r+1}=\lambda_{r+1}>0$ we have that $h(\nu_j)\geq r+1$ and
\[(\nu_j)_1\equiv(\nu_j)_2\equiv\ldots\equiv(\nu_j)_{r-1}\equiv (\nu_j)_r\not\equiv(\nu_j)_{r+1}\Md 2.\]
In particular from Lemma \ref{l55} we have that $\nu_j$ is not a JS-partition. From Lemmas \ref{l47} and \ref{l55} we then have that
\[\epsilon_i(\nu_j)\geq 2-\epsilon_j(\nu_j)=3-\epsilon_j(\lambda)=2.\]
\end{proof}

For the remaining cases we will prove directly that $\End_F(D^\lambda)$ contains a module of the form $D^{(n-2,2)}$ or $D^{(n-1,1)}|(D^{(n)}\oplus D^{(n-2,2)})$.

\begin{lemma}\label{l46}
Let $p=2$ and $\lambda\vdash n$ with $n\equiv 0\Md 4$ and $n\geq 8$ even be 2-regular. Assume that $\epsilon_0(\lambda),\epsilon_1(\lambda)=1$ and
\[\lambda_1\equiv\lambda_2\equiv\ldots\equiv\lambda_{h(\lambda)-1}\Md 2.\]
Then $M\subseteq L\subseteq\End_F(D^\lambda)$, where $L$ is a quotient of $M^{(n-2,2)}$ and $M$ is of one of the following forms:
\begin{itemize}
\item
$M\cong D^{(n-2,2)}$,

\item
$M=D^{(n-1,1)}|(D^{(n)}\oplus D^{(n-2,2)})\subseteq M^{(n-2,2)}$.
\end{itemize}
\end{lemma}

\begin{proof}
As $D^\lambda\downarrow_{\s_{n-1}}$ is not irreducible (Lemma \ref{l53}) we have from Lemma \ref{l55} and by assumption that
\[\lambda_1\equiv\ldots\equiv\lambda_{h(\lambda)-1}\not\equiv\lambda_{h(\lambda)}\Md 2.\]
In particular
\[0\equiv n\equiv \lambda_1\cdot h(\lambda)+1\Md 2.\]
So $\lambda_1,\ldots,\lambda_{h(\lambda)-1}$ are odd, $\lambda_{h(\lambda)}$ is even and $\lambda$ has an odd number of parts. Then
\[
\begin{tikzpicture}
\draw (-0.8,-0.7) node {$\lambda=$};

\draw (0,0) node {0};
\draw (1.7,0) node {$\cdots$};
\draw (3.4,0) node {0};
\draw (0,-0.5) node {$\vdots$};
\draw (2.5,-0.5) node {$\iddots$};
\draw (0,-1.2) node {1};
\draw (1,-1.2) node {$\cdots$};
\draw (2,-1.2) node {1};
\draw (0,-1.6) node {0};
\draw (0.5,-1.6) node {$\cdots$};
\draw (1,-1.6) node {1};
\draw (1.3,-1.6) node {0};
\draw (0,-2) node {1};

\draw (2.9,-0.2) -- (3.55,-0.2) -- (3.55,0.2) -- (-0.15,0.2) -- (-0.15,-1.8) -- (1.15,-1.8) -- (1.15,-1.4) -- (2.15,-1.4) -- (2.15,-1);
\end{tikzpicture}
\]
and so $\phi_0(\lambda),\phi_1(\lambda)\geq 1$ as $(h(\lambda),\lambda_{h(\lambda)}+1)$ and $(h(\lambda)+1,1)$ are conormal (they are always conormal).

We will first prove that $(S^{(n-1,1)})^*\subseteq\End_F(D^\lambda)$. Let $\nu_i$ with $\tilde{e}_iD^\lambda=D^{\nu_i}$ for $0\leq i\leq 1$ (since $\epsilon_i(\lambda)=1$). As $\epsilon_i(\lambda)=1$, from Lemma \ref{l39} we have that
\[f_ie_iD^\lambda\cong f_i\tilde{e}_iD^\lambda=f_iD^{\nu_i}.\]
From Lemma \ref{l47} we also have that $\phi_i(\nu_i)=\phi_i(\lambda)+1\geq 2$. So, from Lemmas \ref{l40} and \ref{l47},
\begin{align*}
f_0e_0D^\lambda\oplus f_1e_1D^\lambda&\sim\hspace{7pt}(\hspace{5pt}\rlap{$\hspace{-19pt}\overbrace{\phantom{D^\lambda}}^{\soc(f_0e_0D^\lambda)}$}D^\lambda\hspace{5pt}|\hspace{5pt}\ldots\hspace{5pt}|\hspace{5pt}\rlap{$\hspace{-18pt}\overbrace{\phantom{D^\lambda}}^{\hd(f_0e_0D^\lambda)}$}D^\lambda\hspace{5pt})\hspace{7pt}\oplus\hspace{7pt} (\hspace{5pt}\rlap{$\hspace{-19pt}\overbrace{\phantom{D^\lambda}}^{\soc(f_1e_1D^\lambda)}$}D^\lambda\hspace{5pt}|\hspace{5pt}\ldots\hspace{5pt}|\hspace{5pt}\rlap{$\hspace{-18pt}\overbrace{\phantom{D^\lambda}}^{\hd(f_1e_1D^\lambda)}$}D^\lambda\hspace{5pt})\\
&\sim \overbrace{(D^\lambda\oplus D^\lambda)}^{\soc(f_0e_0D^\lambda\oplus f_1e_1D^\lambda)}|\ldots|\overbrace{(D^\lambda\oplus D^\lambda)}^{\hd(f_0e_0D^\lambda\oplus f_1e_1D^\lambda)}.
\end{align*}

From Lemma \ref{l45'},
\[D^\lambda\otimes M^{(n-1,1)}\cong f_0e_0D^\lambda\oplus f_1e_1D^\lambda\oplus f_0e_1D^\lambda\oplus f_1e_0D^\lambda.\]
Further from Lemma \ref{l1} and self-duality of $M^{(n-1,1)}$,
\[D^\lambda\otimes M^{(n-1,1)}\sim D^\lambda\otimes(D^{(n)}|(S^{((n-1,1)})^*)\sim D^\lambda|(D^\lambda\otimes (S^{(n-1,1)})^*)\]
there exists $D\subseteq D^\lambda\downarrow_{\s_{n-1}}\uparrow^{\s_n}$ with $D\cong D^\lambda$ and
\[D^\lambda\otimes(S^{(n-1,1)})^*\cong D^\lambda\downarrow_{\s_{n-1}}\uparrow^{\s_n}/D.\]
From the block decomposition of $D^\lambda\downarrow_{\s_{n-1}}\uparrow^{\s_n}$ (which comes from the definition of $e_i$ and $f_i$) it follows that $D\subseteq f_0e_0D^\lambda\oplus f_1e_1D^\lambda$. So, as the head and socle of $f_0e_0D^\lambda\oplus f_1e_1D^\lambda$ are disjoint, $D$ is not contained in the head of $D^\lambda\downarrow_{\s_{n-1}}\uparrow^{\s_n}$. In particular, as $D$ is simple,
\[\hd(D^\lambda\otimes (S^{(n-1,1)})^*)\cong\hd(D^\lambda\downarrow_{\s_{n-1}}\uparrow^{\s_n}/D)=\hd(D^\lambda\downarrow_{\s_{n-1}}\uparrow^{\s_n})\]
and so from Lemmas \ref{l53} and \ref{l58}
\begin{align*}
\dim\Hom_{\s_n}((S^{(n-1,1)})^*,\End_F(D^\lambda))&=\dim\Hom_{\s_n}(D^\lambda,D^\lambda\otimes S^{(n-1,1)})\\
&=\dim\Hom_{\s_n}(D^\lambda\otimes (S^{(n-1,1)})^*,D^\lambda)\\
&=\dim\Hom_{\s_n}(D^\lambda\downarrow_{\s_{n-1}}\uparrow^{\s_n},D^\lambda)\\
&=\dim\End_{\s_{n-1}}(D^\lambda\downarrow_{\s_{n-1}})\\
&=2.
\end{align*}
From Lemma \ref{l1}, $(S^{(n-1,1)})^*=D^{(n-1,1)}|D^{(n)}$. As $D^{(n)}$ is contained exactly once in $\End_F(D^\lambda)$ it follows that $(S^{(n-1,1)})^*\subseteq\End_F(D^\lambda)$.

We can assume that no quotient of $M^{(n-2,2)}$ containing $D^{(n-2,2)}$ as submodule is contained in $\End_F(D^\lambda)$. From Lemma \ref{l62} we then have that $\End_F(D^\lambda)$ does not contain any quotient of $D^{(n-2,1)}\uparrow^{\s_n}$ which contains $D^{(n-2,2)}$ as a submodule. From Lemma \ref{l3} and by self-duality of $D^{(n-2,1)}\uparrow^{\s_n}$ we have that $(S^{(n-1,1)})^*\subseteq D^{(n-2,1)}\uparrow^{\s_n}$. So from Lemmas \ref{l6} (or \ref{l49}) and \ref{l63} we can further assume that $D^{(n-2,1)}\uparrow^{\s_n}\not\subseteq \End_F(D^\lambda)$.

From Lemmas \ref{l53}, \ref{l27} and \ref{l7}, 
\begin{align*}
&\dim\Hom_{\s_n}(D^{(n-2,1)}\uparrow^{\s_n},\End_F(D^\lambda))\\
&\hspace{12pt}=\dim\Hom_{\s_n}(M^{(n-2,1,1)},\End_F(D^\lambda))-\dim\Hom_{\s_n}(M^{(n-1,1)},\End_F(D^\lambda))\\
&\hspace{12pt}=\dim\End_{\s_{n-2}}(D^\lambda\downarrow_{\s_{n-2}})-\dim\End_{\s_{n-1}}(D^\lambda\downarrow_{\s_{n-1}})\\
&\hspace{12pt}\geq 6-2\\
&\hspace{12pt}=4.
\end{align*}
By Lemma \ref{l6}, the socle of $D^{(n-2,1)}\uparrow^{\s_n}$ is isomorphic to $D^{(n-1,1)}$ and $D^{(n-2,2)}\subseteq (D^{(n-2,1)}\uparrow^{\s_n})/D^{(n-1,1)}$. As by assumption neither $D^{(n-2,1)}\uparrow^{\s_n}$ nor any of its quotients containing $D^{(n-2,2)}$ as submodule are contained in $\End_F(D^\lambda)$ it follows that
\begin{align*}
&\dim\Hom_{\s_n}(D^{(n-2,1)}\uparrow^{\s_n}/(D^{(n-1,1)}|D^{(n-2,2)}),\End_F(D^\lambda))\\
&\hspace{12pt}=\dim\Hom_{\s_n}(D^{(n-2,1)}\uparrow^{\s_n}/D^{(n-1,1)},\End_F(D^\lambda))\\
&\hspace{12pt}=\dim\Hom_{\s_n}(D^{(n-2,1)}\uparrow^{\s_n},\End_F(D^\lambda))\\
&\hspace{12pt}\geq 4.
\end{align*}
Notice that
\[\xymat{
&&D^{(n-1,1)}\ar@{-}[dr]\ar@{-}[dl]\\
&D^{(n)}\ar@{-}[dd]&&D^{(n-2,2)}\ar@{-}[d]\ar@{-}[ddll]\\
D^{(n-2,1)}\uparrow^{\s_n}/(D^{(n-1,1)}|D^{(n-2,2)})=&&&D^{(n-1,1)}&.\\
&D^{(n)}
}\]
Let $A$ be any quotient of $D^{(n-2,1)}\uparrow^{\s_n}/(D^{(n-1,1)}|D^{(n-2,2)})$ containing the copy of $D^{(n)}$ contained in the socle of $D^{(n-2,1)}\uparrow^{\s_n}/(D^{(n-1,1)}|D^{(n-2,2)})$. Then $A$ contains an indecomposable module of the form $D^{(n)}|D^{(n)}$. So $1\uparrow_{A_n}^{\s_n}\subseteq A$ from Lemma \ref{l26}. Assume that $A\subseteq\End_F(D^\lambda)$. Then $1\uparrow_{A_n}^{\s_n}\subseteq\End_F(D^\lambda)$ and so
\begin{align*}
\dim\End_{A_n}(D^\lambda\downarrow_{A_n})&=\dim\Hom_{\s_n}(1\uparrow_{A_n}^{\s_n},\End_F(D^\lambda))\\
&\geq\dim\End_{\s_n}(1\uparrow_{A_n}^{\s_n})\\
&=2.
\end{align*}
So in this case $D^\lambda\downarrow{A_n}$ splits. This contradicts $h(\lambda)$ being odd and $\lambda_{h(\lambda)}$ even (see Theorem 1.1 of \cite{b1}). It follows that no such module $A$ is contained in $\End_F(D^\lambda)$. Let $M$ is the submodule of $D^{(n-2,1)}\uparrow^{\s_n}$ of the form $D^{(n-1,1)}|(D^{(n)}\oplus D^{(n-2,2)})$. Since $D^{(n)}$ is contained only once in $D^{(n-2,1)}\uparrow^{\s_n}$ as submodule, we have that
\[(D^{(n-2,1)}\uparrow^{\s_n}/(D^{(n-1,1)}|D^{(n-2,2)}))/D^{(n)}=D^{(n-2,1)}\uparrow^{\s_n}/M.\]
In particular
\begin{align*}
&\dim\Hom_{\s_n}(D^{(n-2,1)}\uparrow^{\s_n}/M,\End_F(D^\lambda))\\
&\hspace{12pt}=\dim\Hom_{\s_n}(D^{(n-2,1)}\uparrow^{\s_n}/(D^{(n-1,1)}|D^{(n-2,2)}),\End_F(D^\lambda))\\
&\hspace{12pt}\geq 4\\
&\hspace{12pt}\geq \dim\Hom_{\s_n}(S^{(n-1,1)},\End_F(D^\lambda))+2
\end{align*}
(the last line following from Lemma \ref{l8}).

From Lemma  \ref{l3} we have that $D^{(n-2,1)}\uparrow^{\s_n}$ has a quotient isomorphic to $S^{(n-1,1)}$. As $S^{(n-1,1)}=D^{(n)}|D^{(n-1,1)}$ from Lemma \ref{l1}, this quotient is the unique quotient of $D^{(n-2,1)}\uparrow^{\s_n}$ of the form $D^{(n)}|D^{(n-1,1)}$. So
\[D^{(n-2,1)}\uparrow^{\s_n}\sim \overbrace{D^{(n-1,1)}|(D^{(n)}\oplus D^{(n-2,2)})}^M|D^{(n-1,1)}|D^{(n-2,2)}|\overbrace{D^{(n)}|D^{(n-1,1)}}^{S^{(n-1,1)}}.\]

As $S^{(n-1,1)}=D^{(n)}|D^{(n-1,1)}$ from Lemma \ref{l1}, we also have from Lemma \ref{l8} that
\[\dim\Hom_{\s_n}(D^{(n-1,1)},\End_F(D^\lambda))\leq \dim\Hom_{\s_n}(S^{(n-1,1)},\End_F(D^\lambda))\leq 2,\]
so that $D^{(n-1,1)}$ is contained at most twice in $\End_F(D^\lambda)$ as a submodule. 

By assumption no quotient of $D^{(n-2,1)}\uparrow^{\s_n}$ containing $D^{(n-2,2)}$ as submodule is contained in $\End_F(D^\lambda)$. As 
\[(D^{(n-2,1)}\uparrow^{\s_n})/M\sim D^{(n-1,1)}|D^{(n-2,2)}|\overbrace{D^{(n)}|\underbrace{D^{(n-1,1)}}_{\hd((D^{(n-2,1)}\uparrow^{\s_n})/M)}}^{S^{(n-1,1)}}\]
and

\begin{align*}
&\dim\Hom_{\s_n}(D^{(n-2,1)}\uparrow^{\s_n}/M,\End_F(D^\lambda)\\
&\hspace{12pt}\geq\dim\Hom_{\s_n}(S^{(n-1,1)},\End_F(D^\lambda))+2\\
&\hspace{12pt}\geq\dim\Hom_{\s_n}(S^{(n-1,1)},\End_F(D^\lambda))+\dim\Hom_{\s_n}(D^{(n-1,1)},\End_F(D^\lambda)),
\end{align*}
it then follows that, if $C\subseteq\End_F(D^\lambda)$ with $C\cong D^{(n-1,1)}$, there exists $C\subseteq B\subseteq\End_F(D^\lambda)$ with $B\not\cong D^{(n-1,1)}$ a quotient of $(D^{(n-2,1)}\uparrow^{\s_n})/M$. In particular, as $(D^{(n-2,1)}\uparrow^{\s_n})/M$ has only one composition factor isomorphic to $D^{(n)}$, we have that $B$ is a quotient of $D^{(n-2,1)}\uparrow^{\s_n}$ with $B\not\cong D^{(n-1,1)}$ and not containing any submodule of the form $D^{(n)}|D^{(n)}$.

By the first part of the proof $(S^{(n-1,1)})^*\subseteq \End_F(D^\lambda)$. Since $D^{(n-1,1)}\subseteq(S^{(n-1,1)})^*$ by Lemma \ref{l1}, we can then conclude by Lemma \ref{l63}.
\end{proof}

\section{JS-partitions}\label{s4}

We will now consider JS-partitions. Notice that from the definition of JS-partitions and from Lemma \ref{l53}, a 2-regular partitions is a JS-partition if and only if it has only 1 normal node.

Also in this case we will need some case analysis. In this case results will be different from those obtained when $\lambda$ has at least 2 normal nodes and results will depend on whether $\lambda=(m+1,m-1)$ or not.

\begin{lemma}\label{l10}
Let $p=2$ and $n=2m\geq 4$. Let $\lambda\not=(n)$ be a JS-partition which satisfies $\dim\Hom_{\s_n}(S^{(n-1,1)},\End_F(D^\lambda))\geq 1$. Then $\lambda=(m+1,m-1)$ and $\dim\Hom_{\s_n}(S^{(n-1,1)},\End_F(D^\lambda))=1$.

If $m$ is odd $S^{(n-1,1)}\subseteq \End_F(D^{(m+1,m-1)})$, while if $m$ is even $D^{(n-1,1)}\subseteq \End_F(D^{(m+1,m-1)})$.
\end{lemma}

\begin{proof}
Assume that $\dim\Hom_{\s_n}(S^{(n-1,1)},\End_F(D^\lambda))\geq 1$.  As $M^{(n-1,1)}$ and $D^{(n-1,1)}$ are self-dual, from Lemma \ref{l58},
\begin{align*}
\Hom_{\s_n}(S^{(n-1,1)},\End_F(D^\lambda))&\cong\Hom_{\s_n}(D^\lambda,D^\lambda\otimes (S^{(n-1,1)})^*),\\
\Hom_{\s_n}(M^{(n-1,1)},\End_F(D^\lambda))&\cong\Hom_{\s_n}(D^\lambda,D^\lambda\otimes M^{(n-1,1)}),\\
\Hom_{\s_n}(D^{(n-1,1)},\End_F(D^\lambda))&\cong\Hom_{\s_n}(D^\lambda,D^\lambda\otimes D^{(n-1,1)}).
\end{align*}
Since $\lambda$ is a JS-partition, it follows that
\begin{align*}
\dim\Hom_{\s_n}(D^\lambda,D^\lambda\otimes M^{(n-1,1)})&=\dim\Hom_{\s_n}(M^{(n-1,1)},\End_F(D^\lambda))\\
&=\dim\End_{\s_{n-1}}(D^\lambda\downarrow_{\s_{n-1}})\\
&=1.
\end{align*}
In particular $D^\lambda$ is contained exactly once in $D^\lambda\otimes M^{(n-1,1)}$. Further by assumption $D^\lambda$ is also contained in $D^\lambda\otimes (S^{(n-1,1)})^*$. So, by self-duality of $M^{(n-1,1)}$ and by Lemma \ref{l1},
\[D^\lambda\otimes M^{(n-1,1)}\sim D^\lambda\otimes (D^{(n)}|(S^{(n-1,1)})^*)\sim D^\lambda|(D^\lambda\otimes (S^{(n-1,1)})^*)\supseteq D^\lambda|D^\lambda.\]
As $D^\lambda$ is contained only once in $D^\lambda\otimes M^{(n-1,1)}$ there then exists $A\subseteq D^\lambda\otimes M^{(n-1,1)}$ with $A=D^\lambda|D^\lambda$.

As $\lambda$ is a JS-partition we have from Lemma \ref{l55} that $\epsilon_i(\lambda)=1$ and $\epsilon_j(\lambda)=0$ for some $i$ and $j=1-i$. So $e_jD^\lambda=0$ from Lemma \ref{l39} and then from Lemma \ref{l45'}
\[D^\lambda\otimes M^{(n-1,1)}\cong f_ie_iD^\lambda\oplus f_je_iD^\lambda.\]
From the block decomposition of $f_ie_iD^\lambda\oplus f_je_iD^\lambda$ we then have that, up to isomorphism, $A\subseteq f_ie_iD^\lambda$.

Let $\rho$ with $D^\rho\cong \tilde{e}_iD^\lambda\cong e_iD^\lambda$ (from Lemma \ref{l39} as $\epsilon_i(\lambda)=1$). Then $A\subseteq f_ie_iD^\lambda\cong f_iD^\rho$. In particular $\phi_i(\rho)\geq 2$ from Lemmas \hyperref[l40b]{\ref*{l40}\ref*{l40b}} and \ref{l47}, that is $\phi_i(\lambda)\geq 1$ from Lemma \ref{l47}. Let $\pi$ and $\nu$ with $D^\pi=\tilde{f}_iD^\lambda$ and $D^\nu=\tilde{f}_i^{\phi_i(\lambda)}D^\lambda$. From Lemmas \ref{l9'} and \ref{l9} we have
\[e_iD^\pi\subseteq f_iD^\rho\cong e_i^{(\epsilon_i(\nu)-1)}D^\nu.\]

As $\lambda$ is has only 1 normal node it has 2 conormal nodes from Lemma \ref{l52}. In particular $\phi_i(\lambda)\leq 2$ and so $\phi_i(\rho)\leq 3$ from Lemma \ref{l47}. So, from the previous part, $\rho$ has 2 or 3 conormal nodes of content $i$.

If $\phi_i(\rho)=2$ then $\phi_i(\lambda)=1$ and so $\pi=\nu$. In this case, from Lemma \ref{l40},
\[D^\lambda|D^\lambda\subseteq f_iD^\rho\cong e_iD^\pi\sim \overbrace{D^\lambda}^{\soc(e_iD^\pi)}|\overbrace{\rlap{$\phantom{D^\lambda}$}\ldots\rlap{$\phantom{D^\lambda}$}}^{\mbox{no }D^\lambda}|\overbrace{D^\lambda}^{\hd(e_iD^\pi)}.\]
So in this case $e_iD^\pi=D^\lambda|D^\lambda$.

Assume now instead that $\phi_i(\rho)=3$. As $\epsilon_i(\pi)=\epsilon_i(\lambda)+1=2$ from Lemma \ref{l47} and from Lemmas \ref{l39} and \ref{l40} we have in this case that
\[D^\lambda|D^\lambda\subseteq f_iD^\rho\sim\underbrace{\overbrace{D^\lambda}^{\soc(e_iD^\pi)=\soc(f_iD^\rho)}|\overbrace{\rlap{$\phantom{D^\lambda}$}\ldots\rlap{$\phantom{D^\lambda}$}}^{\mbox{no }D^\lambda}|\overbrace{D^\lambda}^{\hd(e_iD^\pi)}}_{e_iD^\pi}|\overbrace{\rlap{$\phantom{D^\lambda}$}\ldots\rlap{$\phantom{D^\lambda}$}}^{\mbox{no }D^\lambda}|\overbrace{D^\lambda}^{\hd(f_iD^\rho)}.\]
So also in this case $e_iD^\pi=D^\lambda|D^\lambda$.

On the other hand if $e_iD^\pi=D^\lambda|D^\lambda$ then
\[D^\lambda|D^\lambda=e_iD^\pi\subseteq f_iD^\rho\cong f_ie_iD^\lambda\subseteq D^\lambda\otimes M^{(n-1,1)}\sim D^\lambda|(D^\lambda\otimes (S^{(n-1,1)})^*).\]
So in this case $D^\lambda\subseteq D^\lambda\otimes (S^{(n-1,1)})^*$ and $\dim\Hom_{\s_n}(S^{(n-1,1)},\End_F(D^\lambda))\geq 1$.

As $n$ is even and $\lambda$ is a JS-partition we have one of the following:

\vspace{6pt}
\noindent
(i) the parts of $\lambda$ are even and $\lambda$ has an even number of parts,

\vspace{6pt}
\noindent
(ii) the parts of $\lambda$ are even and $\lambda$ has an odd number of parts,

\vspace{6pt}
\noindent
(iii) the parts of $\lambda$ are odd and $\lambda$ has an even number of parts.
\vspace{6pt}

We will now in each of the above cases check when $e_iD^\pi=D^\lambda|D^\lambda$ holds and study this cases more in details. Notice  that the conormal nodes of $\lambda$ are those corresponding to nodes  $(h(\lambda),\lambda_{h(\lambda)}+1)$ and $(h(\lambda)+1,1)$ (these nodes are always conormal and $\lambda$ has only 2 conormal nodes).

\vspace{6pt}
\noindent
(i) In this case the parts of $\lambda$ are even and $\lambda$ has an even number of parts,, so
\[\begin{tikzpicture}
\draw (-0.6,-0.7) node {$\lambda=$};

\draw (0.2,-0.2) node {0};
\draw (1.6,-0.2) node {$\ldots$};
\draw (3,-0.2) node {1};
\draw (0.2,-0.6) node {$\vdots$};
\draw (2.3,-0.6) node {$\iddots$};
\draw (0.2,-1.2) node {1};
\draw (0.6,-1.2) node {$\ldots$};
\draw (1,-1.2) node {0};
\draw (1.4,-1.2) node {1};
\draw (0.2,-1.6) node {0};

\draw (2,-1.0) -- (1.2,-1.0) -- (1.2,-1.4) -- (0,-1.4) -- (0,0) -- (3.2,0) -- (3.2,-0.4) -- (2.6,-0.4);
\end{tikzpicture}\]
So $i=1$ (it is the residue of $(1,\lambda_1)$, as this is the only normal node). As the only conormal nodes are $(h(\lambda),\lambda_{h(\lambda)}+1)$ and $(h(\lambda)+1,1)$ and only the first of them has residue 1, we have that $\pi=(\lambda_1,\ldots,\lambda_{h(\lambda)-1},\lambda_{h(\lambda)}+1)$ in this case.

First assume that $h(\lambda)\geq 4$ or $\lambda_1>\lambda_2+2$. In this case $\overline{\pi}=(\lambda_1-1,\lambda_2,\ldots,\lambda_{h(\lambda)-1},\lambda_{h(\lambda)}+1)$ is 2-regular, as $\pi$ is 2-regular, $\overline{\pi}=\pi\setminus (1,\lambda_1)$ and
\[\overline{\pi}_1=\lambda_1-1>\left\{\begin{array}{ll}
\lambda_2,&h(\lambda)\geq 4\\
\lambda_{h(\lambda)}+1,&h(\lambda)=2\mbox{ and }\lambda_1>\lambda_2+2
\end{array}\right. =\overline{\pi}_2\]
(since in this case the parts of $\lambda$ are all even and $\lambda$ has an even number of parts). As $\overline{\pi}$ is 2-regular and is obtained from $\pi$ by removing a normal node of residue 1 (the removable node in the first row is always normal), we have that $D^{\overline{\pi}}$ is a composition factor of $e_1D^\pi$ from Lemma \ref{l56}. In particular in this case $e_1D^\pi\not=D^\lambda|D^\lambda$.

If $h=2$ and $\lambda_1=\lambda_2+2$ then $\lambda=(m+1,m-1)$ and $\pi=(m+1,m)$. Also $m$ is odd. In this case $e_1D^\pi=D^\lambda|D^\lambda$ as $\deg(D^\pi)=2\deg(D^\lambda)$ (Theorem 5.1 of \cite{b1}) and $[e_1D^\pi:D^\lambda]=2$ from Lemma \hyperref[l39b]{\ref*{l39}\ref*{l39b}}. From the previous part we have that $D^\lambda\otimes M^{(n-1,1)}=f_1e_1D^\lambda\oplus f_0e_1D^\lambda$. Since $D^\lambda$ and $f_0e_1D^\lambda$ correspond to distinct blocks,
\[[D^\lambda\otimes M^{(n-1,1)}:D^\lambda]\hspace{-1pt}=\hspace{-1pt}[f_1e_1D^\lambda:D^\lambda]\hspace{-1pt}=\hspace{-1pt}[f_1D^\rho:D^\lambda]\hspace{-1pt}=\hspace{-1pt}\phi_1(\rho)\hspace{-1pt}=\hspace{-1pt}\phi_1(\lambda)+1\hspace{-1pt}=\hspace{-1pt}2\]
from Lemmas \hyperref[l40b]{\ref*{l40}\ref*{l40b}} and \ref{l47} and as $(h(\lambda),\lambda_{h(\lambda)}+1)$ is the only conormal node of $\lambda$ of residue 1. Further, from Lemma \ref{l1},
\[D^\lambda\otimes M^{(n-1,1)}\sim D^\lambda|(D^\lambda\otimes(S^{(n-1,1)})^*)\sim D^\lambda|(D^\lambda\otimes D^{(n-1,1)})|D^\lambda\]
it follows that $[D^\lambda\otimes(S^{(n-1,1)})^*]=1$ and $[D^\lambda\otimes D^{(n-1,1)}]=0$. In particular $D^{\lambda}\not\subseteq D^\lambda\otimes D^{(n-1,1)}$ and so $D^{(n-1,1)}\not\subseteq\End_F(D^\lambda)$ and
\begin{align*}
1&\leq\dim\Hom_{\s_n}(S^{(n-1,1)},\End_F(D^\lambda))\\
&=\dim\Hom_{\s_n}(D^\lambda,D^\lambda\otimes (S^{(n-1,1)})^*)\\
&\leq [D^\lambda\otimes(S^{(n-1,1)})^*]\\
&=1.
\end{align*}
The lemma then follows in the case where the parts of $\lambda$ are even and $\lambda$ has an even number of parts, as $S^{(n-1,1)}=D^{(n)}|D^{(n-1,1)}$ from Lemma \ref{l1}.

\vspace{6pt}
\noindent
(ii) We next consider the case where the parts of $\lambda$ are even and $\lambda$ has an odd number of parts. In this case $\lambda$ has at least 3 parts, since $\lambda\not=(n)$ and it has an odd number of parts. Also
\[\begin{tikzpicture}
\draw (-0.6,-0.7) node {$\lambda=$};

\draw (0.2,-0.2) node {0};
\draw (1.6,-0.2) node {$\ldots$};
\draw (3,-0.2) node {1};
\draw (0.2,-0.6) node {$\vdots$};
\draw (2.3,-0.6) node {$\iddots$};
\draw (0.2,-1.2) node {0};
\draw (0.6,-1.2) node {$\ldots$};
\draw (1,-1.2) node {1};
\draw (1.4,-1.2) node {0};
\draw (0.2,-1.6) node {1};

\draw (2,-1.0) -- (1.2,-1.0) -- (1.2,-1.4) -- (0,-1.4) -- (0,0) -- (3.2,0) -- (3.2,-0.4) -- (2.6,-0.4);
\end{tikzpicture}\]
In this case $i=1$ ($i$ is the residue of the node $(1,\lambda_1)$). As the conormal nodes of $\lambda$ are $(h(\lambda),\lambda_{h(\lambda)}+1)$ and $(h(\lambda)+1,1)$ and only the second has residue 1, we have that $\pi=(\lambda_1,\ldots,\lambda_{h(\lambda)},1)$ in this case. Let $\overline{\pi}=(\lambda_1-1,\lambda_2,\ldots,\lambda_{h(\lambda)},1)$. Then $\overline{\pi}=\pi\setminus (1,\lambda_1)$. As $\pi$ is 2-regular and $\lambda_1\geq \lambda_2+2$ (since $\lambda$ has at least 3 parts and all parts of $\lambda$ are even), so that $\overline{\pi}_1>\overline{\pi}_2$, we also have that $\overline{\pi}$ is 2-regular. As $\overline{\pi}=\pi\setminus (1,\lambda_1)$ and $(1,\lambda_1)$ is normal of residue residue 1 in $\pi$, we have that $D^{\overline{\pi}}$ is a composition factor of $e_1D^\pi$ from Lemma \ref{l56}. In particular $e_iD^\pi\not=D^\lambda|D^\lambda$.

\vspace{6pt}
\noindent
(iii) Last we consider the case where the parts of $\lambda$ are odd and $\lambda$ has an even number of parts. In this case
\[
\begin{tikzpicture}
\draw (-0.6,-0.7) node {$\lambda=$};

\draw (0.2,-0.2) node {0};
\draw (1.6,-0.2) node {$\ldots$};
\draw (3,-0.2) node {0};
\draw (0.2,-0.6) node {$\vdots$};
\draw (2.3,-0.6) node {$\iddots$};
\draw (0.2,-1.2) node {1};
\draw (0.6,-1.2) node {$\ldots$};
\draw (1,-1.2) node {1};
\draw (1.4,-1.2) node {0};
\draw (0.2,-1.6) node {0};

\draw (2,-1.0) -- (1.2,-1.0) -- (1.2,-1.4) -- (0,-1.4) -- (0,0) -- (3.2,0) -- (3.2,-0.4) -- (2.6,-0.4);
\end{tikzpicture}
\]
As in this case the residue of $(1,\lambda_1)$ is 0, we have that $i=0$. Further both conormal nodes of $\lambda$, $(h(\lambda),\lambda_{h(\lambda)}+1)$ and $(h(\lambda)+1,1)$, have residue 0. So $\pi=(\lambda_1,\ldots,\lambda_{h(\lambda)-1},\lambda_{h(\lambda)}+1)$. First assume that $h(\lambda)\geq 4$ or $\lambda_1>\lambda_2+2$. In this case $\overline{\pi}=(\lambda_1-1,\lambda_2,\ldots,\lambda_{h(\lambda)-1},\lambda_{h(\lambda)}+1)$ is 2-regular (the proof goes as in case (i), with the only difference that here all parts of $\lambda$ are odd instead of even). So $D^{\overline{\pi}}$ is a composition factor of $e_0D^\pi$ from Lemma \ref{l56} and then $e_0D^\pi\not=D^\lambda|D^\lambda$.

If $h(\lambda)=2$ and $\lambda_1=\lambda_2+2$ then $\lambda=(m+1,m-1)$, $\pi=(m+1,m)$ and $\rho=(m,m-1)$. In this case $m$ is even. Here $e_0D^\pi=D^\lambda|D^\lambda$, from Theorem 5.1 of \cite{b1} and as $[e_0D^\pi:D^\lambda]=2$ (Lemma \hyperref[l39b]{\ref*{l39}\ref*{l39b}}). In this case $\phi_0(\lambda)=2$, that is $\phi_0(\rho)=3$ and then
\[f_0e_0D^\lambda\cong f_0D^\rho\sim\underbrace{\overbrace{D^\lambda}^{\soc(e_0D^\pi)=\soc(f_0D^\rho)}|\overbrace{D^\lambda}^{\hd(e_0D^\pi)}}_{e_0D^\pi}|\overbrace{\rlap{$\phantom{D^\lambda}$}\ldots\rlap{$\phantom{D^\lambda}$}}^{\mbox{no }D^\lambda}|\overbrace{D^\lambda}^{\hd(f_0D^\rho)}\]
from the computations in the first part of the proof and as $e_0D^\pi=D^\lambda|D^\lambda$.

As 
\[D^\lambda\otimes M^{(n-1,1)}\sim D^\lambda|(D^\lambda\otimes(S^{(n-1,1)})^*)\sim D^\lambda|(D^\lambda\otimes D^{(n-1,1)})|D^\lambda\]
and $D^\lambda\otimes M^{(n-1,1)}\cong f_0e_0D^\lambda\oplus f_1e_0D^\lambda$, with $D^\lambda$ and $f_1e_0D^\lambda$ corresponding to different blocks, it follows that
\begin{align*}
D^\lambda\otimes D^{(n-1,1)}&\sim (D^\lambda|\overbrace{\rlap{$\phantom{D^\lambda}$}\ldots\rlap{$\phantom{D^\lambda}$}}^{\mbox{no }D^\lambda})\oplus f_1e_0D^\lambda,\\
D^\lambda\otimes(S^{(n-1,1)})^*&\sim(\underbrace{D^\lambda|\overbrace{\rlap{$\phantom{D^\lambda}$}\ldots\rlap{$\phantom{D^\lambda}$}}^{\mbox{no }D^\lambda}|\overbrace{D^\lambda}^{\hd(M)}}_M)\oplus f_1e_0D^\lambda.
\end{align*}
In particular $D^\lambda\subseteq D^\lambda\otimes D^{(n-1,1)}$ and so $D^{(n-1,1)}\subseteq\End_F(D^\lambda)$, as
\[\Hom_{\s_n}(D^{(n-1,1)},\End_F(D^\lambda))\cong\Hom_{\s_n}(D^\lambda,D^\lambda\otimes D^{(n-1,1)}).\]
Also
\[\dim\Hom_{\s_n}(S^{(n-1,1)},\End_F(D^\lambda))=\dim\Hom_{\s_n}(D^\lambda,D^\lambda\otimes (S^{(n-1,1)})^*)=1.\]
The lemma then follows also in this case as $S^{(n-1,1)}=D^{(n)}|D^{(n-1,1)}$.
\end{proof}

\begin{lemma}\label{l11}
Let $p=2$, $\lambda\not=(n)$ be a JS-partition and $\pi=(\lambda_1-1,\lambda_2,\lambda_3,\ldots)$. Then $D^\lambda\downarrow_{\s_{n-1}}\cong D^\pi$ and $\pi$ has exactly 2 normal nodes. The normal nodes of $\pi$ are the removable nodes in the first two rows of $\pi$ and they both have residue different than the normal node of $\lambda$.
\end{lemma}

\begin{proof}
As $\lambda$ is a JS-partition it has only 1 normal node from Lemma \ref{l55}. So $(1,\lambda_1)$ is the only normal node of $\lambda$ (as this node is always normal) and then $D^\lambda\downarrow_{\s_{n-1}}\cong D^\pi$ from Lemmas \ref{l45} and \ref{l39}.

We will now prove that $\pi$ has exactly 2 normal nodes. Notice that $\pi$ has exactly as many parts as $\lambda$, as it is obtain from $\lambda$ by removing a node on the first row and $\lambda$ has more than one row, since $\lambda\not=(n)$. Let $(a_1,\ldots,a_h)$ be the residues of the nodes at the end of each row of $\lambda$ (which are all removal, as $\lambda$ is 2-regular). As $\lambda$ is a JS-partitions, so that its parts are either all even or either all odd from Lemma \ref{l55}, it follows that $(a_1,\ldots,a_{h(\lambda)})=(i,j,i,j,\ldots)$ with $j=1-i$. By definition the residues of the nodes at the end of each row of $\pi$ are given by $(b_1,\ldots,b_{h(\lambda)})=(1-a_1,a_2,\ldots,a_{h(\lambda)})=(j,j,i,j,i,j,\ldots)$. Let $(c_1,\ldots,c_{h(\lambda)+1})$ be the residues of the addable nodes of $\pi$ (with $c_k$ corresponding to the addable node on row $k$). Then $c_k=1-b_k$ for $k\leq h(\lambda)$. For $3\leq k\leq h(\lambda)$ we have that
\[b_k=a_k=1-a_{k-1}=1-b_{k-1}=c_{k-1}\]
and so the removable node on row $k$ of $\pi$ is not normal if $k\geq 3$. As
\[
\begin{tikzpicture}
\draw (0.6,-0.7) node {$\lambda=$};

\draw (3,-0.2) node {j};
\draw (3.3,-0.2) node {i};
\draw (2.4,-0.6) node {j};

\draw (1.9,-1) node {$\iddots$};

\draw (2.25,-0.8)--(2.55,-0.8)--(2.55,-0.4)--(3.45,-0.4)--(3.45,0)--(1.2,0)--(1.2,-1.4)--(1.5,-1.4);

\draw (4.8,-0.6) node {and\hspace{18pt}$ $};
\end{tikzpicture}
\begin{tikzpicture}
\draw (0.6,-0.7) node {$\pi=$};

\draw (3,-0.2) node {j};
\draw (3.3,-0.2) node {i};
\draw (2.4,-0.6) node {j};

\draw (1.9,-1) node {$\iddots$};

\draw (2.25,-0.8)--(2.55,-0.8)--(2.55,-0.4)--(3.15,-0.4)--(3.15,0)--(1.2,0)--(1.2,-1.4)--(1.5,-1.4);

\draw (3.65,-0.8) node{,};
\end{tikzpicture}
\]
the removable nodes on the first 2 rows of $\pi$ are normal of residue $j=1-i$, where $i$ is the residue of the normal node of $\lambda$. As $\pi$ does not have any further normal node, the lemma holds.
\end{proof}

\begin{cor}\label{c1}
Let $p=2$ and $\lambda\not=(n)$ be a JS-partition. Then
\[\dim\End_{\s_{n-2}}(D^\lambda\downarrow_{\s_{n-2}})=2.\]
\end{cor}

\begin{proof}
It follows from Lemmas \ref{l53} and \ref{l11}.
\end{proof}

\begin{lemma}\label{l16}
Let $p=2$ and $n=2m\geq 6$ with $m$ odd and assume that $\lambda\not=(n),(m+1,m-1)$ is a JS-partition. Then $M\subseteq\End_F(D^\lambda)$, where $M$ is a quotient of $D^{(n-2,1)}\uparrow^{\s_n}$ of one of the following forms:
\begin{itemize}
\item
$D^{(n-2,2)}|D^{(n)}|D^{(n-1,1)}$,

\item
$D^{(n)}|D^{(n-2,2)}|D^{(n)}|D^{(n-1,1)}$.
\end{itemize}
\end{lemma}

\begin{proof}
From Lemmas \ref{l50} and \ref{l5}
\[D^{(n-2,1)}\hspace{-2pt}\uparrow^{\s_n}=\hspace{-2pt}\rlap{$\overbrace{\phantom{D^{(n-1,1)}|D^{(n)}|D^{(n-2,2)}|D^{(n)}|D^{(n-1,1)}}}^{e_0D^{(n-1,1)}}$}D^{(n-1,1)}|D^{(n)}|D^{(n-2,2)}|D^{(n)}|\underbrace{D^{(n-1,1)}|D^{(n)}|D^{(n-2,2)}|D^{(n)}|D^{(n-1,1)}}_{e_0D^{(n-1,1)}} \hspace{-2pt}.\]
From Lemmas \ref{l3} and \ref{l1} we have that $S^{(n-1,1)}=D^{(n)}|D^{(n-1,1)}$ is a quotient of $D^{(n-2,1)}\uparrow^{\s_n}$ and so also of $e_0D^{(n-1,2)}$. So
\[D^{(n-2,1)}\hspace{-2pt}\uparrow^{\s_n}=\hspace{-2pt}\rlap{$\overbrace{\phantom{D^{(n-1,1)}|D^{(n)}|D^{(n-2,2)}|}\overbrace{\phantom{D^{(n)}|D^{(n-1,1)}}}^{S^{(n-1,1)}}}^{e_0D^{(n-1,1)}}$}D^{(n-1,1)}|D^{(n)}|D^{(n-2,2)}|D^{(n)}|\underbrace{D^{(n-1,1)}|D^{(n)}|D^{(n-2,2)}|\underbrace{D^{(n)}|D^{(n-1,1)}}_{S^{(n-1,1)}}}_{e_0D^{(n-1,1)}} \hspace{-3.9pt}.\]

From Lemmas \ref{l1} and \ref{l10} we have that $D^{(n-1,1)}$ is not contained in $\End_F(D^\lambda)$. Also from Lemma \ref{l27}, $M^{(n-2,1)}=D^{(n-1)}\oplus D^{(n-2,1)}$, so that $M^{(n-2,1,1)}=D^{(n-2,1)}\uparrow^{\s_n}\oplus M^{(n-1,1)}$. So, from Corollary \ref{c1},
\begin{align*}
&\dim\Hom_{\s_n}(D^{(n-2,1)}\uparrow^{\s_n}/D^{(n-1,1)},\End_F(D^\lambda))\\
&\hspace{12pt}=\dim\Hom_{\s_n}(D^{(n-2,1)}\uparrow^{\s_n},\End_F(D^\lambda))\\
&\hspace{12pt}=\dim\Hom_{\s_n}(M^{(n-2,1,1)},\End_F(D^\lambda))-\dim\Hom_{\s_n}(M^{(n-1,1)},\End_F(D^\lambda))\\
&\hspace{12pt}=2-1\\
&\hspace{12pt}=1.
\end{align*}
As
\[D^{(n-2,1)}\uparrow^{\s_n}/D^{(n-1,1)}\sim ((e_0D^{(n-1,1)})/D^{(n-1,1)})|((e_0D^{(n-1,1)})/D^{(n-1,1)})\]
with
\[(e_0D^{(n-1,1)})/D^{(n-1,1)}=D^{(n)}|D^{(n-2,2)}|\overbrace{D^{(n)}|D^{(n-1,1)}}^{S^{(n-1,1)}},\]
we have from Lemma \ref{l34} that
\[\dim\Hom_{\s_n}((e_0D^{(n-1,2)})/D^{(n-1,1)},\End_F(D^\lambda))>0.\]
From Lemma \ref{l10} we then have that one of
\begin{align*}
(e_0D^{(n-1,1)})/D^{(n-1,1)}&=D^{(n)}|D^{(n-2,2)}|D^{(n)}|D^{(n-1,1)},\\
((e_0D^{(n-1,1)})/D^{(n-1,1)})/D^{(n)}&=D^{(n-2,2)}|D^{(n)}|D^{(n-1,1)}
\end{align*}
is contained in $\End_F(D^\lambda)$. The lemma now follows since $(e_0D^{(n-1,1)})/D^{(n-1,1)}$ is a quotient of $D^{(n-2,1)}\uparrow^{\s_n}$.
\end{proof}

\begin{lemma}\label{l38}
Let $p=2$ and $n=2m\geq 8$ with $m$ even. Assume that $\lambda\not=(n),(m+1,m-1)$ is a JS-partition. Then $M\subseteq\End_F(D^\lambda)$ with $M$ a quotient of $M^{(n-2,2)}$ of one of the following forms:
\begin{itemize}
\item
$M=D^{(n-2,2)}|D^{(n-1,1)}$,

\item
$M=(D^{(n)}\oplus D^{(n-2,2)})|D^{(n-1,1)}$.
\end{itemize}
\end{lemma}

\begin{proof}
From Lemma \ref{l12} we have that
\[\xymat{
&D^{(n)}\ar@{-}[ddr]&D^{(n-1,1)}\ar@{-}[ddl]\ar@{-}[d]\\
M^{(n-2,2)}=&&D^{(n-2,2)}\ar@{-}[d]&.\\
&D^{(n)}&D^{(n-1,1)}
}\]
Also from Lemma \ref{l59} we have that $M^{(n-2,2)}/D^{(n-1,1)}\cong D^{(n)}\oplus N$, 
where $N=(D^{(n)}\oplus D^{(n-2,2)})|D^{(n-1,1)}\sim D^{(n-2,2)}|S^{(n-1,1)}$. From Theorem 3.6 of \cite{ks} we have that
\[\dim\End_{\s_{n-2,2}}(D^\lambda\downarrow_{\s_{n-2,2}})>\dim\End_{\s_{n-1}}(D^\lambda\downarrow_{\s_{n-1}}).\]
In particular $\dim\End_{\s_{n-2,2}}(D^\lambda\downarrow_{\s_{n-2,2}})\geq 2$. As $\lambda\not=(n),(m+1,m-1)$ is a JS-partition and $S^{(n-1,1)}=D^{(n)}|D^{(n-1,1)}$ (Lemma \ref{l1}), we have from Lemma \ref{l10} that $D^{(n-1,1)}\not\subseteq\End_F(D^\lambda)$. So
\begin{align*}
&\dim\Hom_{\s_n}(N,\End_F(D^\lambda))\\
&\hspace{12pt}=\dim\Hom_{\s_n}(M^{(n-2,2)}/D^{(n-1,1)},\End_F(D^\lambda))-\dim\End_{\s_n}(D^\lambda)\\
&\hspace{12pt}=\dim\Hom_{\s_n}(M^{(n-2,2)},\End_F(D^\lambda))-1\\
&\hspace{12pt}=\dim\End_{\s_{n-2,2}}(D^\lambda\downarrow_{\s_{n-2,2}})-1\\
&\hspace{12pt}\geq 1.
\end{align*}
The only non-zero quotients of $N$ (which are also quotients of $M^{(n-2,2)}$) are
\begin{align*}
N&=(D^{(n)}\oplus D^{(n-2,2)})|D^{(n-1,1)},\\
N/D^{(n)}&=D^{(n-2,2)}|D^{(n-1,1)},\\
N/D^{(n-2,2)}&=D^{(n)}|D^{(n-1,1)}\cong S^{(n-1,1)},\\
N/(D^{(n)}\oplus D^{(n-2,2)})&=D^{(n-1,1)}.
\end{align*}
From Lemma \ref{l10} it follows that $N$ or $N/D^{(n)}$ is a submodule of $\End_F(D^\lambda)$.
\end{proof}

We also need some results on the structure of $\End_F(D^\lambda\downarrow_{\s_{n-1}})$ for the case $n\equiv 0\Md 4$. This will be used to show that no non-trivial irreducible tensor products of the form $D^\lambda\otimes D^{(m+1,m-1)}$ exist when $m$ is even.

\begin{lemma}\label{l32}
Let $p=2$ and $n=2m\geq 4$ with $m$ even. Assume that $\lambda\not=(n)$ is a JS-partition. Let $\pi=(\lambda_1-1,\lambda_2,\lambda_3,\ldots)$. Then $D^{(n-2,1)}$ is contained exactly once in $\End_F(D^\pi)$ as a submodule.
\end{lemma}

\begin{proof}
From Lemmas \ref{l53} and \ref{l11} we have that $\dim\End_{\s_{n-2}}(D^\pi\downarrow_{\s_{n-2}})=2$, as $\pi$ has 2 normal nodes. From Lemma \ref{l27} it then follows that
\begin{align*}
&\dim\Hom_{\s_{n-1}}(D^{(n-2,1)},\End_F(D^\pi))\\
&\hspace{12pt}=\dim\Hom_{\s_{n-1}}(M^{(n-2,1)},\End_F(D^\pi))-\dim\Hom_{\s_{n-1}}(D^{(n-1)},\End_F(D^\pi))\\
&\hspace{12pt}=\dim\End_{\s_{n-2}}(D^\pi\downarrow_{\s_{n-2}})-\dim\End_{\s_{n-1}}(D^\pi)\\
&\hspace{12pt}=1.
\end{align*}
\end{proof}

\begin{lemma}\label{l30}
Let $p=2$ and $n=2m\geq 8$ with $m$ even. Assume that $\lambda\not=(n)$ is a JS-partition. Let $\pi=(\lambda_1-1,\lambda_2,\lambda_3,\ldots)$. Then $M\subseteq \End_F(D^\pi)$, where $M$ is a quotient of $M^{(n-3,1,1)}$ of one of the following forms:
\begin{itemize}
\item
$M\sim (D^{(n-1)}\oplus D^{(n-3,2)})|(D^{(n-1)}\oplus D^{(n-3,2)})$,

\item
$M=D^{(n-3,2)}|(D^{(n-1)}\oplus D^{(n-3,2)})$,

\item
$M=D^{(n-1)}|(D^{(n-1)}\oplus D^{(n-3,2)})$.
\end{itemize}
\end{lemma}

\begin{proof}
From Lemma \ref{l55} there exist $0\leq i\leq 1$ and $j=1-i$ with $\epsilon_i(\lambda)=1$ and $\epsilon_j(\lambda)=0$. From Lemma \ref{l11} we have that $\epsilon_i(\pi)=0$ and $\epsilon_j(\pi)=2$. So, from Lemmas \ref{l45} and \ref{l39}, $D^\pi\downarrow_{\s_{n-3}}=e_je_jD^\pi\oplus e_ie_jD^\pi$. From the block decomposition of $D^\pi\downarrow_{\s_{n-3}}$ and $D^\pi\downarrow_{\s_{n-3,2}}$ it then also follows that $D^\pi\downarrow_{\s_{n-3,2}}=\overline{e}_j^2D^\pi\oplus N$ with $N\downarrow_{\s_{n-3}}=e_ie_jD^\pi$.

Using Lemmas \ref{l39} and \ref{l31} and considering the block decompositions of the modules we have
\begin{align*}
\dim\End_{\s_{n-3}}(D^\pi\downarrow_{\s_{n-3}})&=\dim\End_{\s_{n-3}}(e_je_jD^\pi)+\dim\End_{\s_{n-3}}(e_ie_jD^\pi)\\
&=4+\dim\End_{\s_{n-3}}(N\downarrow_{\s_{n-3}})\\
&\geq \dim\End_{\s_{n-3,2}}(\overline{e}_j^2D^\pi)+2+\dim\End_{\s_{n-3,2}}(N)\\
&=\dim\End_{\s_{n-3,2}}(\overline{e}_j^2D^\pi\oplus N)+2\\
&=\dim\End_{\s_{n-3,2}}(D^\pi\downarrow_{\s_{n-3,2}})+2.
\end{align*}
Let
\[\xymat{
\hspace{-24pt}&D^{(n-1)}\ar@{-}[dd]\ar@{-}[ddr]&D^{(n-3,2)}\ar@{-}[dd]\ar@{-}[ddl]\\
L=\hspace{-24pt}\\
\hspace{-24pt}&D^{(n-1)}&D^{(n-3,2)}
}\]
with $L\subseteq M^{(n-3,1,1)}$ (see Lemma \ref{l25}). Then, from Lemmas \ref{l27}, \ref{l25} and \ref{l32},
\begin{align*}
&\dim\Hom_{\s_{n-1}}(L,\End_F(D^\pi))\\
&\hspace{12pt}=\dim\End_{\s_{n-3}}(D^\pi\downarrow_{\s_{n-3}})-2\dim\Hom_{\s_{n-1}}(D^{(n-2,1)},\End_F(D^\pi))\\
&\hspace{12pt}\geq\dim\End_{\s_{n-3,2}}(D^\pi\downarrow_{\s_{n-3,2}})+2-2\\
&\hspace{12pt}=\dim\Hom_{\s_{n-1}}(M^{(n-3,2)},\End_F(D^\pi))\\
&\hspace{12pt}=\dim\Hom_{\s_{n-1}}(D^{(n-1)}\oplus D^{(n-2,1)}\oplus D^{(n-3,2)},\End_F(D^\pi))\\
&\hspace{12pt}=\dim\Hom_{\s_{n-1}}(D^{(n-1)}\oplus D^{(n-3,2)},\End_F(D^\pi))\\
&\hspace{36pt}+\dim\Hom_{\s_{n-1}}(D^{(n-2,1)},\End_F(D^\pi))\\
&\hspace{12pt}=\dim\Hom_{\s_{n-1}}(\hd(L),\End_F(D^\pi))+1.
\end{align*}
As the only quotients of $L$ but not of its head are
\begin{align*}
L&\sim(D^{(n-1)}\oplus D^{(n-3,2)})|(D^{(n-1)}\oplus D^{(n-3,2)}),\\
L/D^{(n-1)}&=D^{(n-3,2)}|(D^{(n-1)}\oplus D^{(n-3,2)}),\\
L/D^{(n-3,2)}&=D^{(n-1)}|(D^{(n-1)}\oplus D^{(n-3,2)}),
\end{align*}
the lemma follows.
\end{proof}

\begin{lemma}\label{l28}
Let $p=2$ and $n=2m\geq 8$ with $m$ even. Then there exists $M\subseteq \End_F(D^{(m,m-1)})$ with $M=D^{(n-1)}|(D^{(n-1)}\oplus D^{(n-3,2)})$ a quotient of $M^{(n-3,1,1)}$.
\end{lemma}

\begin{proof}
From Lemmas \ref{l55} and \ref{l30} it is enough to show that $D^{(n-3,2)}\not\subseteq\End_F(D^{(m,m-1)})$. From Lemma \ref{l27} it is then enough to prove that
\begin{align*}
&\dim\Hom_{\s_{n-1}}(M^{(n-2,1)},\End_F(D^{(m,m-1)}))\\
&\hspace{12pt}=\dim\Hom_{\s_{n-1}}(M^{(n-3,2)},\End_F(D^{(m,m-1)})).
\end{align*}
From Lemma \ref{l55} and \ref{l32} the nodes at the end of the 2 rows of $(m,m-1)$ are normal. Both of these nodes have residue 1, as $m$ is even. In particular, from Lemmas \ref{l45} and \ref{l39},
\[D^{(m-1,m-2)}\oplus D^{(m-1,m-2)}=e_1^2D^{(m,m-1)}\subseteq D^{(m,m-1)}\downarrow_{\s_{n-3}}.\]
Comparing degrees it follows from Theorem 5.1 of \cite{b1} that $D^{(m,m-1)}\downarrow_{\s_{n-3}}\cong e_1^2D^{(m,m-1)}$. Then, comparing the block decompositions of $D^{(m,m-1)}\downarrow_{\s_{n-3}}$ and of $D^{(m,m-1)}\downarrow_{\s_{n-3,2}}$, using Lemma \ref{l31} we have that
\[D^{(m,m-1)}\downarrow_{\s_{n-3,2}}\cong \overline{e}_1^2D^{(m,m-1)}=D^{(m-1,m-2)}|D^{(m-1,m-2)}.\]
As $\pi$ has 2 normal nodes, from Lemma \ref{l53} we have that
\begin{align*}
&\dim\Hom_{\s_{n-1}}(M^{(n-2,1)},\End_F(D^{(m,m-1)}))\\
&\hspace{12pt}=\dim\End_{\s_{n-2}}(D^{(m,m-1)}\downarrow_{\s_{n-2}})\\
&\hspace{12pt}=2\\
&\hspace{12pt}=\dim\End_{\s_{n-3,2}}(D^{(m,m-1)}\downarrow_{\s_{n-3,2}})\\
&\hspace{12pt}=\dim\Hom_{\s_{n-1}}(M^{(n-3,2)},\End_F(D^{(m,m-1)})),
\end{align*}
from which the lemma follows.
\end{proof}

\begin{lemma}\label{l29}
Let $p=2$ and $n=2m\geq 8$ with $m$ even. Assume that $\lambda\not=(n)$ is a JS-partition with $D^\lambda\downarrow_{A_n}$ irreducible. Let $\pi=(\lambda_1-1,\lambda_2,\lambda_3,\ldots)$. Then $M\subseteq \End_F(D^\pi)$, where $M\sim (D^{(n-1)}\oplus D^{(n-3,2)})|(D^{(n-1)}\oplus D^{(n-3,2)})$ or $M=D^{(n-3,2)}|(D^{(n-1)}\oplus D^{(n-3,2)})$ is a quotient of $M^{(n-3,1,1)}$.
\end{lemma}

\begin{proof}
From Lemma \ref{l30} it is enough to prove that if $A=D^{(n-1)}|D^{(n-1)}$ then $A\not\subseteq\End_F(D^\pi)$. From Lemma \ref{l26} it is then enough to prove that $1\uparrow_{A_{n-1}}^{\s_{n-1}}\not\subseteq\End_F(D^\pi)$. If $1\uparrow_{A_{n-1}}^{\s_{n-1}}\subseteq\End_F(D^\pi)$ then
\begin{align*}
\dim\End_{A_{n-1}}(D^\pi\downarrow_{A_{n-1}})&=\dim\Hom_{\s_{n-1}}(1\uparrow_{A_{n-1}}^{\s_{n-1}},\End_F(D^\pi))\\
&\geq\dim\End_{\s_{n-1}}(1\uparrow_{A_{n-1}}^{\s_{n-1}})\\
&=2
\end{align*}
and so $D^\pi\downarrow_{A_{n-1}}$ splits. We will now show that this is impossible.

Assume instead that $D^\pi\downarrow_{A_{n-1}}$ splits. From Theorem 1.1 of \cite{b1} we would then have that $\pi_{2h+1}-\pi_{2h+2}\leq 2$ and $\pi_{2h+1}+\pi_{2h+2}\not\equiv 2\Md 4$ for $h\geq 0$. By definition of $\pi$ we also have that $\lambda_{2h+1}-\lambda_{2h+2}\leq 2$ and $\lambda_{2h+1}+\lambda_{2h+2}\not\equiv 2\Md 4$ for $h\geq 1$. As $\lambda$ is a JS-partition all its parts are even or all its parts are odd by Lemma \ref{l55}. By assumption $\pi_2=\lambda_2>0$ and $\pi_1=\lambda_1-1$. So $\pi_1-\pi_2=1$, that is $\lambda_1-\lambda_2=2$. As $D^\lambda\downarrow_{A_n}$ does not splits we have from Theorem 1.1 of \cite{b1} that $\lambda_1+\lambda_2\equiv 2\Md 4$. So $\lambda_1$ is even and then this holds for all parts of $\lambda$. If $\lambda_3=0$ then $\lambda=(m+1,m-1)$, contradicting $m$ and $\lambda_1$ both being even. If $\lambda_3>0$ then $\pi_4=\pi_3-2$ (as $\pi_3=\lambda_3>\lambda_4=\pi_4$ are both even and $\pi_3-\pi_4\leq 2$) and then $\pi_3+\pi_4=\lambda_3+\lambda_4\equiv 2\Md 4$, contradicting the assumption that $D^\pi\downarrow_{A_{n-1}}$ splits.
\end{proof}

\section{Proof of Theorem \ref{t3}}\label{s1}

We are now ready to prove Theorems \ref{t3} and \ref{t4}.

\begin{proof}[Proof of Theorem \ref{t3}]
The theorem holds for $n=2$ (since $\s_2$ is abelian). So we can assume that $n\geq 6$. We will first prove that if $\lambda,\mu\not=(n),(m+1,m-1)$ then $D^\lambda\otimes D^\mu$ is not irreducible. To do this let $\lambda,\mu\not=(n),(m+1,m-1)$. Then, from Lemmas \ref{l17}, \ref{l15}, \ref{l19} and \ref{l16}, there exist $M\subseteq\End_F(D^\lambda)$ and $N\subseteq \End_F(D^\mu)$, with $M$ and $N$ quotients of $D^{(n-2,1)}\uparrow^{\s_n}$ of one of the following forms:
\begin{itemize}
\item
$D^{(n-2,2)}|D^{(n)}|D^{(n-1,1)}=A$,

\item
$D^{(n)}|D^{(n-2,2)}|D^{(n)}|D^{(n-1,1)}=B$,

\item
$D^{(n-1,1)}|D^{(n)}|D^{(n-2,2)}|D^{(n)}|D^{(n-1,1)}=C$.
\end{itemize}
From Lemma \ref{l5} we have that
\[D^{(n-2,1)}\hspace{-1pt}\uparrow^{\s_n}=\hspace{-1pt}D^{(n-1,1)}|D^{(n)}|D^{(n-2,2)}|D^{(n)}|D^{(n-1,1)}|D^{(n)}|D^{(n-2,2)}|D^{(n)}|D^{(n-1,1)}.\]
From Lemma \ref{l50} it then follows that $C\cong e_0D^{(n-1,2)}$ (it is the only quotient of $D^{(n-2,1)}\uparrow^{\s_n}$ with socle $D^{(n-1,1)}$ and 2 composition factors isomorphic to $D^{(n-1,1)}$). From Lemma \hyperref[l39a]{\ref*{l39}\ref*{l39a}} we then have that $C$ is self-dual. Also $A$ and $B$ are quotients of $C$ (from their composition sequences as $D^{(n-2,1)}\uparrow^{\s_n}$ is uniserial). Also from the self-duality of $C$ it follows that $D^{(n)}|D^{(n-2,2)}|D^{(n)}\subseteq B$ is self-dual and this module contains $D^{(n)}|D^{(n-2,2)}=A^*/D^{(n-1,1)}$.

If $D^{(n)}\subseteq M$ let $\overline{M}:=M$, else let $\overline{M}:=M\oplus D^{(n)}$. Define $\overline{N}$ similarly. Since $\End_F(D^\lambda)$, $\End_F(D^\mu)$ and $D^{(n-2,1)}\uparrow^{\s_n}$ are self-dual, considering the possible combinations of $M$ and $N$ we obtain that
\begin{align*}
\dim\End_{\s_n}(D^\lambda\otimes D^\mu)&=\dim\Hom_{\s_n}(\End_F(D^\lambda),\End_F(D^\mu))\\
&\geq\dim\Hom_{\s_n}(\overline{M}^*,\overline{N})\\
&\geq 2.
\end{align*}
In particular $D^\lambda\otimes D^\mu$ is not irreducible.

We can now assume that $\lambda=(m+1,m-1)$. From Theorem 3.1(b) of \cite{gk} we know that there are at most $(m+1)/2$ partitions $\lambda$ for which $D^\lambda\otimes D^\mu$ is irreducible. One of these is $(n)$. As from Corollary 3.21 of \cite{gj} we have that $D^{(m+1,m-1)}\otimes D^{(n-2j+1,2j+1)}\cong D^{(m-j,m-j-1,j+1,j)}$ for $0\leq j<(m-1)/2$ the theorem follows.
\end{proof}

\begin{proof}[Proof of Theorem \ref{t4}]
The theorem is easily checked for $n=4$, as there is only one irreducible representation of $F\s_4$ of degree larger than 1. So we can assume that $n\geq 8$. Let $\lambda,\mu\not=(n)$ be 2-regular partitions. From \cite{ck} we can assume that $\lambda$ is a JS-partition.

Assume first that $\lambda,\mu\not=(m+1,m-1)$. We will now prove that in this case $D^\lambda\otimes D^\mu$ is not irreducible. From Lemma \ref{l38} there exists $M\subseteq \End_F(D^\lambda)$ with $M$ a quotient of $M^{(n-2,2)}$ such that $M=D^{(n-2,2)}|D^{(n-1,1)}$ or $M=(D^{(n)}\oplus D^{(n-2,2)})|D^{(n-1,1)}$. By Lemmas \ref{l36}, \ref{l41}, \ref{l43}, \ref{l46} and \ref{l38} we also have that $N\subseteq \End_F(D^\mu)$ with $N=D^{(n-2,2)}$ or with $N=D^{(n-1,1)}|(D^{(n)}\oplus D^{(n-2,2)})\subseteq M^{(n-2,2)}$.

First assume that $N=D^{(n-2,2)}$. Then, from self duality of $\End_F(D^\lambda)$ and $\End_F(D^\mu)$,
\begin{align*}
\dim\End_{\s_n}(D^\lambda\otimes D^\mu)&=\dim\Hom_{\s_n}(\End_F(D^\lambda),\End_F(D^\mu))\\
&\geq \dim\End_{\s_n}(D^{(n)}\oplus D^{(n-2,2)})\\
&=2.
\end{align*}
In particular $D^\lambda\otimes D^\mu$ is not irreducible.

Assume now that $N=D^{(n-1,1)}|(D^{(n)}\oplus D^{(n-2,2)})$. From the structure of $M^{(n-2,2)}$ (see Lemma \ref{l12}) and from its self-duality we have that $M^*\subseteq N$. If $M=D^{(n-2,2)}|D^{(n-1,1)}$ then
\begin{align*}
&\dim\End_{\s_n}(D^\lambda\!\otimes\! D^\mu)\\
&\hspace{6pt}\geq \dim\Hom_{\s_n}(D^{(n)}\!\oplus\! M^*,D^{(n)}\!\oplus\! N)\\
&\hspace{6pt}=\dim\Hom_{\s_n}(D^{(n)}\!\oplus\! (D^{(n-1,1)}|D^{(n-2,2)}),D^{(n)}\!\oplus\! (D^{(n-1,1)}|(D^{(n)}\!\oplus\! D^{(n-2,2)})))\\
&\hspace{6pt}=2.
\end{align*}
If $M=(D^{(n)}\oplus D^{(n-2,2)})|D^{(n-1,1)}$ then
\begin{align*}
&\dim\End_{\s_n}(D^\lambda\!\otimes\! D^\mu)\\
&\hspace{6pt}\geq \dim\Hom_{\s_n}(M^*,D^{(n)}\!\oplus\! N)\\
&\hspace{6pt}=\dim\Hom_{\s_n}(D^{(n-1,1)}|(D^{(n)}\!\oplus\! D^{(n-2,2)}),D^{(n)}\!\oplus\! (D^{(n-1,1)}|(D^{(n)}\!\oplus\! D^{(n-2,2)})))\\
&\hspace{6pt}=2.
\end{align*}
In either case $D^\lambda\otimes D^\mu$ is not irreducible.

Since $D^{(n)}$ is 1-dimensional, we can now assume that $\lambda=(m+1,m-1)$ and $\mu\not=(n)$ is a 2-regular partition. Assume that $D^\lambda\otimes D^\mu\cong D^\nu$ is irreducible. From Lemma \ref{l10} we have that $D^{(n-1,1)}\subseteq \End_F(D^\lambda)$. If $\mu$ is not a JS-partition, then $D^{(n-1,1)}\subseteq \End_F(D^\mu)$ from Lemmas \ref{l18} and \ref{l20}, so that
\[\dim\End_{\s_n}(D^\lambda\otimes D^\mu)\geq \dim\End_{\s_n}(D^{(n)}\oplus D^{(n-1,1)})=2,\]
contradicting $D^\lambda\otimes D^\mu$ being irreducible. So we can further assume that $\mu$ is a JS-partition.

For a 2-regular partition $\alpha$ let $\phi^\alpha$ be the Brauer character of $D^\alpha$. For any partition $\beta$ let $\xi^\beta$ be the Brauer character of $M^\beta$. For $\gamma$ consisting only of odd parts and $\chi$ a Brauer character of $\s_n$ let $\chi_\gamma$ be the value that $\chi$ takes on the conjugacy class labeled by $\gamma$.

Let $\gamma$ consists only of odd parts. From Theorem 5.1 of \cite{b1} and Theorem VII of \cite{s5} we have that $\phi^{(m+1,m-1)}_\gamma=\pm 2^{\lfloor (h(\gamma)-1)/2\rfloor}$. Assume now that $\phi^{(m+1,m-1)}_\gamma$ is odd. Then $h(\gamma)\leq 2$ and so, since $n=2m$ with $m$ even, it follows that $\gamma=(n-s,s)$ with $1\leq s<m$ odd. Also notice that $\xi^{(n-k,k)}_{(n-s,s)}=\delta_{s,k}$ for $1\leq s,k<m$. Assume that $D^\lambda\otimes D^\mu\cong D^\nu$ and let $1\leq s_1<\ldots<s_r<m$ odd with $\phi^\mu_{(n-s,s)}$ odd if and only if $s=s_l$ for some $l$. Define $\phi:=\phi^\mu+\xi^{(n-s_1,s_1)}+\ldots+\xi^{(n-s_r,s_r)}$. Then, for each $\gamma$ consisting only of odd parts, $\phi^{(m+1,m-1)}_\gamma\phi_\gamma$ is even. In particular $\phi^{(m+1,m-1)}\phi=2\chi$ for some Brauer character $\chi$, as irreducible Brauer characters are linearly independent modulo 2. Since $\phi^{(m+1,m-1)}\phi^\mu=\phi^\nu$ and the definition of $\phi$, we then have that
\begin{align*}
\phi^{(m+1,m-1)}\xi^{(n-s_1,s_1)}\hspace{-1.5pt}+\hspace{-1.5pt}\ldots\hspace{-1.5pt}+\hspace{-1.5pt}\phi^{(m+1,m-1)}\xi^{(n-s_r,s_r)}&=\phi^{(m+1,m-1)}\phi\hspace{-1.5pt}-\hspace{-1.5pt}\phi^{(m+1,m-1)}\phi^\mu\\
&=2\chi\hspace{-1.5pt}-\hspace{-1.5pt}\phi^\nu.
\end{align*}
In particular
\[[S^{(m+1,m-1)}\otimes M^{(n-s_l,s_l)}:D^\nu]\geq[D^{(m+1,m-1)}\otimes M^{(n-s_l,s_l)}:D^\nu]\geq 1\]
for some $1\leq l\leq h$. So $\nu$ has at most 4 parts, since $D^\nu$ is a component of $S^{(m+1,m-1)}\otimes M^{(n-s_l,s_l)}$.

Notice that $D^{(m+1,m-1)}$ splits when reduced to $A_n$ from Theorem 1.1 of \cite{b1}. As $D^{(m+1,m-1)}\otimes D^\mu\cong D^\nu$ we then have that $D^\nu$ also splits when reduced to $A_n$, while $D^\mu$ does not. Again from Theorem 1.1 of \cite{b1} and as $\nu$ at most 4 parts, $\nu=(a,a-b,c,c-d)$ with $1\leq b,d\leq 2$ and $2a-b,2c-d\not\equiv 2\Md 4$ or with $c=d=0$, $1\leq b\leq 2$ and $2a-b\not\equiv 2\Md 4$.

Assume first that $c,d=0$, $1\leq b\leq 2$ and $2a-b\not\equiv 2\Md 4$. Then $b=2$ as $n$ is even and so $\nu=(m+1,m-1)=\lambda$. As $\mu\not=(n)$ so that $D^\mu$ has degree at least 2 (as we are in characteristic 2), this gives a contradiction.

So $1\leq b,d\leq 2$ and $2a-b,2c-d\not\equiv 2\Md 4$. As $|\nu|=2a+2c-b-d$ is even it follows $b=d$. If $b=2$ then from $2a-2,2c-2\not\equiv 2\Md 4$ it follows that $a$ and $c$ are odd. This would mean that $\nu$ is a JS-partition from Lemma \ref{l55} and so that $D^{(m+1,m-1)}\downarrow_{\s_{n-1}}\otimes D^\mu\downarrow_{\s_{n-1}}$ is irreducible, which is impossible from \cite{ck} ($D^{(m+1,m-1)}\downarrow_{\s_{n-1}}$ and $D^\mu\downarrow_{\s_{n-1}}$ are both irreducible of degree at least 2). So $b=d=1$ and then $\nu=(m-s,m-s-1,s+1,s)$ for some $0\leq s<(m-2)/2$.

As $m$ is even, so that
\[\begin{tikzpicture}
\draw (-1.1,0.6) node {$\nu=$};

\draw (0,0) node {$j$};
\draw (0.4,0) node {$i$};
\draw (0.4,0.4) node {$j$};
\draw (0.8,0.4) node {$i$};
\draw (1.6,0.8) node {$i$};
\draw (2,0.8) node {$j$};
\draw (2,1.2) node {$i$};
\draw (2.4,1.2) node {$j$};

\draw (-0.6,-0.2)--(0.2,-0.2)--(0.2,0.2)--(0.6,0.2)--(0.6,0.6)--(1.8,0.6)--(1.8,1)--(2.2,1)--(2.2,1.4)--(-0.6,1.4)--(-0.6,-0.2);

\draw (4.2,0.6) node {or};

\draw (6.4,0.6) node {$\nu=$};

\draw (7.1,0.2) node {$j$};
\draw (7.5,0.2) node {$i$};
\draw (8.3,0.6) node {$i$};
\draw (8.7,0.6) node {$j$};
\draw (8.7,1) node {$i$};
\draw (9.1,1) node {$j$};

\draw (6.9,0)--(7.3,0)--(7.3,0.4)--(8.5,0.4)--(8.5,0.8)--(8.9,0.8)--(8.9,1.2)--(6.9,1.2)--(6.9,0);

\draw (9.5,0.4) node {,};
\end{tikzpicture}\]
we have that $\epsilon_i(\nu)=2$ and $\epsilon_j(\nu)=0$ and then $\dim\End_{\s_{n-1}}(D^\nu\downarrow_{\s_{n-1}})=2$ from Lemma \ref{l53}. Let $\pi=(\mu_1-1,\mu_2,\mu_3,\ldots)$. If $D^{(m+1,m-1)}\otimes D^\mu=D^\nu$ then
\[D^{(m,m-1)}\otimes D^\pi\cong D^{(m+1,m-1)}\downarrow_{\s_{n-1}}\otimes D^\mu\downarrow_{\s_{n-1}}\cong D^\nu\downarrow_{\s_{n-1}}\]
from Lemma \ref{l11} and so
\begin{align*}
\dim\Hom_{\s_{n-1}}(\End_F(D^\pi),\End_F(D^{(m,m-1)}))&=\dim\End_{\s_{n-1}}(D^{(m,m-1)}\otimes D^\pi)\\
&=2.
\end{align*}
We will show that this is impossible when $\mu$ is a JS-partition with $D^\mu\downarrow_{A_n}$ irreducible. 

From Lemma \ref{l32} we have that $D^{(n-2,1)}$ is contained in $\End_F(D^{(m,m-1)})$ and in $\End_F(D^\pi)$. From Lemmas \ref{l28} and \ref{l29} we have that there are modules $M\subseteq\End_F(D^{(m,m-1)})$ and $N\subseteq\End_F(D^\pi)$ which are quotients of $M^{(n-3,1,1)}$ such that $M=D^{(n-1)}|(D^{(n-1)}\oplus D^{(n-3,2)})$ and $N\sim(D^{(n-1)}\oplus D^{(n-3,2)})|(D^{(n-1)}\oplus D^{(n-3,2)})$ or $N=D^{(n-3,2)}|(D^{(n-1)}\oplus D^{(n-3,2)})$.

Assume first that $N\sim(D^{(n-1)}\oplus D^{(n-3,2)})|(D^{(n-1)}\oplus D^{(n-3,2)})$. Then, from Lemma \ref{l25} and the self-duality of $M^{(n-3,1,1)}$, we have that $M$ is a quotient of $N\cong N^*$. So, from the structure of $M$ and $N$,
\begin{align*}
&\dim\Hom_{\s_{n-1}}(\End_F(D^\pi),\End_F(D^{(m,m-1)}))\\
&\hspace{12pt}\geq \dim\Hom_{\s_{n-1}}(N\oplus D^{(n-2,1)},M\oplus D^{(n-2,1)})\\
&\hspace{12pt}=3.
\end{align*}

If $N=D^{(n-3,2)}|(D^{(n-1)}\oplus D^{(n-3,2)})$ then, again by Lemma \ref{l25} and the self-duality of $M^{(n-3,1,1)}$, we have that there exists $L=D^{(n-1)}|D^{(n-3,2)}$ with $L\subseteq M$ and $L$ a quotient of $N^*$. Then, from the structure of $L$,
\begin{align*}
&\dim\Hom_{\s_{n-1}}(\End_F(D^\pi),\End_F(D^{(m,m-1)}))\\
&\hspace{12pt}\geq \dim\Hom_{\s_{n-1}}(D^{(n-1)}\oplus N^*\oplus D^{(n-2,1)},M\oplus D^{(n-2,1)})\\
&\hspace{12pt}\geq \dim\Hom_{\s_{n-1}}(D^{(n-1)}\oplus L\oplus D^{(n-2,1)},L\oplus D^{(n-2,1)})\\
&\hspace{12pt}=3.
\end{align*}

In either case $\dim\Hom_{\s_{n-1}}(\End_F(D^\pi),\End_F(D^{(m,m-1)}))\geq 3$, leading to a contradiction.
\end{proof}

\section*{Acknowledgements}

The author thanks Christine Bessenrodt for suggesting the here studied question and for discussion and comments on the paper. The author also thanks the referee for comments and for pointing out reference \cite{mo}.


\begin{thebibliography}{99}

\bibitem{a} M. Aschbacher. On the maximal subgroups of the finite classical groups. {\em Invent. Math.} 76 (1984), 469-514.

\bibitem{as} M.Aschbacher, L. Scott. Maximal subgroups of finite groups. {\em J. Algebra} 92 (1985), 44-80.

\bibitem{b1} D. Benson. Spin modules for symmetric groups. {\em J. London Math. Soc.} 38 (1988), 250-262.


\bibitem{b2} C. Bessenrodt.  On mixed products of complex characters of the double covers of the symmetric groups. {\em Pacific J. Math.} 199 (2001) 257-268.

\bibitem{bk2} C. Bessenrodt, A. Kleshchev. Irreducible Tensor Products over Alternating Groups. {\em J. Algebra} 228 (2000), 536-550.

\bibitem{bk3} C. Bessenrodt, A. Kleshchev. On Kronecker products of irreducible complex representations of the symmetric and alternating groups. {\em Pacific J. Math.} 190 (1999), 201-223.

\bibitem{bk4} C. Bessenrodt, A. Kleshchev. On Kronecker products of spin characters of the double covers of the symmetric groups. {\em Pacific J. Math.} 198 (2001), 295-305.

\bibitem{ck} C. Bessenrodt, A. Kleshchev. On tensor products of modular representations of symmetric groups. {\em Bull. London Math. Soc.} 32 (2000), 292-296.


\bibitem{gj} J. Graham, G. James. On a conjecture of Gow and Kleshchev concerning tensor products. {\em J. Algebra} 227 (2000), 767-782.

\bibitem{gk} R. Gow, A. Kleshchev. Connections between the representations of the symmetric group and the symplectic group in characteristic 2. {\em J. Algebra} 221 (1999), 60-89.

\bibitem{j1} G. D. James. {\em The  Representation  Theory  of  the  Symmetric  Groups.} Springer-Verlag,  Berlin  Heidelberg,  1978.  Lecture  Notes  in  Mathematics, 682.

\bibitem{jk} G. James, A. Kerber. {\em The Representation Theory of the Symmetric Group.} Addison-Wesley Publishing Company, 1981.

\bibitem{k3} A. Kleshchev.  Branching Rules for Modular Representations of Symmetric Groups, IV. {\em J. Algebra} 201 (1998), 547-572.

\bibitem{k1} A. Kleshchev. {\em Linear and projective representations of symmetric groups.} Cambridge University Press, Cambridge, 2005. Cambridge Tracts in Mathematics, 163.

\bibitem{k2} A. Kleshchev. On restrictions of Irreducible Modular Representations of Semisimple Algebraic Groups and Symmetric Groups to Some Natural Subgroups, I. {\em  Proc. London Math. Soc.} 69 (1994), 515-540.

\bibitem{ks} A. Kleshchev, J. Sheth. Representations of the symmetric group are reducible over singly transitive subgroups. {\em Math. Z.} 235 (2000), 99-109.

\bibitem{kst} A. S. Kleshchev, P. Sin, P. H. Tiep. Representations of the alternating group which are irreducible over subgroups. II, {\em Amer. J. Math.} 138 (2016), 1383-1423.

\bibitem{kt} A. Kleshchev, P. H. Tiep. On restrictions of modular spin representations of symmetric and alternating groups. {\em Trans. Amer. Math. Soc.} 356 (2004), 1971-1999.

\bibitem{kt2} A. S. Kleshchev, P. H. Tiep. Representations of the general linear groups which are irreducible over subgroups. {\em Amer. J. Math.} 132 (2010), 425-473.

\bibitem{mt} K. Magaard, P. H. Tiep. Irreducible tensor products of representations of finite quasi-simple groups of Lie type. {\em Modular representation theory of finite groups}, de Gruyter, Berlin (2001) 239-262.

\bibitem{mo} J. M{\"u}ller, J. Orlob, On the structure of the tensor square of the natural module of the symmetric group, {\em Algebra Colloq.} 18 (2011), 589-610.

\bibitem{s5} J. Schur. {\"U}ber die Darstellung der symmetrischen und alternierenden Gruppe durch gebrochene lineare Substitutionen, {\em J. Reine Angew. Math.} 139 (1911), 155-250.

\bibitem{z1} I. Zisser. Irreducible products of characters in $A_n$. {\em  Israel J. Math.} 84 (1993), 147-151.

\end{thebibliography}
\end{document}